\documentclass[12pt,reqno]{amsart}
\usepackage{amssymb,epsf}
\usepackage{amstext,amsmath,amsfonts,amscd,amsthm,graphicx}
\usepackage{upref}
\usepackage{epsfig}
\usepackage[utf8]{inputenc}
\usepackage{latexsym}
\usepackage{eucal}
\usepackage{wrapfig}
\usepackage{color}
\newfont{\fnt}{cmsy10}
\newfont{\sss}{cmti10}

\theoremstyle{definition}
\newtheorem{df}{Definition}[section]
\theoremstyle{plain}
\newtheorem{thm}{Theorem}[section]
\newtheorem{lemma}{Lemma}[section]
\newtheorem{prp}{Proposition}[section]
\theoremstyle{definition}
\newtheorem{rmk}{Remark}[section]

\begin{document}
\author{Barbora Volná}
\address{Mathematical Institute, Silesian University in Opava, \newline \indent Na Rybníčku 1, 746 01 Opava, Czech Republic}
\email{Barbora.Volna@math.slu.cz}

\title{Existence of Chaos in Plane $\mathbb{R}^2$ and its Application in Macroeconomics}

\begin{abstract}
The Devaney, Li-Yorke and distributional chaos in plane $\mathbb{R}^2$ can occur in the continuous dynamical system generated by Euler equation branching. Euler equation branching is a type of differential inclusion $\dot x \in \{f(x),g(x) \} $, where $f,g:X \subset \mathbb{R}^n \rightarrow \mathbb{R}^n$ are continuous and $f(x) \neq g(x)$ in every point $x \in X$. Stockman and Raines in \cite{stockman_raines} defined so-called chaotic set in plane $\mathbb{R}^2$ which existence leads to an existence of Devaney, Li-Yorke and distributional chaos. In this paper, we follow up on \cite{stockman_raines} and we show that chaos in plane $\mathbb{R}^2$ with two "classical" (with non-zero determinant of Jacobi's matrix) hyperbolic singular points of both branches not lying in the same point in $\mathbb{R}^2$ is always admitted. But the chaos existence is caused also by set of solutions of Euler equation branching which have to fulfil conditions following from the definition of so-called chaotic set. So, we research this set of solutions. In the second part we create new overall macroeconomic equilibrium model called IS-LM/QY-ML. The construction of this model follows from the fundamental macroeconomic equilibrium model called IS-LM but we include every important economic phenomena like inflation effect, endogenous money supply, economic cycle etc. in contrast with the original IS-LM model. We research the dynamical behaviour of this new IS-LM/QY-ML model and show when a chaos exists with relevant economic interpretation.
\end{abstract}

\maketitle

\section*{Introduction}

In this paper we focus on the research of chaos existence in plane $\mathbb{R}^2$. Two dimensional systems are very often in economics, thus the chaos description in plane $\mathbb{R}^2$ is very useful and applicable in economics. We follow up on the work of Stockman and Raines in \cite{stockman_raines}. The core of Devaney, Li-Yorke and distributional chaos existence in plane $\mathbb{R}^2$ is based on the special type of differential inclusion called Euler equation branching and on the continuous dynamical system generated by this differential inclusion. The continuous dynamical systems and chaos is also researched in e.g. \cite{wiggins} and the differential inclusions in e.g. \cite{smirnov}. Euler equation branching consists of two branches. The one separate branch is classical two-dimensional system of differential equations. The set of solutions of Euler equation branching contains the solutions only of separate branches and also switching solutions between these two branches. Every branch can produce some singular points and the combination of hyperbolic singular points not lying in the same point of these two branches can provide so-called chaotic set. The definition of chaotic set is provided in \nolinebreak \cite{stockman_raines}. Stockman and Raines in \cite{stockman_raines} also proved that existence of the chaotic set leads to an existence of Devaney, Li-Yorke and distributional chaos. We provide a comprehensive overview of every possible combinations of hyperbolic singular points of both branches not lying in the same point and we show that so-called chaotic set is there always admitted. We need to have also a set of switching solutions of Euler equation branching which fulfil conditions of the definition of chaotic set, i.e. which are Devaney, Li-Yorke and distributional chaotic, for certainty of the existence of so-called chaotic set in plane \nolinebreak $\mathbb{R}^2$, i.e. of Devaney, Li-Yorke and distributional chaos in the dynamical system generated by Euler equation branching in plane $\mathbb{R}^2$. We prove that such chaotic set of solutions exists and that the set of such chaotic sets of solutions are uncountable. We also add some lemmas and remarks to complete the theoretical part of this paper focusing on the existence of chaos in plane \nolinebreak $\mathbb{R}^2$. 

The economic situation in these days, the phenomena, where the "classical" (macro-) economic models or the prediction of the future economic progress according these models fail, inspire me to create the new overall macroeconomic model called IS-LM/QY-ML. This new macroeconomic model describes macroeconomic situation including every important economic phenomena like an aggregate macroeconomic equilibrium or (un)stability, an inflation effect, an endogenous money supply, an economic cycle etc. in one overall model. As we can see in this paper, from the perspective of this new model the dynamical behaviour of economy can be very chaotic and unexpected, the aggregate macroeconomic stability can be very frail and sensitive with respect to external influences. New IS-LM/QY-ML model is based on the fundamental macroeconomic IS-LM model. This model explains the aggregate macroeconomic equilibrium, i.e. the goods market equilibrium and the money market or financial assets market equilibrium simultaneously. Already in 1937, J. R. Hicks in \cite{hicks} published the original IS-LM model as a mathematical expression  of J. M. Keynes's theory. After formulation of the original IS-LM model during many decades many versions of this model and related problems were presented in several works, see e.g. \cite{cesare_sportelli}, \cite{chiba_leong}, \cite{gandolfo}, \cite{king}, \cite{neri_venturi}, \cite{romer} and \cite{zhou-li}. The original IS-LM model has several deficiencies, some of subsequent versions deal with modification of this original model but for us requirements we found our way to research this problem and to eliminate the deficiencies of original model. Primarily, Hicks built his model on one concrete economic situation, i.e. the original IS-LM model describes the economy in a recessionary gap. From this follows his assumptions of constant price level and of demand-oriented model. But our overall model describes all phases of the economic cycle and the properties connected with this. So, we firstly include inflation effect to our new model. For do this we inspire by one of the version of original IS-LM model, by IS-ALM model with expectations and the term structure of interest rates, see \cite{baily_friedman}. Then, we consider also a supply-oriented view of the macroeconomic situation using QY-ML model newly constructed in this paper. The QY-ML model describes simultaneous goods market equilibrium and money market equilibrium under supply-oriented point of view in contrast with IS-LM model. Thus, our new overall IS-LM/QY-ML model consists of two "sub-models": demand-oriented "sub-model" - modified IS-LM and supply-oriented "sub-model" - new QY-ML model. Depending on the phase of economic cycle the one of these sub-models holds. The switching between these phases is represented by switching between these two sub-models. The mathematical tool to describe the holding of two "sub-models" and switching between them is exactly Euler equation branching and the continuous dynamical system generated by this differential inclusion. Secondly, Hicks and also economists of that time assumed a strictly exogenous money supply. This supply of money is certain constant money stock determined by central bank. For this conception IS-LM model was the most criticised. The opposite conception is an endogenous money supply which assumes money generated in economy by credit creation, see e.g. \cite{sojka}. But even today's economists can not find any consensus in the problem of endogenity or exogenity of money supply, see e.g. \cite{badarudin_ariff_khalid} or \cite{sedlacek}. We resolve this dilemma by conjunction of the endogenous and exogenous conception of money supply including some money supply function to this model. 

The dynamical behaviour of new macroeconomic IS-LM/QY-ML model usually leads to hyperbolic singular points of both branches lying in the different point in $\mathbb{R}^2$. So, there possibly exists area of chaotic behaviour of the economy. Furthermore, if some economic cycle with in this article described types or similar types of periods influences economy, then the economy behaves chaotically in the area of $\mathbb{R}^2$ given by levels of the aggregate income and of the long-term real interest rate. Besides this also another authors deal with some type of chaos or bifurcations in economics, see e.g. \cite{puu} and \cite{zhang}. In this paper we examine the most typical case of economy describing by IS-LM/QY-ML model and show existence of chaos in such case with relevant economic interpretation of causes. But for less typical cases represented by unusual behaviour of economic subjects this chaos existence possibility is similar. 

Summary, this paper consists of two part - the theoretical research about chaos in plane $\mathbb{R}^2$ and its application in macroeconomics. In the first theoretical part there is the description of chaos existence in plane $\mathbb{R}^2$ which is given by continuous dynamical system generated by special type of differential inclusion called Euler equation branching. In the second application part the new overall macroeconomic equilibrium IS-LM/QY-ML model is constructed and the dynamical behaviour of this model can produce chaos in economy.

\section{Preliminaries}
\label{preliminaries}

All used definitions and theorems in this section follow from \cite{stockman_raines} and are modified to the special type of differential inclusion in plane $\mathbb{R}^2$.

\begin{df}
\label{df:Euler_equation_branching}
Let $X \subset \mathbb{R}^2$ be open set and $f,g:X \rightarrow \mathbb{R}^2$ be continuous. Let us consider differential inclusion given by $\dot x \in \{f(x),g(x) \} $. We say that there is \textit{Euler equation branching in the point} $x \in X$ if $f(x) \neq g(x)$. If there is Euler equation branching in every point $x \in X$ than we say that there is \textit{Euler equation branching on the set} $X$.
\end{df} 
\begin{rmk}
The solution of the differential inclusion of Euler equation branching type is function $x(t)$ which is the continuous and continuously differentiable a.e. and satisfies $\dot x \in \{f(x),g(x) \}$. The set of solutions includes the solution of the one branch satisfying $\dot x = f(x)$, the solution of the second branch satisfying $\dot x = g(x)$ and the "switching" between these two branches.
\end{rmk}
In the text below we consider $X \subseteq \mathbb{R}^2$ is non-empty open set with Euclidean metric $d$ and $T:=[0, \infty)$ is time index. Let $F: X \rightarrow 2^{\mathbb{R}^2}$ be set-valued function given by $F(x):=\{ f(x),g(x) \}$ where $f,g:X \rightarrow \mathbb{R}^2$ are continuous and $f(x) \neq g(x)$ is satisfied for all $x \in X$. $Z=\{ \gamma| \gamma: T \rightarrow X \}$, where functions $\gamma: T \rightarrow X$ are continuous and continuously differentiable a.e.
\begin{df}
The \textit{dynamical system generated by} $F$ is given by
$$D:=\{\gamma \in Z | \dot \gamma(t) \in F(\gamma(t)) \textit{ a.e.} \}$$
\end{df}
\begin{df}
We say that $V \subset \mathbb{R}^2$ non-empty is \textit{compact $F$-invariant set}, if $V$ is compact and for each $x \in V$ there exist a $\gamma \in D$ such that $\gamma(0)=x$ and $\gamma(t) \in V$ for all $t \in T$. 
\end{df}
Let $V^*=\{ \gamma \in D | \gamma(t) \in V, \textit{ for all }t \in T \}$ where $V \subset \mathbb{R}^2$ is compact  $F$-invariant set.
\begin{df}
Let $a,b \in X \subseteq \mathbb{R}^2$ and $D$ is a dynamical system with sense mentioned above. Let $\gamma \in D$, $t_0,t_1 \in T$ such that $t_0 < t_1$. A \textit{simple path from $a$ to $b$ generated by $D$} is given by $P:=\{ \gamma(t):t_0 \leq t \leq t_1 \}$ such that $\gamma(t_0)=a$, $\gamma(t_1)=b$ and $\dot \gamma$ has finitely many discontinuities on $[t_0,t_1]$ and $a \neq \gamma(s) \neq b$ for all $t_0<s<t_1$.
\end{df}
\begin{df}
\label{df:chaotic_set}
Let $V \subset X \subseteq \mathbb{R}^2$ be a non-empty compact $F$-invariant set and $V^*=\{ \gamma \in D | \gamma(t) \in V, \textit{ for all }t \in T \}$. $V$ is so-called \textit{chaotic set} provided
\begin{enumerate}
\item \label{df:chaotic_set_a} for all $a,b \in V$, there exists a simple path from $a$ to $b$ generated by $V^*$,
\item \label{df:chaotic_set_b} there exists $U \subset V$ non-empty and open (relative to $V$) and $\gamma \in V^*$ such that $\gamma(t) \in V \setminus U$ for all $t \in T$ (i.e. there exists $\gamma \in V^*$ such that $\{ \gamma(t):t \in T \}$ is not dense in $V$).
\end{enumerate}
\end{df}
\begin{thm}
\label{thm:chaotic_set-Devaney_chaos}
If $V$ is chaotic set then $V^*$ has Devaney chaos.
\end{thm}
\begin{thm}
\label{thm:chaotic_set_non_empty_interior-Li-Yorke_distributional_chaos}
If $V$ is chaotic set with non-empty interior then $V^*$ has Li-Yorke chaos and distributional chaos.
\end{thm}
\begin{thm}
\label{thm:chaotic_set_empty_interior-Li-Yorke_distributional_chaos}
If $V$ is chaotic set homeomorphic to $[0,1]$ with $\|f(x)\| \|g(x)\|>0$ and $\theta:=cos^{-1}\left( \frac{f(x) \cdot g(x)}{\|f(x)\| \|g(x)\|} \right)=\pi$ for all $x \in V$, then $V^*$ has Li-Yorke and distributional chaos.
\end{thm}
\begin{lemma}
\label{lemma_chaotic_set_in_R2}
Let $K \subset \mathbb{R}^2$ be non-empty closed set such that $g$ has unbounded solutions in $K$. Let $a \in K$ and let $P$ be simple path from $a$ to $a$ generated by $D$ such that $P \subset K$ (so $P$ can be a finite union of arcs and Jordan curves). $C_i$, $i=1..n$ ($n$ is number of Jordan curves) denote interiors (bounded components) of such Jordan curves (according to Jordan Curve Theorem). Then $P \cup \bigcup_{i=1..n}C_i$ is chaotic set.
\end{lemma}
\begin{thm}
\label{thm:hyperb_sing_point-chaotic_set}
Let $x^* \in X \subseteq \mathbb{R}^2$,~$f(x^*)=0$ and $g(x^*) \neq 0$, let $\lambda_1,~\lambda_2$ be eigenvalues of Jacobi's matrix of the system $\dot x = f(x)$ in the point $x^*$ and $e_1,~e_2$ be corresponding eigenvectors. We choose $\delta > 0$ such that $g(x) \neq 0$ for every $x \in \bar{B}_{\delta}(x^*)$. Let the solution of $\dot x = g(x)$ be unbounded in $\bar{B}_{\delta}(x^*)$ (non-empty closed subset of $X \subset \mathbb{R}^2$). 
\begin{enumerate}
\item \label{thm:sink_source-chaotic_set} We assume that there exists $\varepsilon > 0$ such that $x^*$ is source (i.e. unstable node or focus) or sink (i.e. stable node or focus) for $f$ on $B_{\varepsilon}(x^*)$. Then $F$ admits a chaotic set.
\item \label{thm:saddle-chaotic_set} We assume that $\lambda_1<0,~\lambda_2>0$ (i.e. $x^*$ is saddle point) and $g(x^*) \neq \alpha e_1$ and $g(x^*) \neq \beta e_2$, where $\alpha,\beta \in \mathbb{R} \setminus \{0\}$. Then $F$ admits a chaotic set with non-empty interior.
\end{enumerate}
\end{thm}

\newpage
\section{Chaos in Continuous Dynamical System \\ Generated by Euler Equation Branching in plane $\mathbb{R}^2$}

First, we show that the existence of so-called chaotic set $V \subset X \subseteq \mathbb{R}^2$ (see Definition \ref{df:chaotic_set}) is always admitted in the continuous dynamical system generated by Euler equation branching. In some sense, we can find so-called Parrondo's paradox in this continuous dynamical system. In this way, it means that two asymptotically stable solution of two branches $\dot x = f(x)$ and $\dot x = g(x)$ can produce chaotic sets in the continuous dynamical system generated by Euler equation branching $\dot x \in \{f(x),g(x) \}$. We extend the theory presented by Stockman and Raines in \cite{stockman_raines}. We create a comprehensive overview of all possibilities with detecting of so-called chaotic sets (see definition \ref{df:chaotic_set}) in such systems.

The chaotic set $V$ has three properties which are connected with the set of solutions $V^*$. The set $V^*$ contains solutions $\gamma$ which are together Devaney, Li-Yorke and distributional chaotic. The element $\gamma$ of dynamical system generated by Euler equation branching $D:=\{\gamma \in Z | \dot \gamma(t) \in F(\gamma(t)) \textit{ a.e.} \}$ is solution covering the part corresponding to the branch $\dot x = f(x)$, the part corresponding to the branch $\dot x = g(x)$ and also information when one branch is switched to the other. We show how such Devaney, Li-Yorke and distributional chaotic set of solutions $V^*$ can look like and we research the set of all Devaney, Li-Yorke and distributional chaotic sets $V^*$.

We consider only classical singular points corresponding to both branches $\dot x=f(x)$ and $\dot x=g(x)$, means with non-zero determinant of Jacobi's matrix of considering system. Furthermore, we consider that the singular points of both branches do not lie in the same point in $\mathbb{R}^2$, means $f(x^*)=0$ and $g(y^*)=0$ for $x^* \neq y^*$. Finally, we assume that both branches produce hyperbolic singular points, i.e. eigenvalues of Jacobi's matrices corresponding to the branch $f$ in the point $x^*$ and to $g$ in $y^*$ are not purely imaginary.

The cases where the both branches produce hyperbolic singular points and these points lie in the same point in $\mathbb{R}^2$ or at least one branch produce periodic solution (cycle) are not considered in this paper, but the principle and results in these cases seem to be very similar as in the considered case.  

\subsection{Comprehensive Overview of All Possibilities Admitting Chaotic Set $V$}
\label{overview_of_possibilities_admitting_chaotic_set}

First we research the possibilities on which Theorem \ref{thm:hyperb_sing_point-chaotic_set} can be applicable and which all admit chaotic sets, i.e. 
\begin{itemize}
\item combinations of $x^*$ sink or source (the stable or unstable node or focus) in the first branch and unbounded solution in $\bar{B}_{\delta}(x^*)$ in the second branch, 
\item or combinations of $x^*$ saddle in the first branch and unbounded solution in $\bar{B}_{\delta}(x^*)$ in the second branch with condition that the trajectory of unbounded solution passing through the saddle point $x^*$ has not the same or the directly opposite direction as the stable or unstable manifold of the saddle point $x^*$ in the point $x^* \in \mathbb{R}^2$, i.e. the vector $g(x^*)$ is not collinear with the eigenvectors $e_1$ or $e_2$ of the Jacobi's matrix corresponding to $f$ in the point $x^*$.
\end{itemize}

In Table \ref{tab:overview_hyp_sing_points_combinations_chaotic_sets} there is the overview of all combinations of hyperbolic singular points, references to the theorems from which the possible existence of chaotic sets and thus Devaney, Li-Yorke and distributional chaos follow, and links to the corresponding figures where the chaotic sets are indicated.

\begin{table}[ht]
\begin{center}
\begin{tabular}{|l|l|l|}
\hline
Branches $\dot x = f(x)$ and $\dot x = g(x)$ & Proof & Fig. \\
\hline \hline
\textit{unstable node - unstable node} &  Theorem \ref{thm:hyperb_sing_point-chaotic_set} (\ref{thm:sink_source-chaotic_set}), \ref{thm:chaotic_set-Devaney_chaos}, \ref{thm:chaotic_set_non_empty_interior-Li-Yorke_distributional_chaos} or \ref{thm:chaotic_set_empty_interior-Li-Yorke_distributional_chaos} & \ref{fig:unstable_node_unstable_node_1} or \ref{fig:unstable_node_unstable_node_2}\\
\hline
\textit{unstable node - stable node} & Theorem \ref{thm:hyperb_sing_point-chaotic_set} (\ref{thm:sink_source-chaotic_set}), \ref{thm:chaotic_set-Devaney_chaos}, \ref{thm:chaotic_set_non_empty_interior-Li-Yorke_distributional_chaos} or \ref{thm:chaotic_set_empty_interior-Li-Yorke_distributional_chaos} & \ref{fig:unstable_node_stable_node_1} or \ref{fig:unstable_node_stable_node_2} \\
\hline
\textit{unstable node - unstable focus} & Theorem \ref{thm:hyperb_sing_point-chaotic_set} (\ref{thm:sink_source-chaotic_set}), \ref{thm:chaotic_set-Devaney_chaos}, \ref{thm:chaotic_set_non_empty_interior-Li-Yorke_distributional_chaos} & \ref{fig:unstable_node_unstable_focus} \\
\hline
\textit{unstable node - stable focus} & Theorem \ref{thm:hyperb_sing_point-chaotic_set} (\ref{thm:sink_source-chaotic_set}), \ref{thm:chaotic_set-Devaney_chaos}, \ref{thm:chaotic_set_non_empty_interior-Li-Yorke_distributional_chaos} & \ref{fig:unstable_node_stable_focus}\\
\hline
\textit{unstable node - unstable saddle} & Theorem \ref{thm:hyperb_sing_point-chaotic_set} (\ref{thm:sink_source-chaotic_set}) or (\ref{thm:saddle-chaotic_set}), \ref{thm:chaotic_set-Devaney_chaos}, \ref{thm:chaotic_set_non_empty_interior-Li-Yorke_distributional_chaos} or \ref{thm:chaotic_set_empty_interior-Li-Yorke_distributional_chaos} & \ref{fig:unstable_node_unstable_saddle_1} or \ref{fig:unstable_node_unstable_saddle_2} \\
\hline
\textit{stable node - stable node} & Theorem \ref{thm:hyperb_sing_point-chaotic_set} (\ref{thm:sink_source-chaotic_set}), \ref{thm:chaotic_set-Devaney_chaos}, \ref{thm:chaotic_set_non_empty_interior-Li-Yorke_distributional_chaos} or \ref{thm:chaotic_set_empty_interior-Li-Yorke_distributional_chaos} & \ref{fig:stable_node_stable_node_1} or \ref{fig:stable_node_stable_node_2}\\
\hline
\textit{stable node - unstable focus} & Theorem \ref{thm:hyperb_sing_point-chaotic_set} (\ref{thm:sink_source-chaotic_set}), \ref{thm:chaotic_set-Devaney_chaos}, \ref{thm:chaotic_set_non_empty_interior-Li-Yorke_distributional_chaos} & \ref{fig:stable_node_unstable_focus} \\
\hline
\textit{stable node - stable focus} & Theorem \ref{thm:hyperb_sing_point-chaotic_set} (\ref{thm:sink_source-chaotic_set}), \ref{thm:chaotic_set-Devaney_chaos}, \ref{thm:chaotic_set_non_empty_interior-Li-Yorke_distributional_chaos} & \ref{fig:stable_node_stable_focus} \\
\hline
\textit{stable node - unstable saddle} & Theorem \ref{thm:hyperb_sing_point-chaotic_set} (\ref{thm:sink_source-chaotic_set}) or (\ref{thm:saddle-chaotic_set}), \ref{thm:chaotic_set-Devaney_chaos}, \ref{thm:chaotic_set_non_empty_interior-Li-Yorke_distributional_chaos} or \ref{thm:chaotic_set_empty_interior-Li-Yorke_distributional_chaos} & \ref{fig:stable_node_unstable_saddle_1} or \ref{fig:stable_node_unstable_saddle_2} \\
\hline
\textit{unstable focus - unstable focus} & Theorem  \ref{thm:hyperb_sing_point-chaotic_set} (\ref{thm:sink_source-chaotic_set}), \ref{thm:chaotic_set-Devaney_chaos}, \ref{thm:chaotic_set_non_empty_interior-Li-Yorke_distributional_chaos} & \ref{fig:unstable_focus_unstable_focus} \\
\hline
\textit{unstable focus - stable focus} & Theorem \ref{thm:hyperb_sing_point-chaotic_set} (\ref{thm:sink_source-chaotic_set}), \ref{thm:chaotic_set-Devaney_chaos}, \ref{thm:chaotic_set_non_empty_interior-Li-Yorke_distributional_chaos} & \ref{fig:unstable_focus_stable_focus} \\
\hline
\textit{unstable focus - unstable saddle} & Theorem \ref{thm:hyperb_sing_point-chaotic_set} (\ref{thm:sink_source-chaotic_set}) or (\ref{thm:saddle-chaotic_set}), \ref{thm:chaotic_set-Devaney_chaos}, \ref{thm:chaotic_set_non_empty_interior-Li-Yorke_distributional_chaos} & \ref{fig:unstable_focus_unstable_saddle_1} \\
\hline
\textit{stable focus - stable focus} & Theorem \ref{thm:hyperb_sing_point-chaotic_set} (\ref{thm:sink_source-chaotic_set}), \ref{thm:chaotic_set-Devaney_chaos}, \ref{thm:chaotic_set_non_empty_interior-Li-Yorke_distributional_chaos} & \ref{fig:stable_focus_stable_focus} \\
\hline
\textit{stable focus - unstable saddle} & Theorem  \ref{thm:hyperb_sing_point-chaotic_set} (\ref{thm:sink_source-chaotic_set}) or (\ref{thm:saddle-chaotic_set}), \ref{thm:chaotic_set-Devaney_chaos}, \ref{thm:chaotic_set_non_empty_interior-Li-Yorke_distributional_chaos} & \ref{fig:stable_focus_unstable_saddle_1} \\
\hline
\textit{unstable saddle - unstable saddle} & Theorem  \ref{thm:hyperb_sing_point-chaotic_set} (\ref{thm:saddle-chaotic_set}), \ref{thm:chaotic_set-Devaney_chaos}, \ref{thm:chaotic_set_non_empty_interior-Li-Yorke_distributional_chaos} & \ref{fig:unstable_saddle_unstable_saddle_1_2} \\
\hline
\end{tabular}
\vspace{0.3cm}
\caption{Overview of hyperbolic singular points combinations admitting chaotic sets with condition for saddle $g(x^*) \neq \alpha e_1$ and $g(x^*) \neq \beta e_2$}
\label{tab:overview_hyp_sing_points_combinations_chaotic_sets}
\end{center}
\end{table}

On the illustrative figures we present the both situations, where the branch $f$ corresponds to the black trajectories or to the blue trajectories according to relevant theorems. That is why there are predominately two figured chaotic sets. The chaotic sets are displayed by red hatched areas. The arrows show the directions of the trajectories. The principle of such chaotic behaviour is based on switching between these two branches. First the solution of such system goes alongside the black trajectory in direction of the corresponding arrow, then this solution switches to the second branch and goes alongside the blue trajectory in direction of the corresponding arrow etc. And vice versa, the moving point can go first alongside the blue trajectory and then after system switch alongside the black trajectory etc.

\begin{figure}[htp]
\centering
\begin{minipage}[c]{225pt}
\hspace{0.2cm}
  \includegraphics[width=195pt]{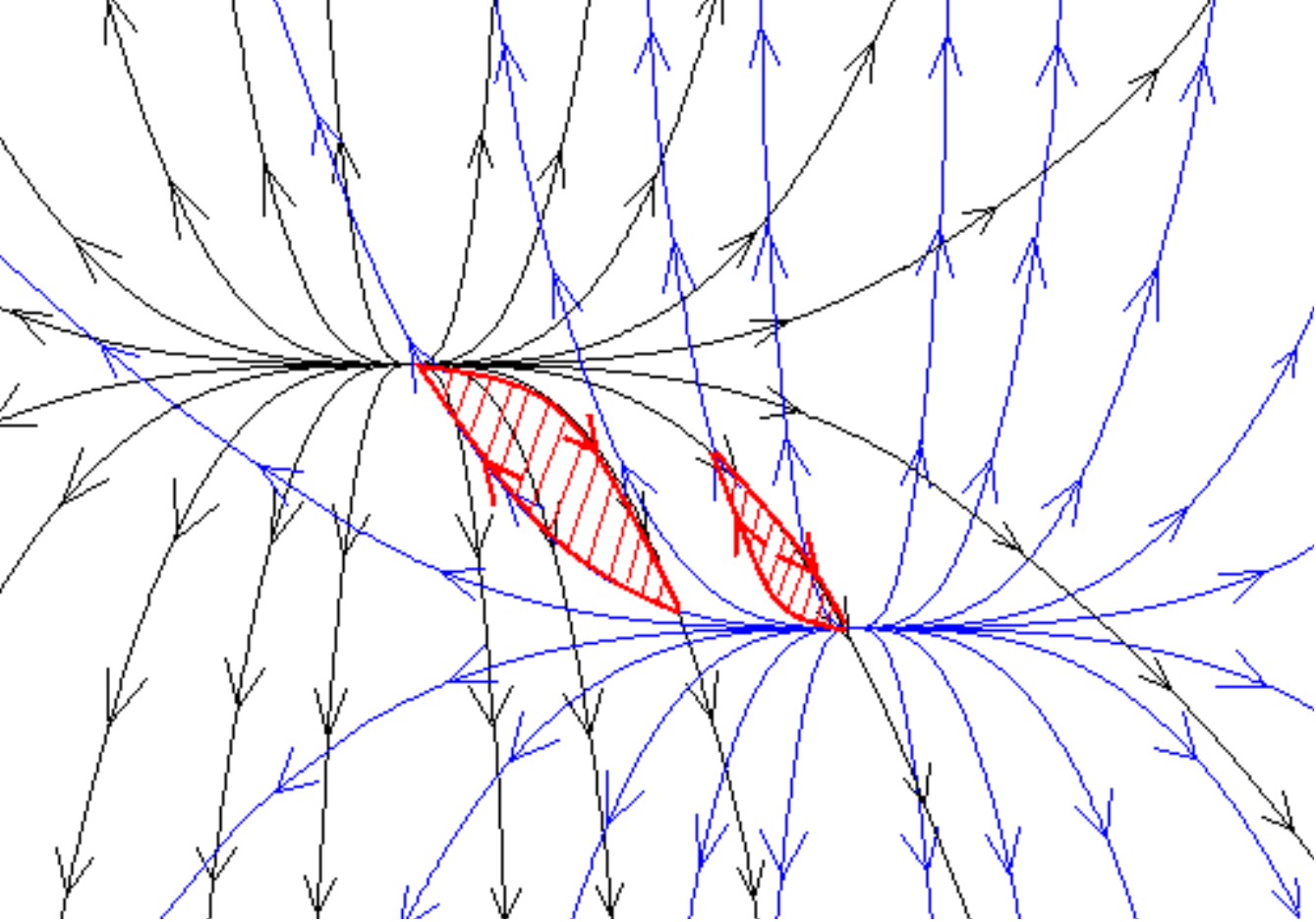}
\end{minipage}
\begin{minipage}[c]{225pt}
  \includegraphics[width=210pt]{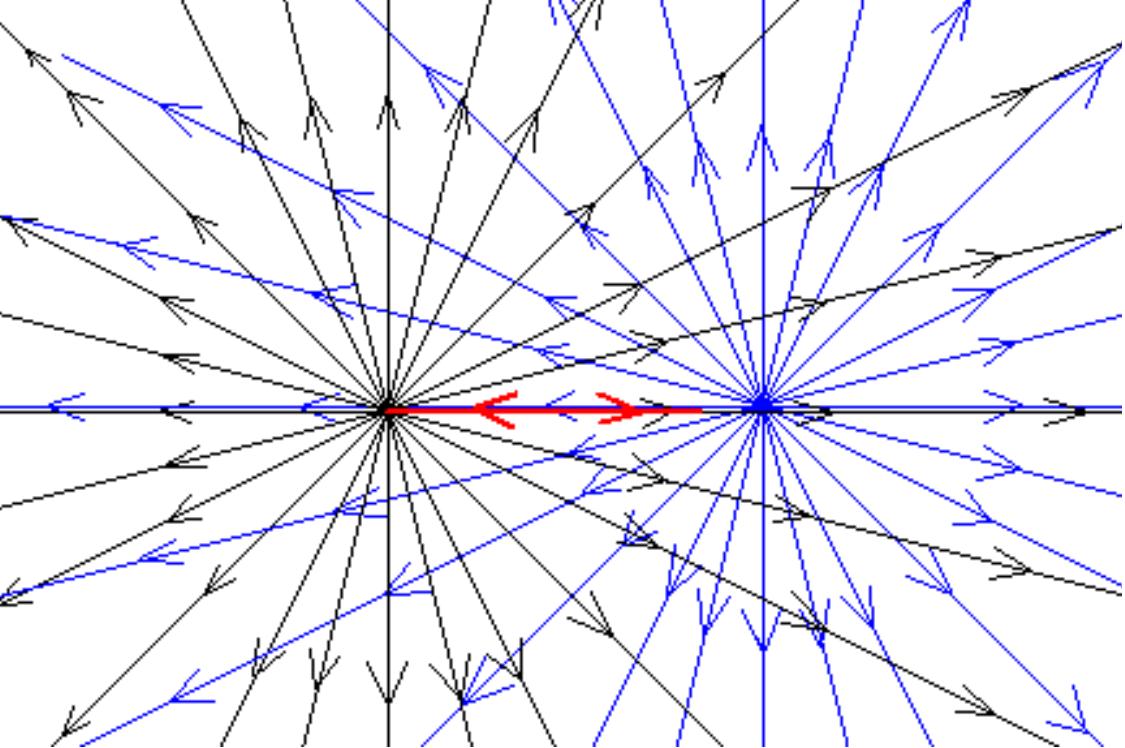}
\end{minipage}
\\
\begin{minipage}[c]{225pt}
  \caption{Chaotic sets with non-empty interior between two unstable nodes}
  \label{fig:unstable_node_unstable_node_1}
\end{minipage}
\begin{minipage}[c]{225pt}
  \caption{Chaotic set homeomorphic to $[0,1]$ with $\|f(x)\| \|g(x)\|>0$ and $\theta=\pi$ between two unstable nodes}
  \label{fig:unstable_node_unstable_node_2}
\end{minipage}
\vspace{0.7cm}
\\
\begin{minipage}[c]{225pt}
\hspace{0.5cm}  
  \includegraphics[width=180pt]{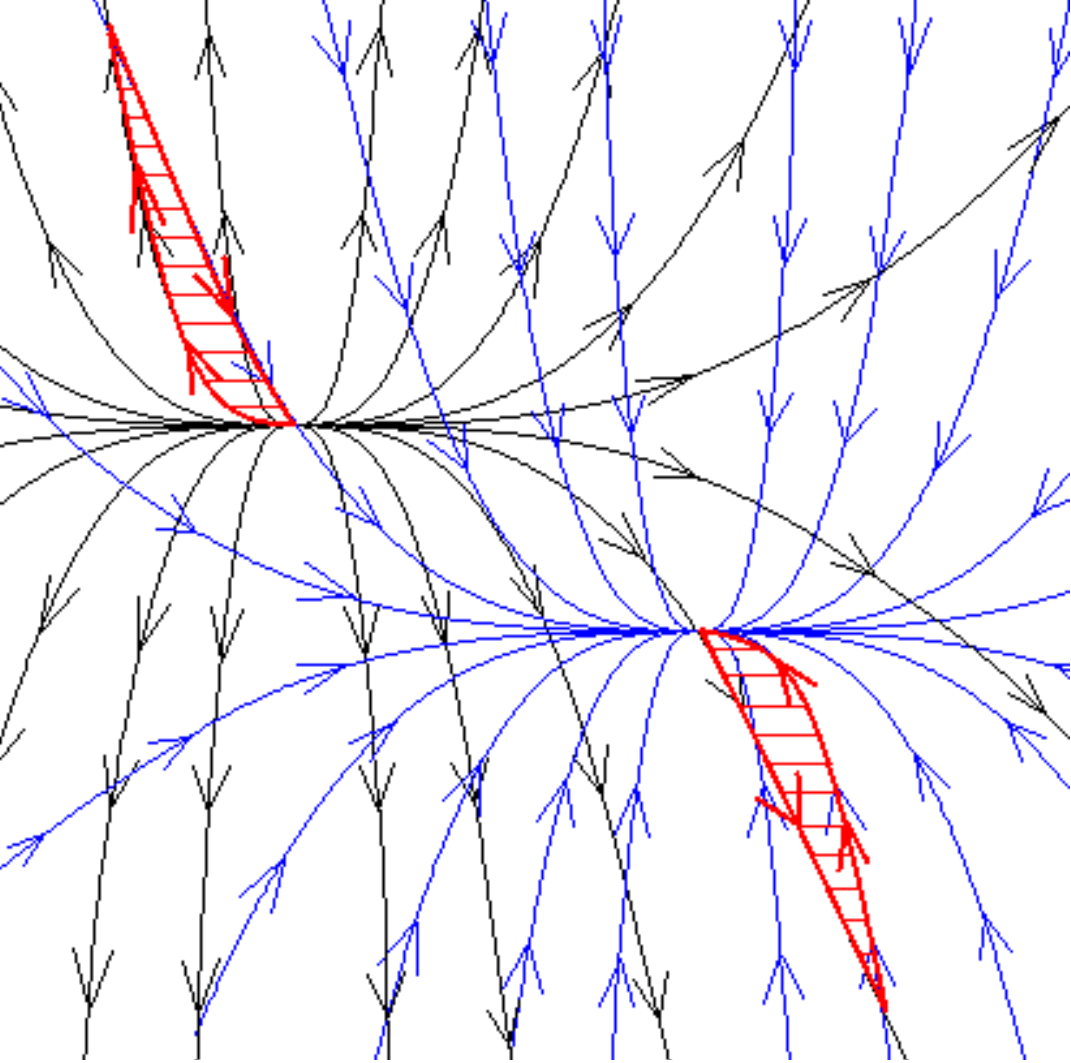}
\end{minipage}
\begin{minipage}[c]{225pt}
\hspace{0.7cm}  
  \includegraphics[width=178pt]{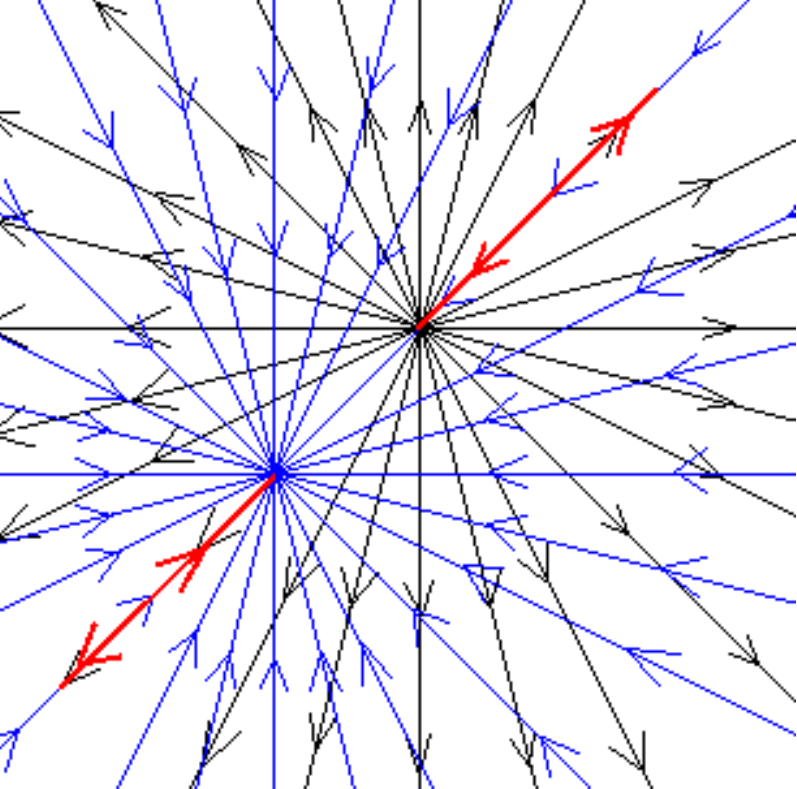}
\end{minipage}
\\
\begin{minipage}[c]{225pt}
  \caption{Chaotic sets with non-empty interior between unstable node and stable node}
  \label{fig:unstable_node_stable_node_1}
\end{minipage}
\begin{minipage}[c]{225pt}
  \caption{Chaotic sets homeomorphic to $[0,1]$ with $\|f(x)\| \|g(x)\|>0$ and $\theta=\pi$ between unstable and stable node}
  \label{fig:unstable_node_stable_node_2}
\end{minipage}
\end{figure}

\begin{figure}[htp]
\centering
\begin{minipage}[c]{215pt}
  \includegraphics[width=205pt]{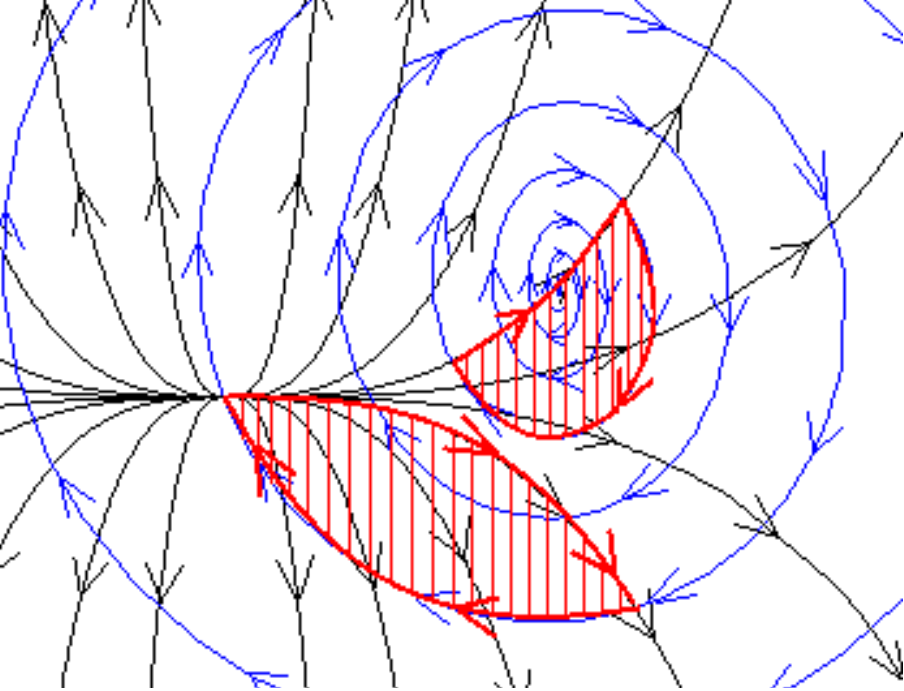}
\end{minipage}
\begin{minipage}[c]{210pt}
\hspace{0.5cm}  
 \includegraphics[width=180pt]{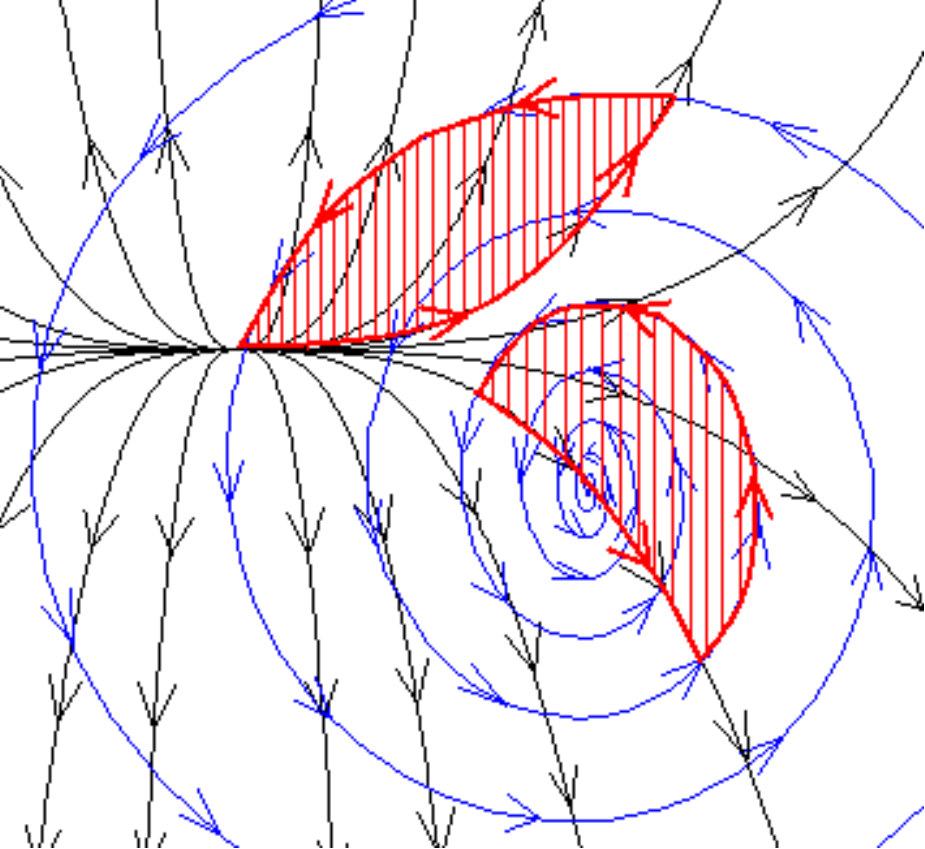}
\end{minipage} 
\begin{minipage}[c]{225pt} 
  \caption{Chaotic sets with non-empty interior between unstable node and unstable focus}
  \label{fig:unstable_node_unstable_focus}
\end{minipage}
\begin{minipage}[c]{225pt}
  \caption{Chaotic sets with non-empty interior between unstable node and stable focus}
  \label{fig:unstable_node_stable_focus}
\end{minipage}
\\
\vspace{1cm}
\begin{minipage}[c]{225pt}
\hspace{0.3cm}    
  \includegraphics[width=197pt]{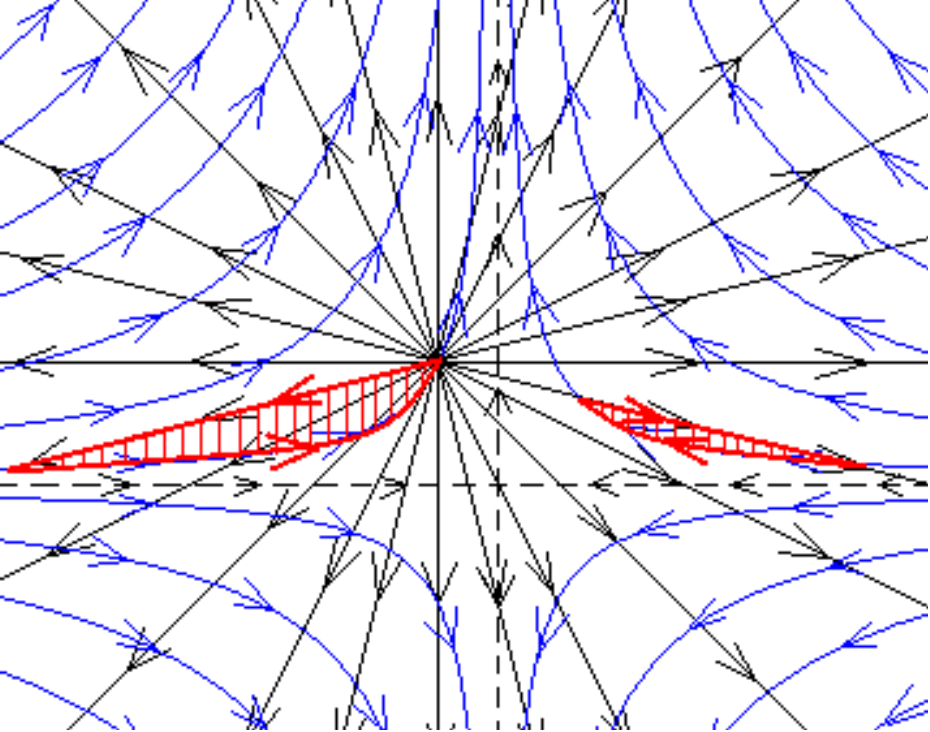}
\end{minipage}
\begin{minipage}[c]{225pt}
\hspace{0.1cm}  
 \includegraphics[width=214pt]{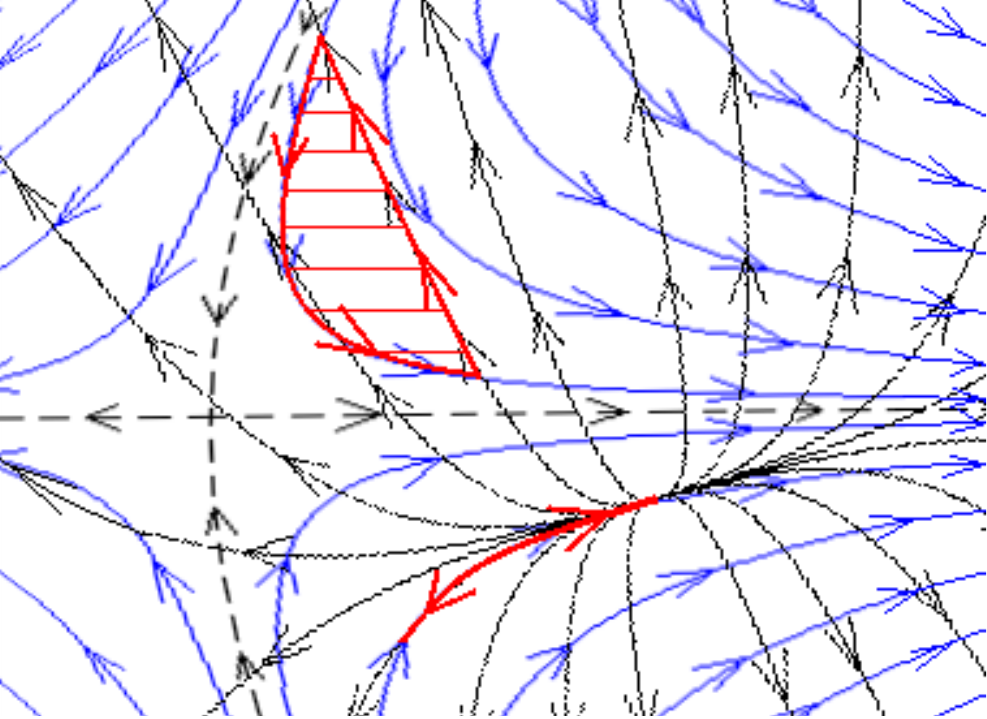}
\end{minipage}  
\\
\begin{minipage}[c]{225pt}
 \caption{Chaotic sets with non-empty interior between unstable node and unstable saddle fulfilling conditions $g(x^*) \neq \alpha e_1$ and $g(x^*) \neq \beta e_2$}
  \label{fig:unstable_node_unstable_saddle_1}
\end{minipage}
\begin{minipage}[c]{225pt}
 \caption{Chaotic sets with non-empty interior and homeomorphic to $[0,1]$ with $\|f(x)\| \|g(x)\|>0$ and $\theta=\pi$ between unstable node and unstable saddle fulfilling conditions $g(x^*) \neq \alpha e_1$ and $g(x^*) \neq \beta e_2$}
  \label{fig:unstable_node_unstable_saddle_2}
\end{minipage}
\end{figure}

\begin{figure}[htp]
\centering
\begin{minipage}[c]{225pt}
\hspace{0.6cm}  
  \includegraphics[width=175pt]{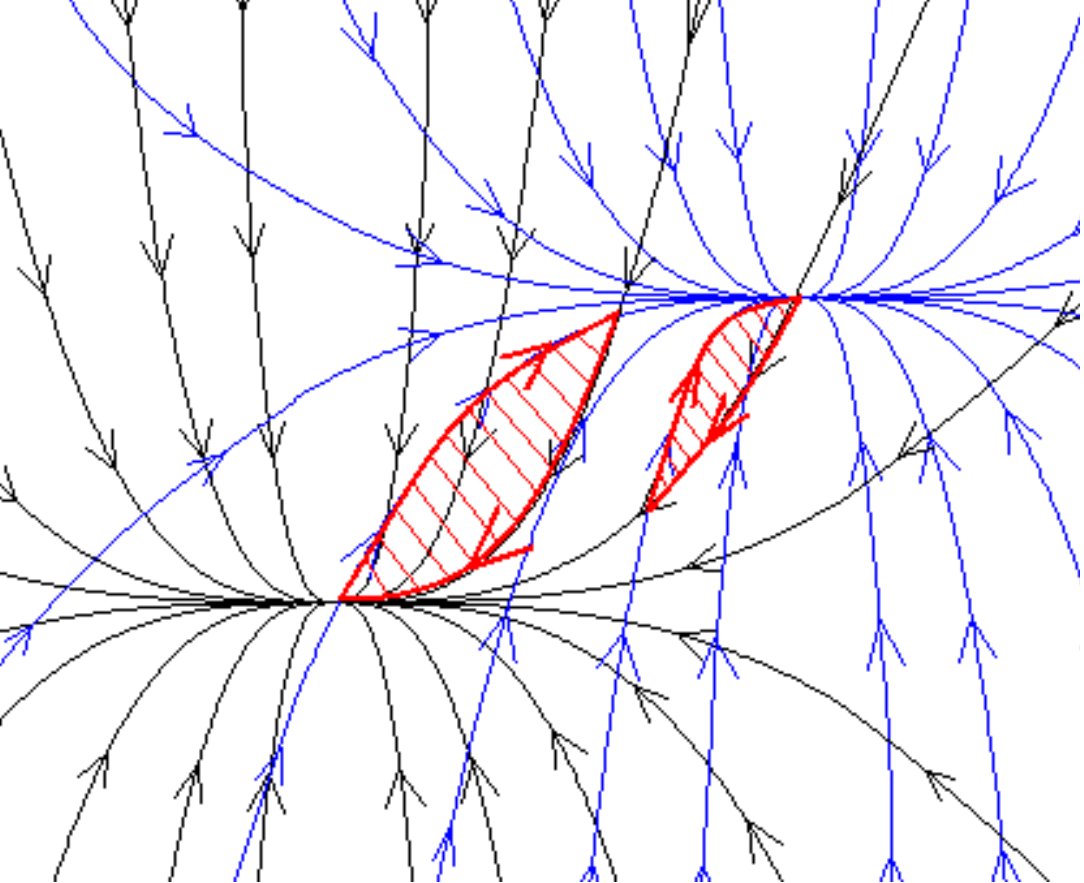}
\end{minipage}
\begin{minipage}[c]{225pt}
  \includegraphics[width=220pt]{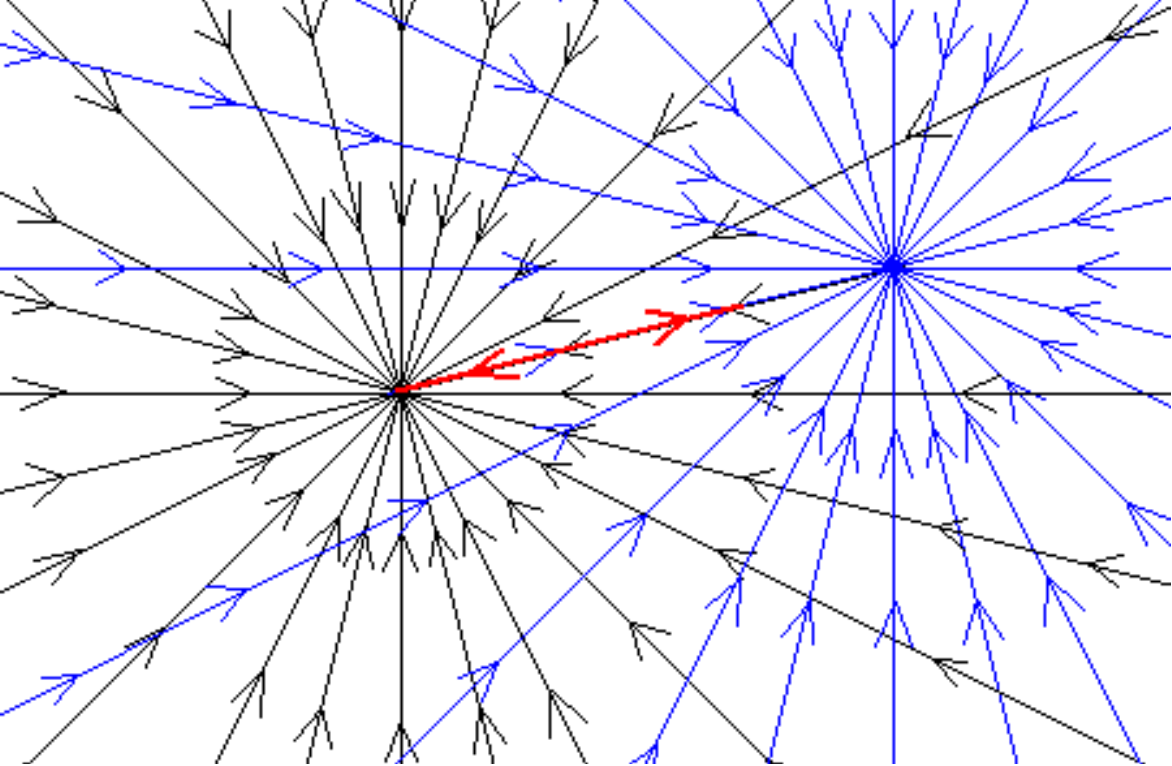}
\end{minipage}
\\
\begin{minipage}[c]{225pt}
  \caption{Chaotic sets with non-empty interior between two stable nodes}
  \label{fig:stable_node_stable_node_1}
\end{minipage}
\begin{minipage}[c]{225pt}
  \caption{Chaotic set homeomorphic to $[0,1]$ with $\|f(x)\| \|g(x)\|>0$ and $\theta=\pi$ between two stable nodes}
  \label{fig:stable_node_stable_node_2}
\end{minipage}
\\
\vspace{1cm}
\begin{minipage}[c]{210pt}
  \includegraphics[width=190pt]{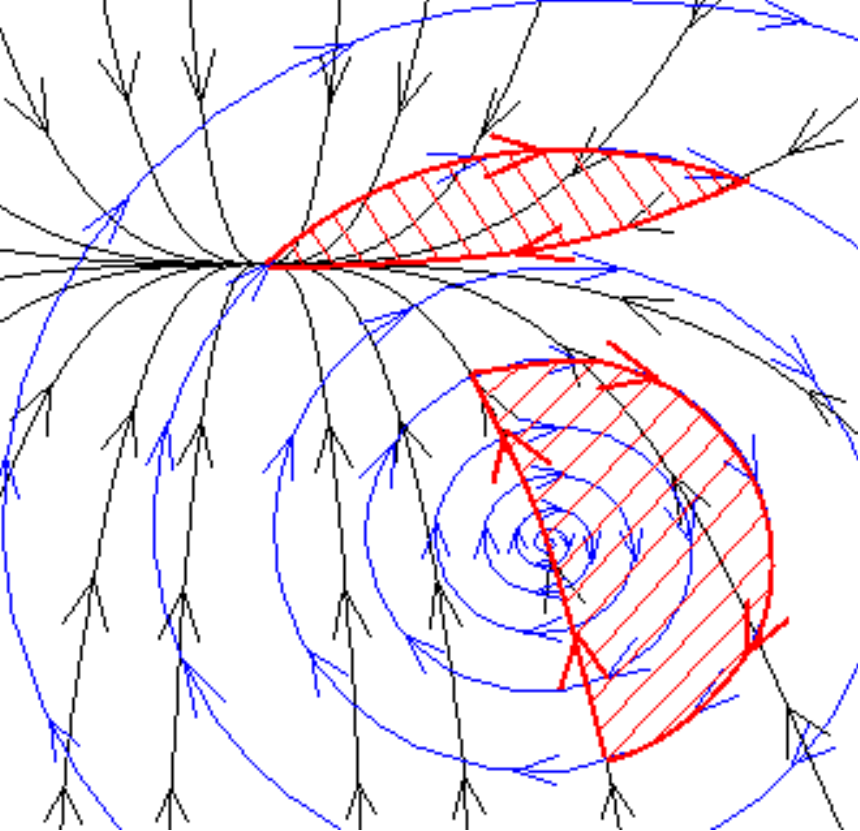}
\end{minipage}
\begin{minipage}[c]{210pt}
\hspace{0.6cm}  
 \includegraphics[width=170pt]{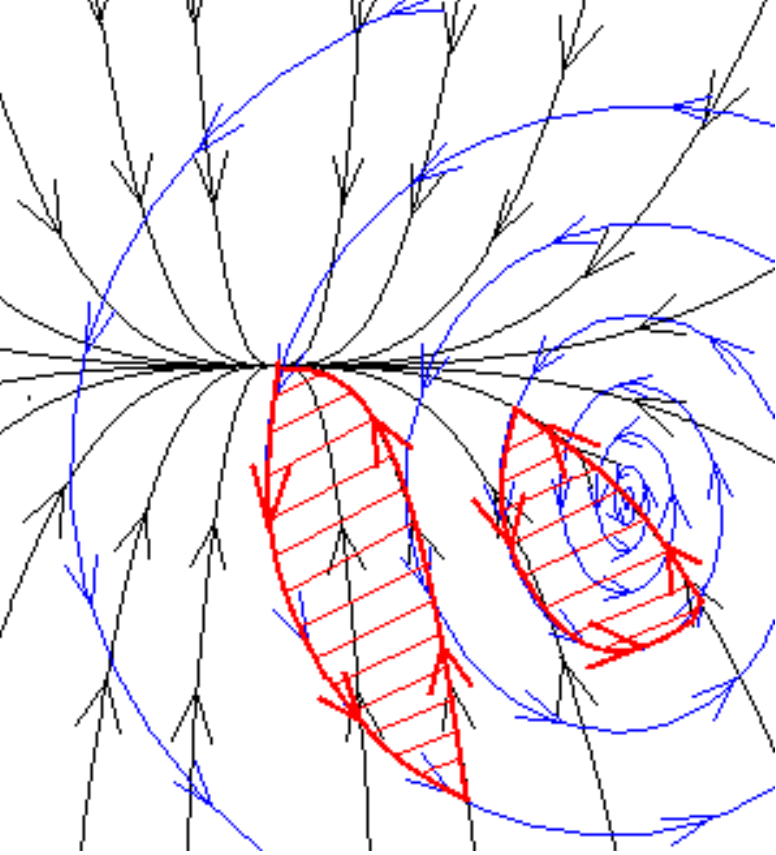}
\end{minipage}
\begin{minipage}[c]{225pt}
  \caption{Chaotic sets with non-empty interior between stable node and unstable focus}
  \label{fig:stable_node_unstable_focus}
\end{minipage}
\begin{minipage}[c]{225pt}
  \caption{Chaotic sets with non-empty interior between stable node and stable focus}
  \label{fig:stable_node_stable_focus}
\end{minipage}
\end{figure}

\begin{figure}[htp]
\centering
\begin{minipage}[c]{225pt}
\hspace{0.5cm}  
  \includegraphics[width=180pt]{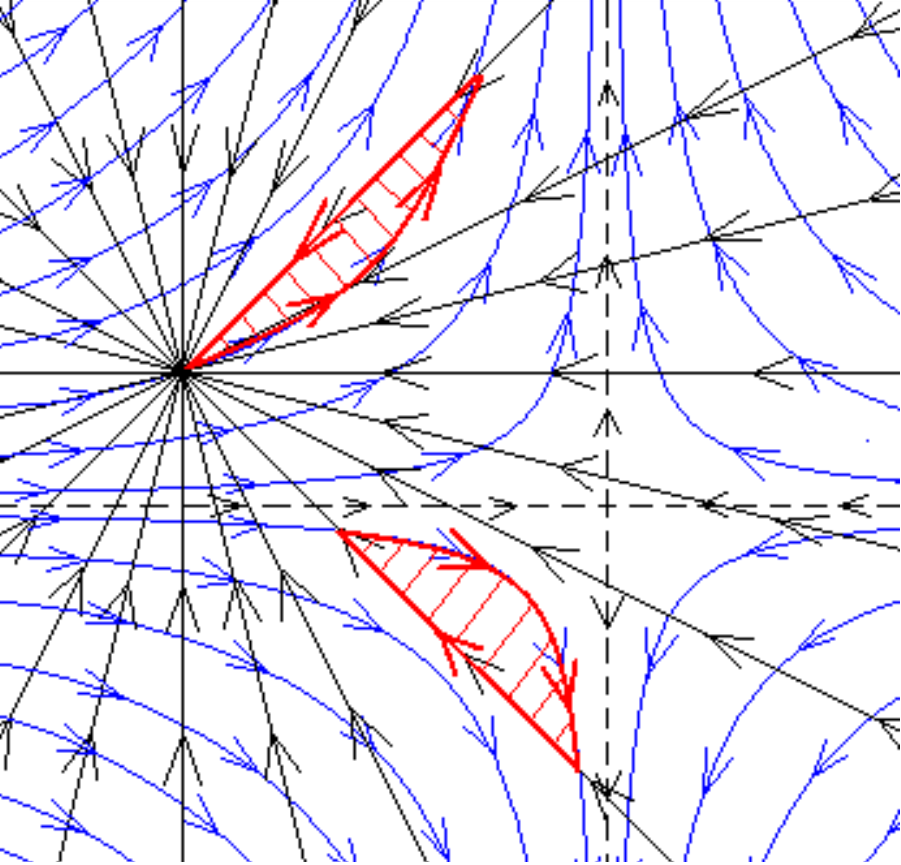}
\end{minipage}
\begin{minipage}[c]{225pt}
\hspace{0.6cm}  
 \includegraphics[width=180pt]{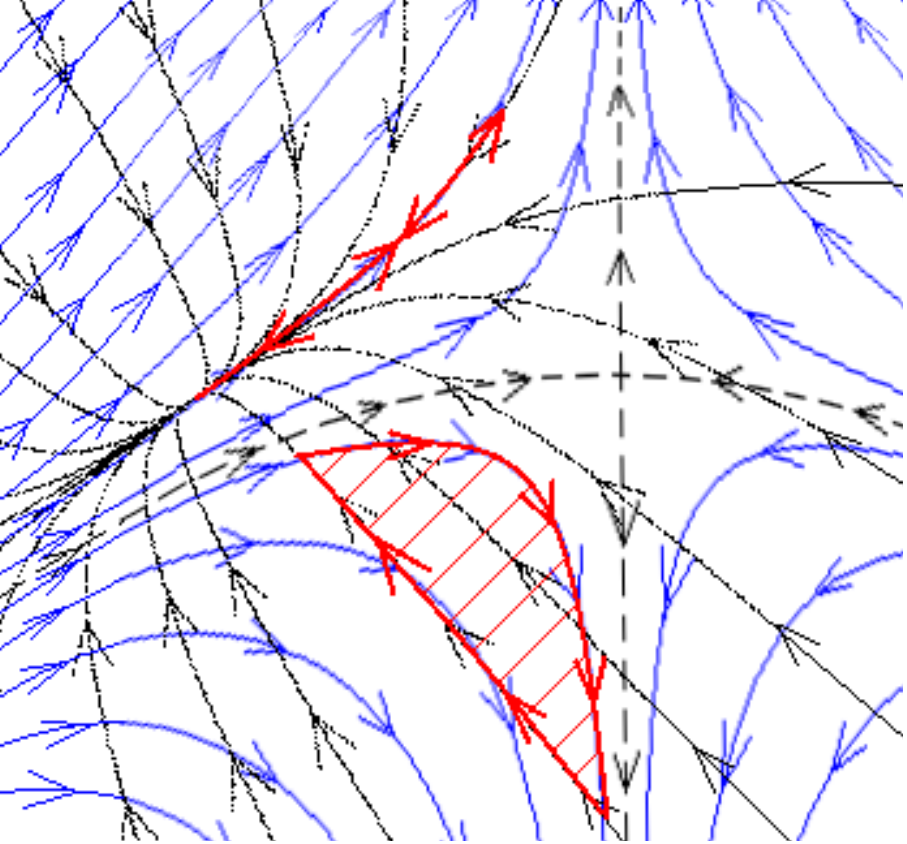}
\end{minipage}  
\begin{minipage}[c]{225pt}
 \caption{Chaotic sets with non-empty interior between stable node and unstable saddle fulfilling conditions $g(x^*) \neq \alpha e_1$ and $g(x^*) \neq \beta e_2$}
  \label{fig:stable_node_unstable_saddle_1}
\end{minipage}
\begin{minipage}[c]{225pt}
 \caption{Chaotic sets with non-empty interior and homeomorphic to $[0,1]$ with $\|f(x)\| \|g(x)\|>0$ and $\theta=\pi$ between stable node and unstable saddle fulfilling conditions $g(x^*) \neq \alpha e_1$ and $g(x^*) \neq \beta e_2$}
  \label{fig:stable_node_unstable_saddle_2}
\end{minipage}
\\
\vspace{0.5cm}
\begin{minipage}[c]{149pt}
  \includegraphics[width=130pt]{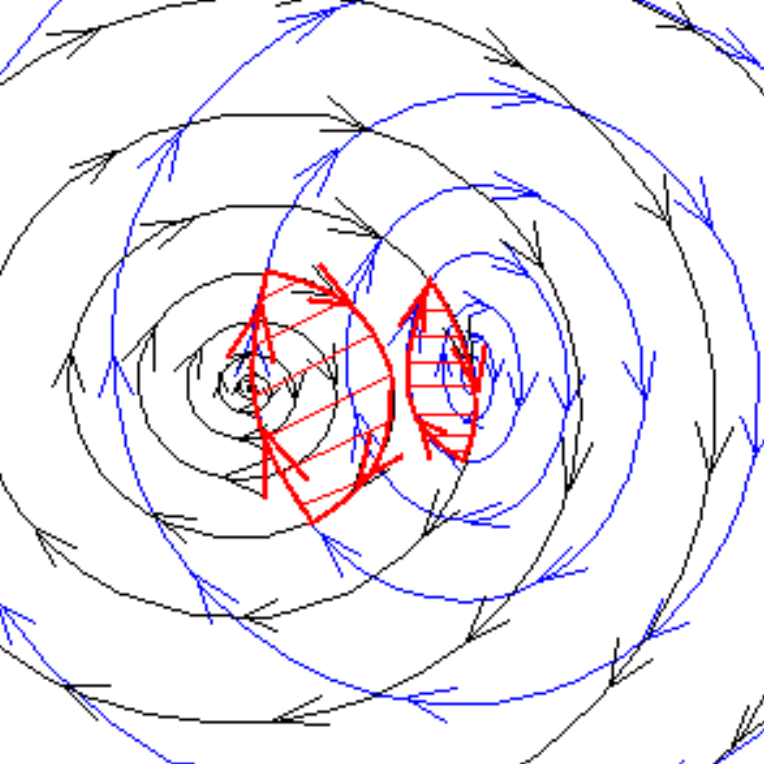}
\end{minipage}
\begin{minipage}[c]{149pt}
 \includegraphics[width=132pt]{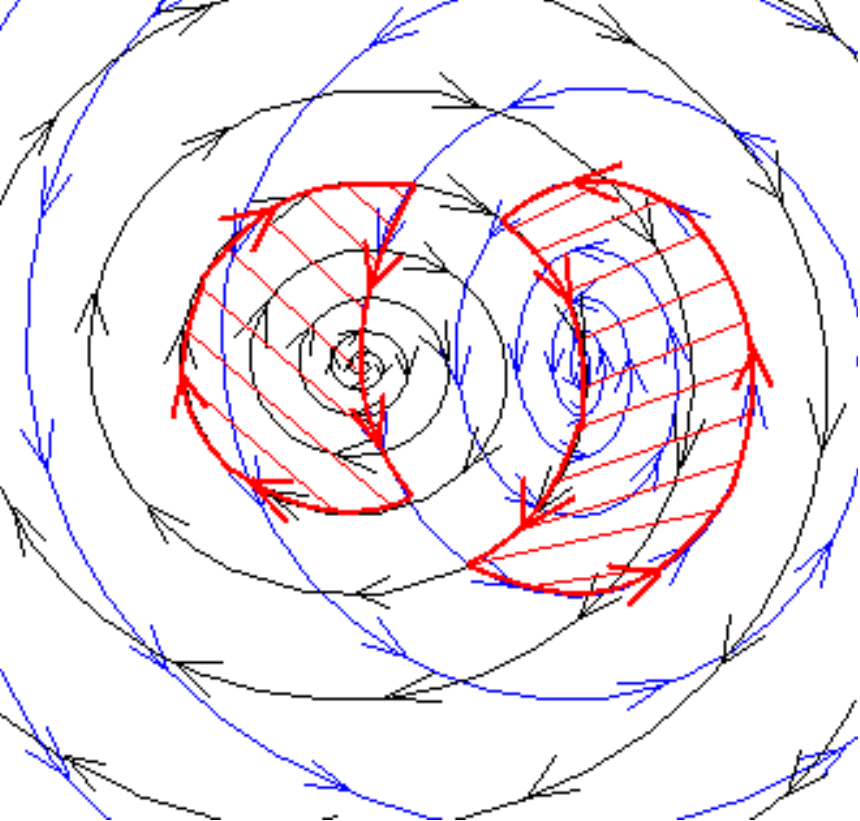}
\end{minipage}
\begin{minipage}[c]{149pt}
 \includegraphics[width=137pt]{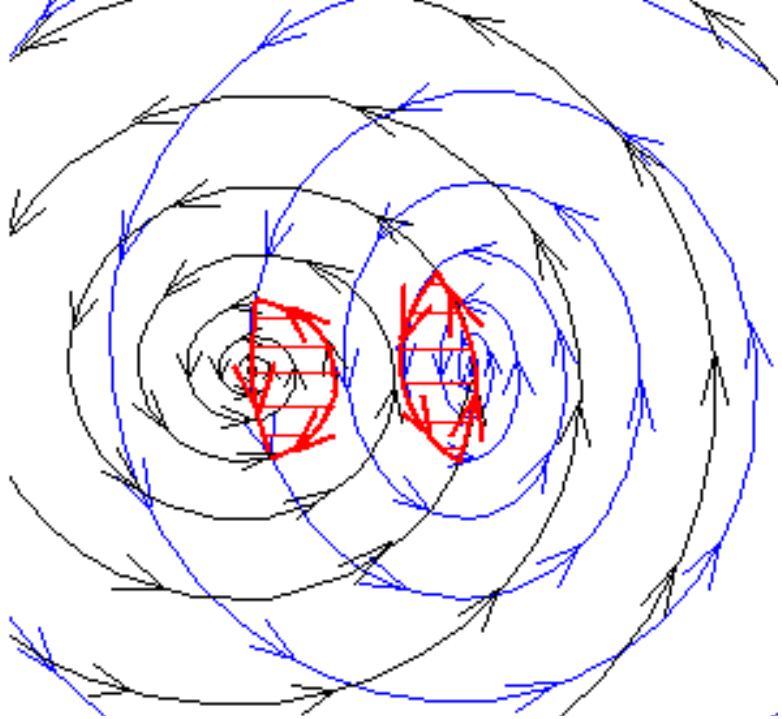}
\end{minipage}

\begin{minipage}[c]{149pt}
  \caption{Chaotic sets with non-empty interior between two unstable foci}
  \label{fig:unstable_focus_unstable_focus}
\end{minipage}
\begin{minipage}[c]{149pt}
  \caption{Chaotic sets with non-empty interior between unstable and stable focus}
  \label{fig:unstable_focus_stable_focus}
\end{minipage}
\begin{minipage}[c]{149pt}
  \caption{Chaotic sets with non-empty interior between two stable foci}
  \label{fig:stable_focus_stable_focus}
\end{minipage}
\end{figure}

\begin{figure}[htp]
\centering
\begin{minipage}[c]{225pt}
\hspace{0.5cm}  
  \includegraphics[width=190pt]{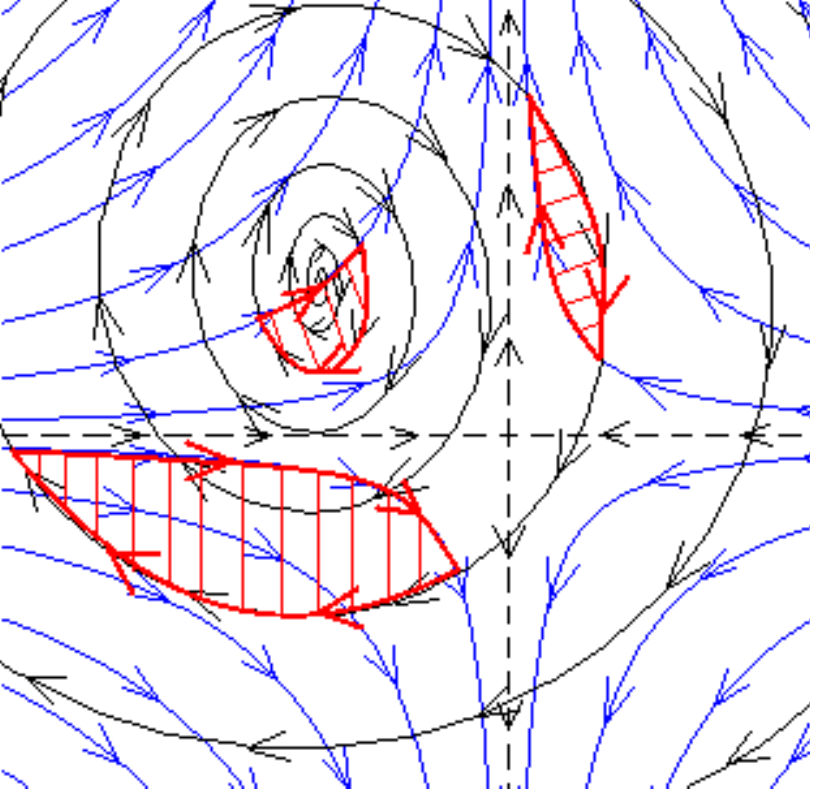}
\end{minipage}
\begin{minipage}[c]{225pt}
\hspace{0.4cm}  
 \includegraphics[width=190pt]{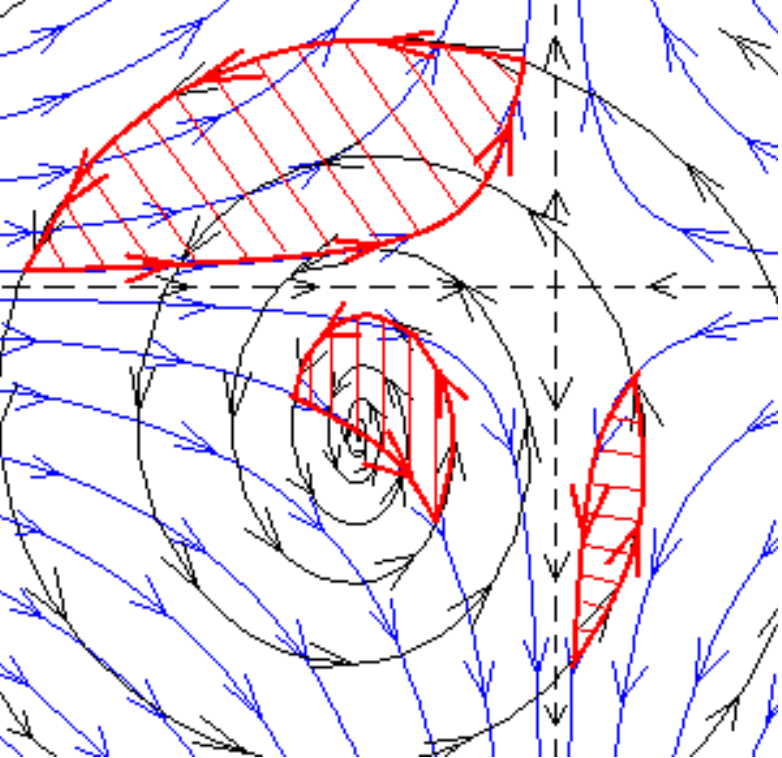}
\end{minipage}  
\begin{minipage}[c]{225pt} 
  \caption{Chaotic sets with non-empty interior between unstable focus and unstable saddle fulfilling conditions $g(x^*) \neq \alpha e_1$ and $g(x^*) \neq \beta e_2$}
  \label{fig:unstable_focus_unstable_saddle_1}
\end{minipage}
\begin{minipage}[c]{225pt} 
  \caption{Chaotic sets with non-empty interior between stable focus and unstable saddle fulfilling conditions $g(x^*) \neq \alpha e_1$ and $g(x^*) \neq \beta e_2$}
  \label{fig:stable_focus_unstable_saddle_1}
\end{minipage}

\vspace{1cm}
\begin{minipage}[c]{225pt}
\hspace{0.5cm}  
  \includegraphics[width=190pt]{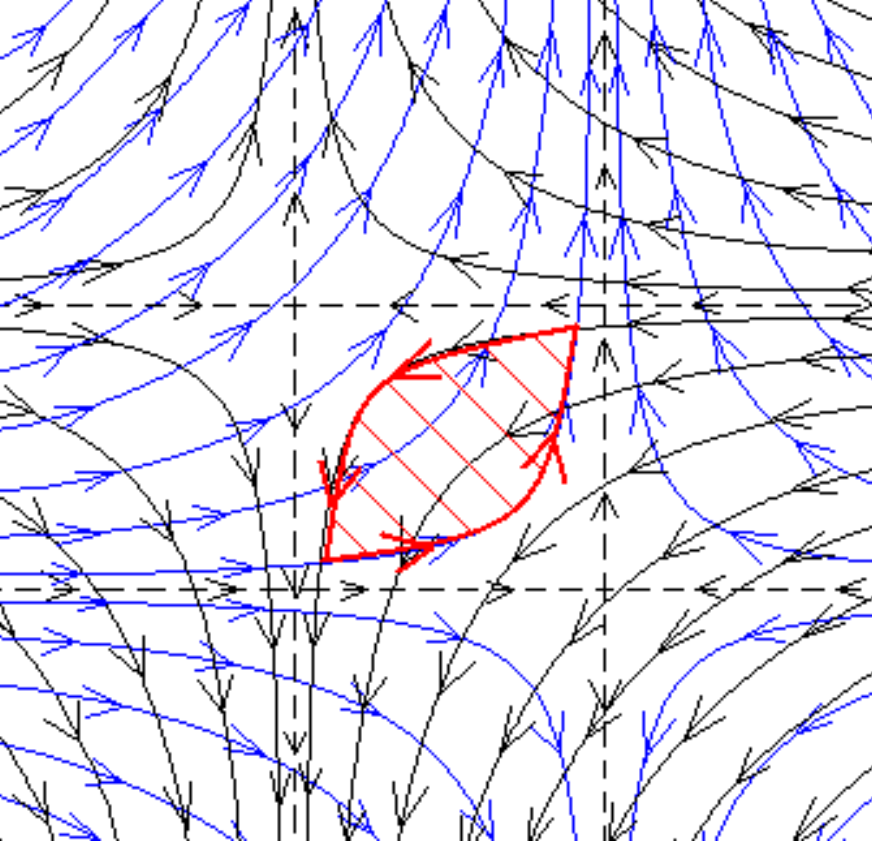}
\end{minipage}
\begin{minipage}[c]{225pt}
\hspace{0.4cm}  
  \includegraphics[width=190pt]{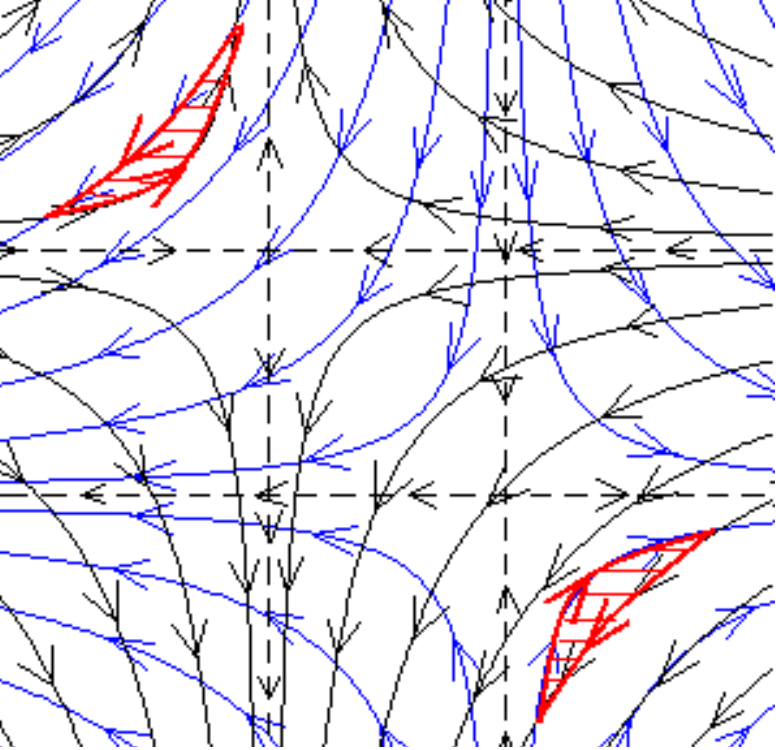}
\end{minipage}  
  \caption{Chaotic sets with non-empty interior between two unstable saddles fulfilling conditions $g(x^*) \neq \alpha e_1$ and $g(x^*) \neq \beta e_2$}
  \label{fig:unstable_saddle_unstable_saddle_1_2}
\end{figure}

\clearpage

Now, we research the possibilities which are not considered above and on which Theorem \nolinebreak \ref{thm:hyperb_sing_point-chaotic_set} can not be applicable. These possibilities are cases when the trajectory of unbounded solution passing through the saddle point $x^*$ has the same or the directly opposite direction as the stable or unstable manifold of the saddle in the point $x^*$, means the vector $g(x^*)$ is collinear with the eigenvectors $e_1$ or $e_2$. We show that every such possibilities can produce chaotic sets.

Let remind notations: $x^* \in X \subset \mathbb{R}^2$,~$f(x^*)=0$ and $g(x^*) \neq 0$, $\lambda_1,~\lambda_2$ are eigenvalues of Jacobi's matrix of the system $\dot x = f(x)$ in the point $x^*$ and $e_1,~e_2$ are corresponding eigenvectors. We choose $\delta > 0$ such that the solution of $\dot x = g(x)$ is unbounded in $\bar{B}_{\delta}(x^*)$ and $g(x) \neq 0$ for every $x \in \bar B_{\delta}(x^*)$.

\begin{thm}
\label{thm:saddle_collinear-chaotic_set}
Let $\lambda_1<0,~\lambda_2>0$ (i.e. $x^*$ is saddle point) and $g(x^*) = \alpha e_1$ or $g(x^*) = \beta e_2$, where $\alpha,\beta \in \mathbb{R} \setminus \{0\}$. Then $F$ admits a chaotic set. 
\end{thm}
\begin{proof}
Let $\varphi^t(x)$ denote flow generated by $f$ and $\psi^t(x)$ denote flow generated by $g$. Let $W^s$ denote stable manifold and $W^u$ denote unstable manifold corresponding to $f$. Because $f$ and $g$ are continuous, the solution corresponding to $g$ in $\bar{B}_{\delta}(x^*)$ is unbounded and $g(x) \neq 0$ for every $x \in \bar B_{\delta}(x^*)$, then for sufficiently small $0<\varepsilon<\delta$ we can distinguish two possible cases:
\begin{enumerate}
\item \label{proof:g_neq_mu.f} $g(x) \neq \mu f(x)$ for every $x \in \{ W^s \cap \bar{B}_{\varepsilon}(x^*) \} \setminus \{x^*\}$ or $x \in \{ W^u \cap \bar{B}_{\varepsilon}(x^*) \} \setminus \{x^*\}$ for any $\mu \in \mathbb{R} \setminus \{0\}$;
\item \label{proof:g_eq_mu.f} $g(x) = \mu f(x)$ for every $x \in \{ W^s \cap \bar{B}_{\varepsilon}(x^*) \} \setminus \{x^*\}$ or $x \in \{ W^u \cap \bar{B}_{\varepsilon}(x^*) \} \setminus \{x^*\}$ for some $\mu \in \mathbb{R} \setminus \{0\}$.
\end{enumerate}
$\mu$ can not be zero. Zero $\mu$ would lead to $g(x^*)=0$, but we assume the opposite ($g(x) \neq 0$ for every $x \in \bar B_{\delta}(x^*)$). \\
Ad (\ref{proof:g_neq_mu.f}) Assume $x^*$ is unstable saddle point, $g(x^*) = \alpha e_1$ or $g(x^*) = \beta e_2$ where $\alpha,\beta \in \mathbb{R} \setminus \{0\}$, and $g(x) \neq \mu f(x)$ for every $x \in \{ W^s \cap \bar{B}_{\varepsilon}(x^*) \} \setminus \{x^*\}$ or $x \in \{ W^u \cap \bar{B}_{\delta}(x^*) \} \setminus \{x^*\}$ for any $\mu \in \mathbb{R} \setminus \{0\}$. The stable manifold $W^s$ and the unstable manifold $W^u$ corresponding to $f$ divide the ball $B_{\varepsilon}(x^*)$ in four quadrants (I, II, III, IV). We consider $g(x^*)=\beta e_2$ (corresponding to $W^u$). The condition $g(x^*)=\beta e_2$ and $g(x) \neq \mu f(x)$ for every $x \in \{ W^u \cap \bar{B}_{\varepsilon}(x^*) \} \setminus \{x^*\}$ for any $\mu \in \mathbb{R} \setminus \{0\}$ imply that for sufficiently small neighbourhood of the point $x^*$ the flow $\psi^t(x^*)$ is in the quadrant $j$ for $t<0$ (or $t>0$) sufficiently close to $0$ and in the quadrant $j+1$ mod $4$ for $t>0$ (or $t<0$) sufficiently close to $0$ for some $j \in \{ 1,2,3,4 \}$, see the Figure \ref{fig:proof_admitted_chaos_theorem_1}. It depends on notations of quadrants and on positions of $W^s$ and $W^u$.
\begin{figure}[ht]
  \centering
  \includegraphics[height=6cm]{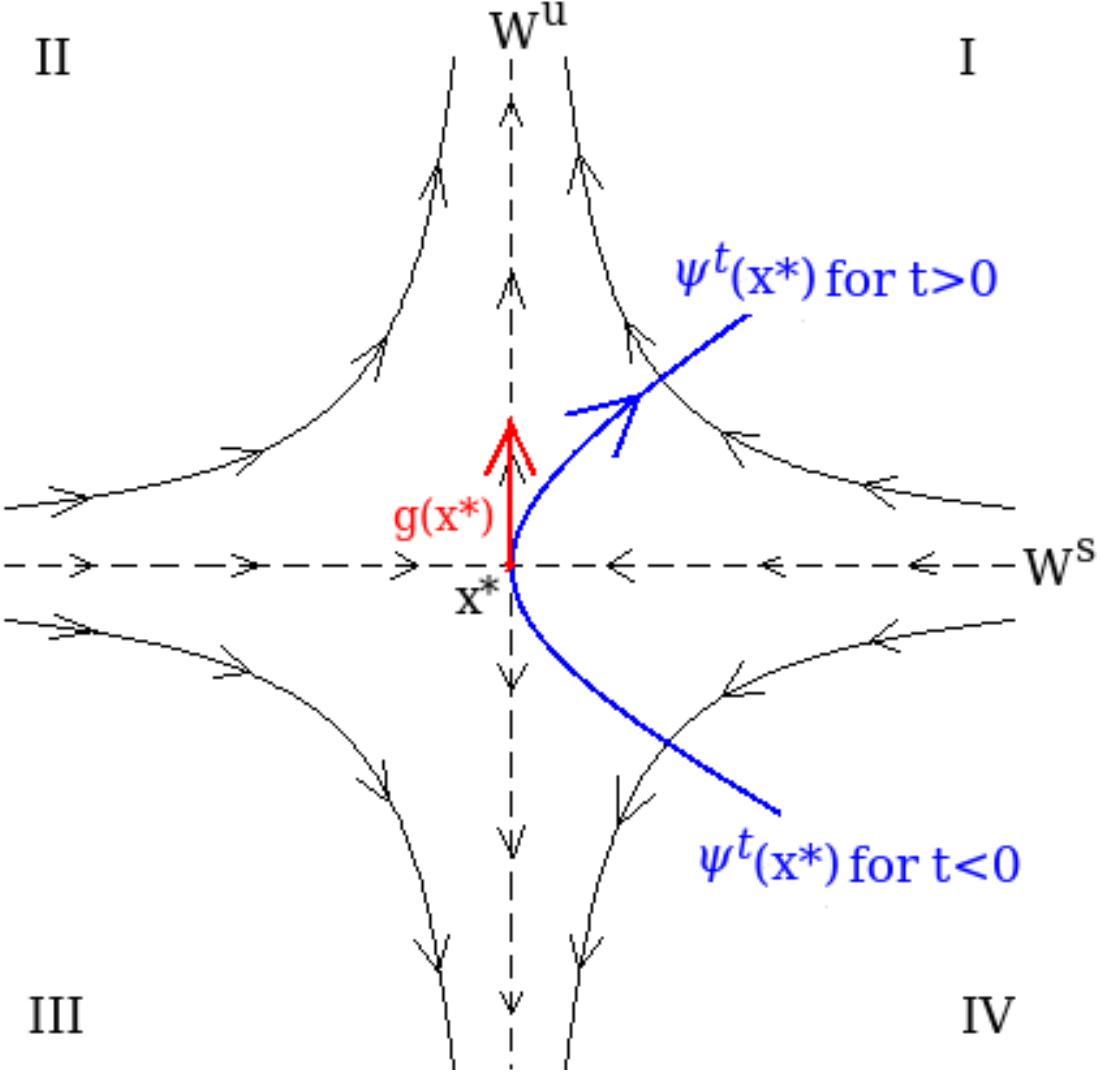}
  \caption{Proof of Theorem \ref{thm:saddle_collinear-chaotic_set} - the first scheme of part (\ref{proof:g_neq_mu.f})}
  \label{fig:proof_admitted_chaos_theorem_1}
\end{figure}
Without loss of generality we consider that $\psi^t(x^*)$ for $t<0$ is in the quadrant IV and $\psi^t(x^*)$ for $t>0$ is in the quadrant I, see Figure \ref{fig:proof_admitted_chaos_theorem_1}. Thus, let $T_1<0$ be maximal and $T_2>0$ be minimal such that $\psi^{T_1}(x^*),\psi^{T_2}(x^*) \in \partial B_{\varepsilon}(x^*)$, see Figure \ref{fig:proof_admitted_chaos_theorem_2}. Then let $\varepsilon'<\frac{\varepsilon}{2}$ and $t_1<0$ be maximal and $t_2>0$ be minimal ($T_1<t_1<0<t_2<T_2$) such that $B_{\frac{\varepsilon'}{2}}(\psi^{t_1}(x^*)) \cap B_{\varepsilon'}(x^*)\cap W^u=\emptyset$ and $B_{\frac{\varepsilon'}{2}}(\psi^{t_2}(x^*)) \cap B_{\varepsilon'}(x^*)\cap W^u=\emptyset$, see Figure \ref{fig:proof_admitted_chaos_theorem_2}.
\begin{figure}[ht]
  \centering
  \includegraphics[height=6.5cm]{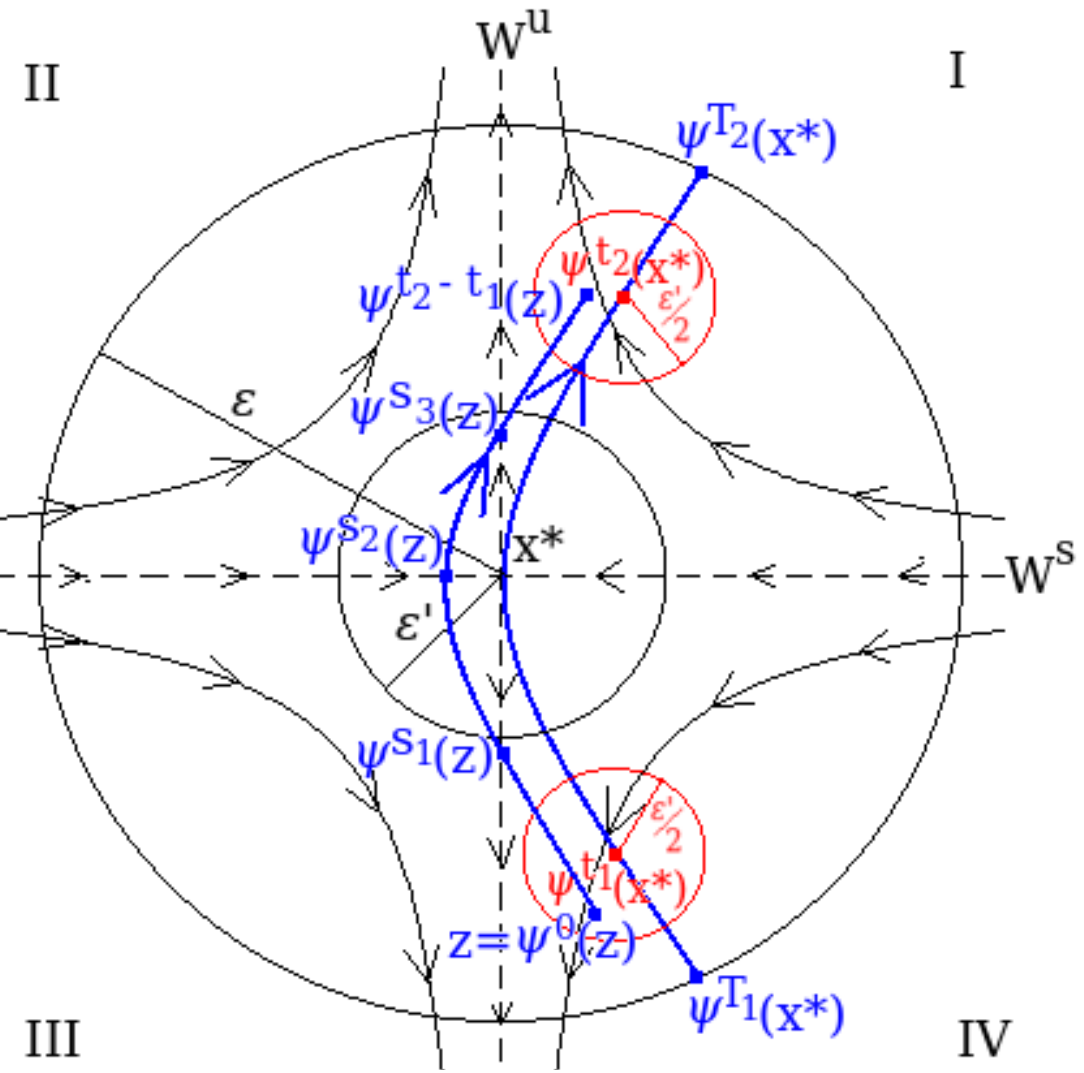}
  \caption{Proof of Theorem \ref{thm:saddle_collinear-chaotic_set} - the second scheme of part (\ref{proof:g_neq_mu.f})}
  \label{fig:proof_admitted_chaos_theorem_2}
\end{figure}
Thus, there exists $z \in B_{\frac{\varepsilon'}{2}}(\psi^{t_1}(x^*))$ such that $d(\psi^t(z),\psi^{t_1+t})(x^*))<\frac{\varepsilon'}{2}$ for all $0 \leq t\leq t_2-t_1$ and such that $\psi^{s_1}(z) \in W^u$, $\psi^{s_2}(z) \in W^s$ and $\psi^{s_3}(z) \in W^u$ for some $0<s_1<s_2<s_3<t_2-t_1$, see the Figure \ref{fig:proof_admitted_chaos_theorem_2}. So now, we are interested only in quadrant III. If we have the opposite direction of $\psi^t(x^*)$, then we will be interested in the quadrant II, see the Figure \ref{fig:proof_admitted_chaos_theorem_2}. Then, firstly we choose $\rho>0, t_0>0$ such that $B_{\rho}(\varphi^{t_0}(\psi^{s_2}(z))) \cap \{ \psi^s(z): s_1 \leq s \leq s_2 \}=\emptyset$, see Figure \ref{fig:proof_admitted_chaos_theorem_3}. Secondly, we choose $r>0$ such that if $w \in B_r(\psi^{s_2}(z))$ then $d(\varphi^t(w),\varphi^t(\psi^{s_2}(z)))<\rho$ for every $0 \leq t \leq t_0$, see Figure \ref{fig:proof_admitted_chaos_theorem_3}. At the end we pick $w \in \{ \psi^s(z): s_1 \leq s < s_2 \} \cap B_r(\psi^{s_2}(z))$. So, $w \notin W^s \cup W^u$ and thus $\varphi^t(w) \notin W^s \cup W^u$ for all $t \geq 0$. Hence there exists $T>0$ such that $\varphi^T(w) \in \{ \psi^s(z): s_1 \leq s \leq s_2 \}$, see Figure \ref{fig:proof_admitted_chaos_theorem_3}.
\begin{figure}[ht]
  \centering
  \includegraphics[height=6.5cm]{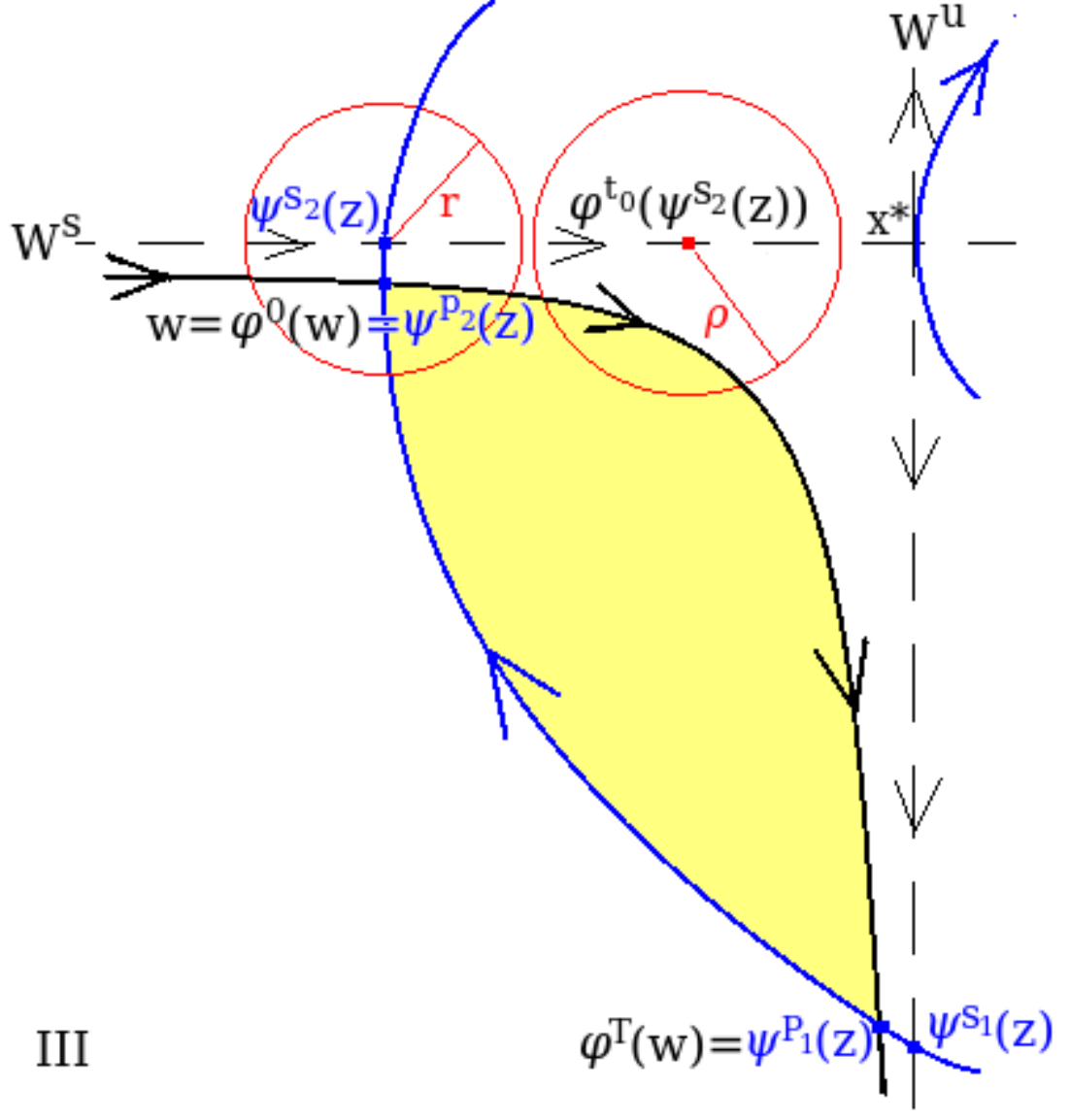}
  \caption{Proof of Theorem \ref{thm:saddle_collinear-chaotic_set} - the third scheme of part (\ref{proof:g_neq_mu.f})}
  \label{fig:proof_admitted_chaos_theorem_3}
\end{figure}
Hence there exist a simple path from $w$ to $w$ which consists of two arcs $\{ \varphi^t(w):0 \leq t \leq T \}$ and $\{\psi^s(z):p_1 \leq s \leq p_2 \}$, where $s_1<p_1<p_2<s_2$, $w=\varphi^0(w)=\psi^{p_2}(z)$ and $\varphi^T(w)=\psi^{p_1}(z)$. Then according Lemma \ref{lemma_chaotic_set_in_R2} $F$ admits chaotic set $V$ (i.e. yellow area on Figure \ref{fig:proof_admitted_chaos_theorem_3}). Like this constructed chaotic set $V$ has obviously non-empty interior (has non-empty bounded component, see Lemma \ref{lemma_chaotic_set_in_R2}). The proof for $g(x^*) = \alpha e_1$ is analogous.\\
Ad (\ref{proof:g_eq_mu.f}) Assume $x^*$ is unstable saddle point, $g(x^*) = \alpha e_1$ or $g(x^*) = \beta e_2$, where $\alpha,\beta \in \mathbb{R} \setminus \{0\}$, and $g(x) = \mu f(x)$ for every $x \in \{ W^s \cap \bar{B}_{\varepsilon}(x^*) \} \setminus \{x^*\}$ or $x \in \{ W^u \cap \bar{B}_{\varepsilon}(x^*) \} \setminus \{x^*\}$ for some $\mu \in \mathbb{R} \setminus \{0\}$ ($\mu$ can be different for each $x$). We consider $g(x^*)=\alpha e_1$ (corresponding to $W^s$). The condition $g(x^*)=\alpha e_1$ and $g(x) = \mu f(x)$ for every $x \in \{ W^s \cap \bar{B}_{\varepsilon}(x^*) \} \setminus \{x^*\}$ for some $\mu \in \mathbb{R} \setminus \{0\}$ imply that $\psi^t(x^*)$ overlaps $W^s$ in $\bar{B}_{\varepsilon}(x^*)$, see Figure \ref{fig:proof_admitted_chaos_theorem_4}.
\begin{figure}[ht]
  \centering
  \includegraphics[height=6.5cm]{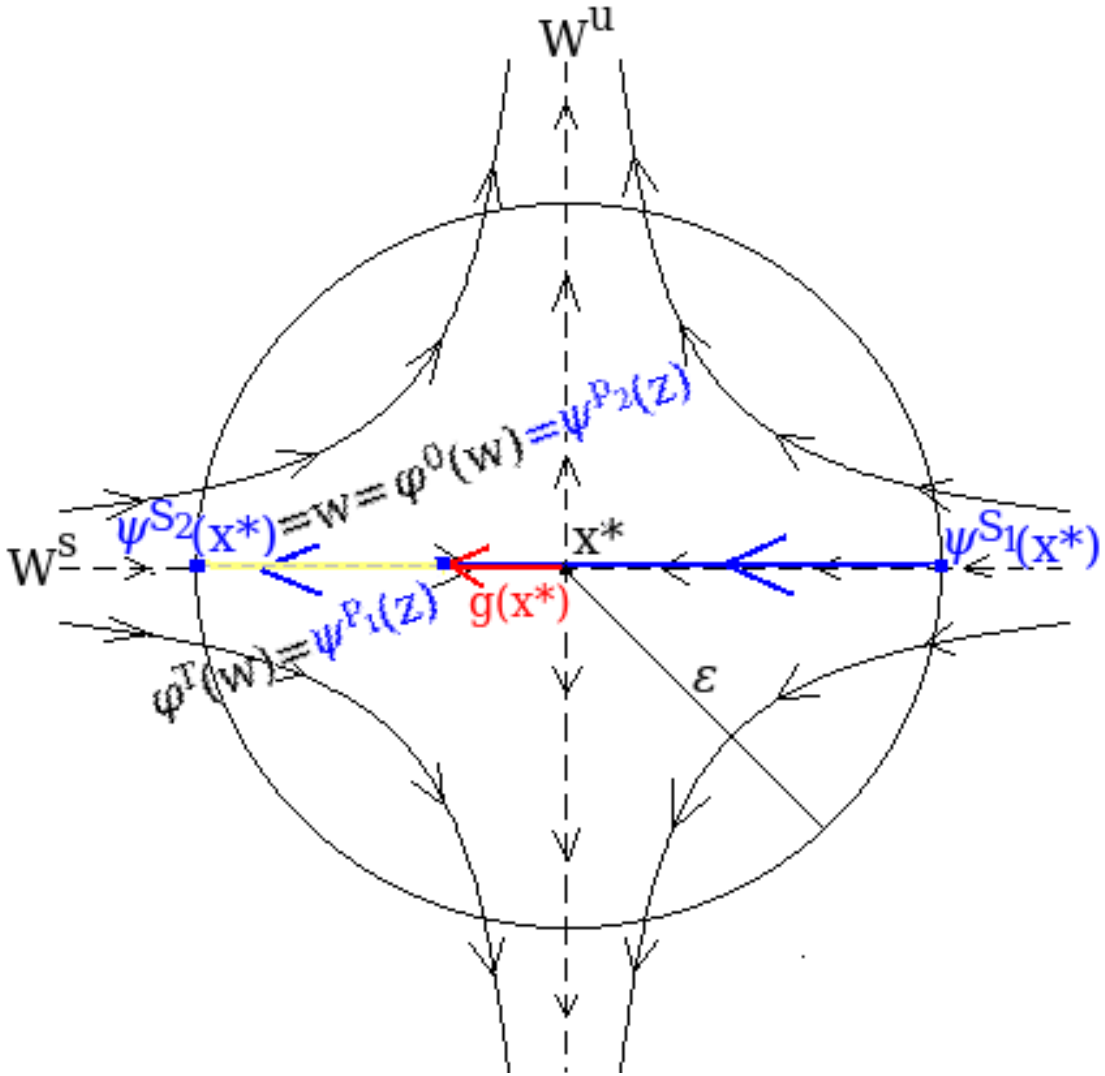}
  \caption{Proof of Theorem \ref{thm:saddle_collinear-chaotic_set} - the scheme of part (\ref{proof:g_eq_mu.f})}
  \label{fig:proof_admitted_chaos_theorem_4}
\end{figure}
Thus, let $S_1<0$ be maximal and $S_2>0$ be minimal such that $\psi^{S_1}(x^*),\psi^{S_2}(x^*) \in \partial B_{\varepsilon}(x^*) \cap W^s$, see Figure \ref{fig:proof_admitted_chaos_theorem_4}. Let $w:= \psi^{S_2}(x^*)$. Thus there exist $T>0$ such that $\varphi^T(w) \in \{ \psi^s(x^*): 0 < s < S_2 \}$, see Figure \ref{fig:proof_admitted_chaos_theorem_4}. Hence there exist a simple path from $w$ to $w$ which consists of two arcs $\{ \varphi^t(w):0 \leq t \leq T \}$ and $\{\psi^s(x^*):p_1 \leq s \leq p_2 \}$, where $0<p_1<p_2=S_2$, $w=\varphi^0(w)=\psi^{p_2}(x^*)$ and $\varphi^T(w)=\psi^{p_1}(x^*)$. Then according Lemma \ref{lemma_chaotic_set_in_R2} $F$ admits chaotic set $V$ (i.e. yellow line segment on Figure \ref{fig:proof_admitted_chaos_theorem_4}). Like this constructed chaotic set $V$ is obviously homeomorphic to $[0,1]$ with $\|f(x)\| \|g(x)\|>0$ and $\theta:=cos^{-1}\left( \frac{f(x) \cdot g(x)}{\|f(x)\| \|g(x)\|} \right)=\pi$ for all $x \in V$. The proof for $g(x^*) = \beta e_2$ is analogous with reverse time direction and with $w:=\psi^{T_1}(x^*)$.
\end{proof}

\begin{rmk}
\label{rmk:to_thm-saddle_collinear-chaotic_set}
Even though we construct only a chaotic set homeomorphic to $[0,1]$ with $\|f(x)\| \|g(x)\|>0$ and $\theta:=cos^{-1}\left( \frac{f(x) \cdot g(x)}{\|f(x)\| \|g(x)\|} \right)=\pi$ for all $x \in V$ in the part (\ref{proof:g_eq_mu.f}) of the proof of Theorem \ref{thm:saddle_collinear-chaotic_set}, there can exist also a chaotic sets with non-empty interior in the cases describing below. Let further $y^* \in X \subset \mathbb{R}^2$, $g(y^*)=0$, $\tilde{\lambda}_1,~\tilde{\lambda}_2$ eigenvalues of the Jacobi's matrix of $g$ in the $y^*$ with corresponding eigenvectors $\tilde{e}_1,~\tilde{e}_2$. \\
For node type of the second singular point $y^*$ the chaotic set with non-empty interior exists if $\psi^t(x^*)$ overlaps the whole $W^s$ or $W^u$ since the point $y^*$. We show this for $\tilde \lambda_1<0,~\tilde \lambda_2<0$. Hence $g(x^*) = \alpha e_1$ and $y^* \in W^s$, see Figure \ref{fig:remark_to_proof_admitted_chaos_theorem_1}.
\begin{figure}[ht]
  \centering
  \includegraphics[height=6.5cm]{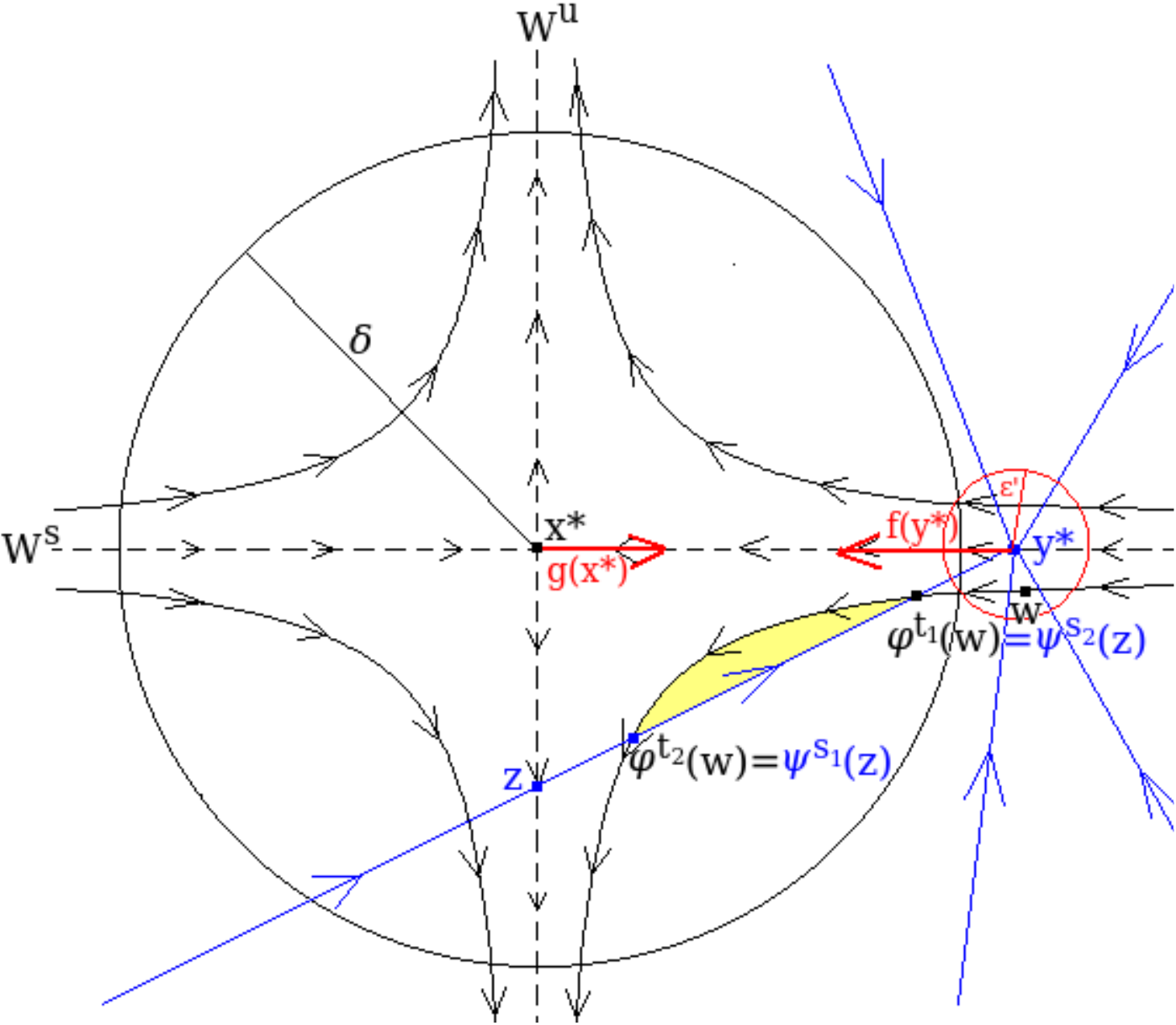}
  \caption{Construction scheme of chaotic set with non-empty interior in part (\ref{proof:g_eq_mu.f}) of proof of Theorem \ref{thm:saddle_collinear-chaotic_set} for $y^*$ node}
  \label{fig:remark_to_proof_admitted_chaos_theorem_1}
\end{figure}
Then we choose $0<\varepsilon'<<\delta$. Thus for sufficiently large $B_{\delta}(x^*)$ (means sufficiently large $\delta>0$) there exist $\{ \varphi^{t}(w):t_1 \leq t \leq t_2 \} \in B_{\delta}(x^*)$ for some $0<t_1<t_2$, $w \in B_{\varepsilon'}(y^*) \setminus (B_{\delta}(x^*) \cup W^s)$ and $\{ \psi^{s}(z):s_1 \leq s \leq s_2 \} \in B_{\delta}(x^*)$ for some $0<s_1<s_2$, $z \in (W^u \cap B_{\delta}(x^*))\setminus \{x^*\}$ such that $\varphi^{t_1}(w)=\psi^{s_2}(z)$ and $\varphi^{t_2}(w)=\psi^{s_1}(z)$, see Figure \ref{fig:remark_to_proof_admitted_chaos_theorem_1}. Hence $F$ admits the chaotic set $V$ (i.e. yellow area in Figure \ref{fig:remark_to_proof_admitted_chaos_theorem_1}) consisting of two arcs $\{ \varphi^{t}(w):t_1 \leq t \leq t_2 \}$ and $\{ \psi^{s}(z):s_1 \leq s \leq s_2 \}$, where $\varphi^{t_1}(w)=\psi^{s_2}(z)$ and $\varphi^{t_2}(w)=\psi^{s_1}(z)$, and its (non-empty) interior. The construction for $\tilde \lambda_1>0,~\tilde \lambda_2>0$ is analogous.\\
For saddle type of the second singular point $y^*$ the chaotic set with non-empty interior exists if stable manifold $W^s$ corresponding to the saddle $x^*$ overlaps the whole unstable manifold $\widetilde W^u$ corresponding to the saddle $y^*$ since the point $y^*$ and the unstable manifold $W^u$ corresponding to $f$ intersects the stable manifold $\widetilde W^s$ corresponding to $g$. And vice versa. We consider the first possibility. The construction for the second possibility is analogous. Hence $g(x^*) = \alpha e_1$ and $y^* \in W^s(=\widetilde W^u)$. Let the intersection of $W^u$ and $\widetilde W^s$ denote $z$, see Figure \ref{fig:remark_to_proof_admitted_chaos_theorem_2}. Let $\triangle x^*y^*z$ denote the "triangle" given by the points $x^*$, $y^*$ and $z$.
\begin{figure}[ht]
  \centering
  \includegraphics[height=7cm]{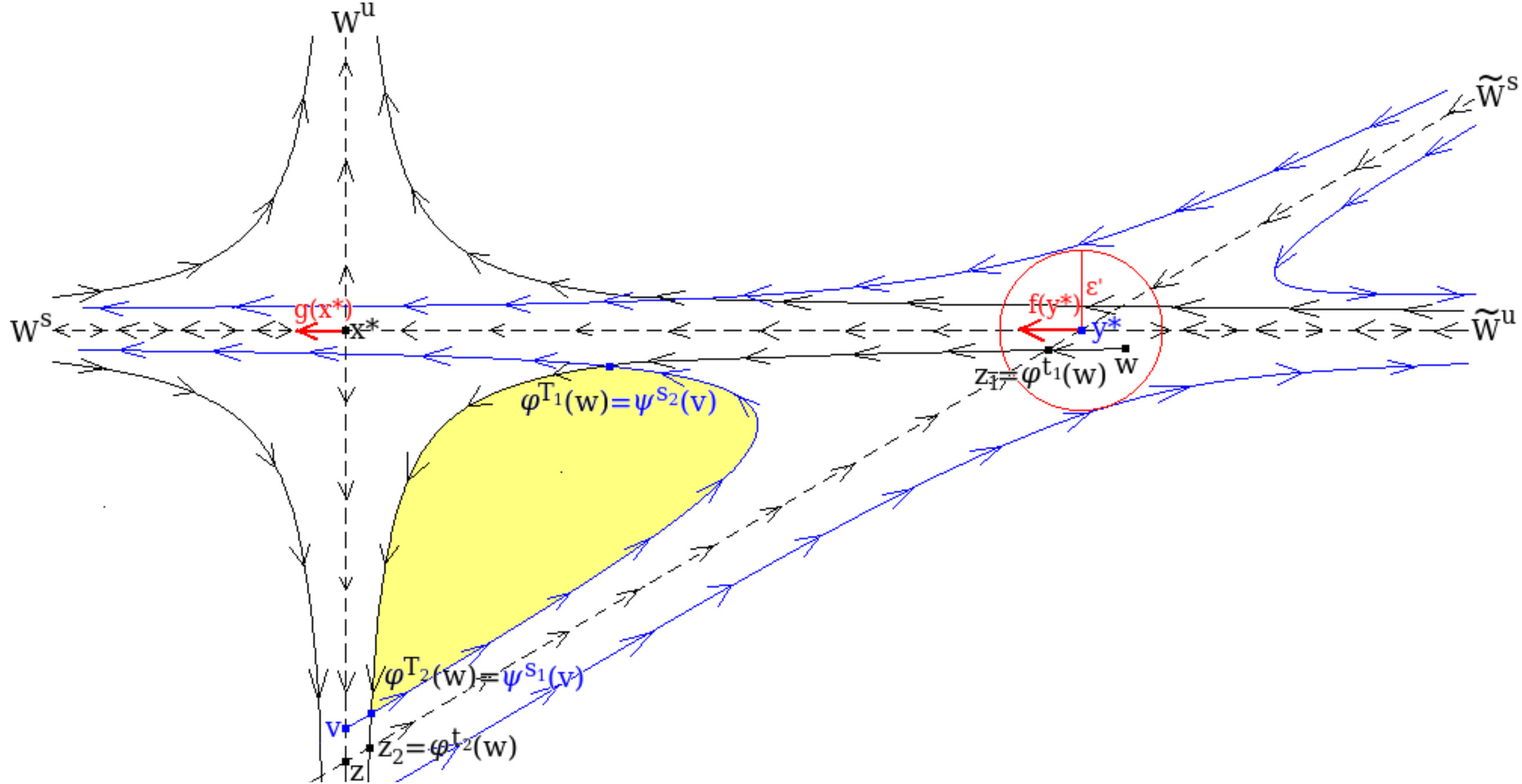}
  \caption{Construction scheme of chaotic set with non-empty interior in part (\ref{proof:g_eq_mu.f}) of proof of Theorem \ref{thm:saddle_collinear-chaotic_set} for $y^*$ saddle}
  \label{fig:remark_to_proof_admitted_chaos_theorem_2}
\end{figure}
Then we choose $0<\varepsilon'<<d(x^*,y^*)$. Thus for sufficiently large $\delta>0$ there exist $z_1,z_2 \in \widetilde W^s$ such that there exist $\varphi^{t_1}(w)=z_1$ and $\varphi^{t_2}(w)=z_2$ for some $0<t_1<t_2$ and some $w \in B_{\varepsilon'}(y^*) \setminus (\triangle x^*y^*z \cup \widetilde W^s \cup \widetilde W^u)$, see Figure \ref{fig:remark_to_proof_admitted_chaos_theorem_2}. Then there certainly exists $\{\psi^s(v):s>0\}$ for some $v \in W^s$ (corresponding to $f$) such that $\psi^{s_1}(v),\psi^{s_2}(v) \in \{\varphi^t(w):t_1<t<t_2\}$ for some $0<s_1<s_2$ and $\psi^{s_1}(v)=\varphi^{T_2}(w),~\psi^{s_2}(v)=\varphi^{T_1}(w)$ for some $t_1<T_1<T_2<t_2$, see Figure \ref{fig:remark_to_proof_admitted_chaos_theorem_2}. Thus $F$ admits the chaotic set (i.e. yellow area on Figure \ref{fig:remark_to_proof_admitted_chaos_theorem_2}) consisting of two arcs $\{ \varphi^{t}(w):T_1 \leq t \leq T_2 \}$ and $\{ \psi^{s}(v):s_1 \leq s \leq s_2 \}$, where $\varphi^{T_1}(w)=\psi^{s_2}(v)$ and $\varphi^{T_2}(w)=\psi^{s_1}(v)$, and its (non-empty) interior. \\
If type of the second singular point $y^*$ is focus, then there always exists chaotic set with non-empty interior. The focus type of singular point naturally ensures such chaotic set, see yellow area on Figure \ref{fig:remark_to_proof_admitted_chaos_theorem_3}.
\begin{figure}[ht]
  \centering
  \includegraphics[height=5.5cm]{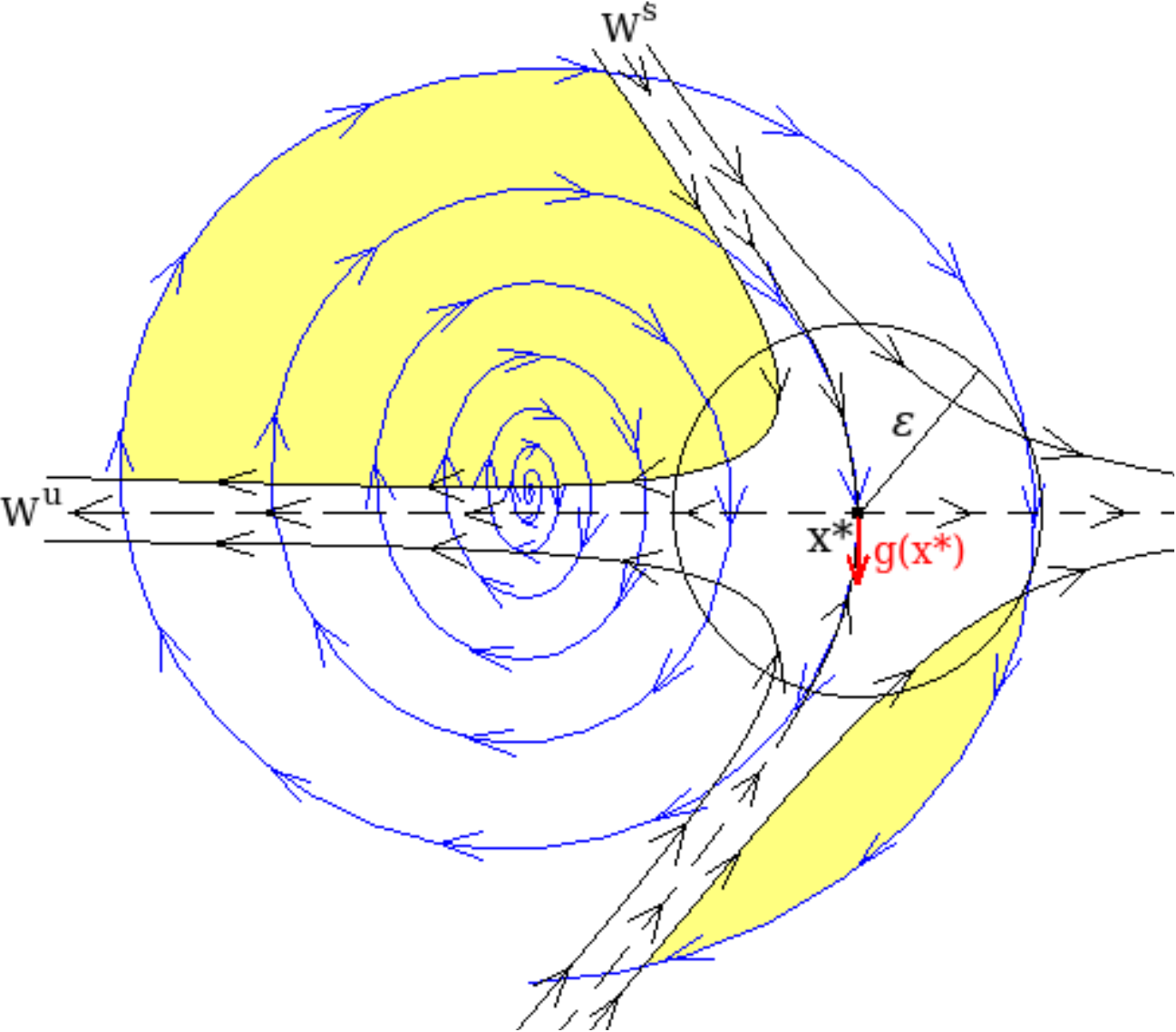}
  \caption{Construction scheme of chaotic set with non-empty interior in part (\ref{proof:g_eq_mu.f}) of proof of Theorem \ref{thm:saddle_collinear-chaotic_set} for $y^*$ focus}
  \label{fig:remark_to_proof_admitted_chaos_theorem_3}
\end{figure}
\\ The existence of chaotic set with non-empty interior in the case describing by the part (\ref{proof:g_eq_mu.f}) of the proof of Theorem \ref{thm:saddle_collinear-chaotic_set} is provided by appropriate trajectory corresponding to $g$ which has to intersect first in time $W^u$ and then $W^s$ in non-degenerated case (see Figure \ref{fig:proof_admitted_chaos_theorem_3}) or which has to intersect the separate not overlapped manifold in two points in degenerated case (see Figure \ref{fig:remark_to_proof_admitted_chaos_theorem_2}).  
\end{rmk}

In the Table \ref{tab:overview_hyp_sing_points_saddle_collinear} we illustrate all remaining combinations of hyperbolic singular points where $g(x^*) = \alpha e_1$ or $g(x^*) = \beta e_2$. There are references to the part of the proof of previous Theorem \ref{thm:saddle_collinear-chaotic_set} which proves the possible existence of chaotic sets, and thus Devaney, Li-Yorke and distributional chaos according to relevant theorems, and links to the corresponding figures where the chaotic sets are indicated. There is also description how the mutual positions of these pairs of hyperbolic singular points are on next illustrative Figures \ref{fig:unstable_node_unstable_saddle_3}-\ref{fig:unstable_saddle_unstable_saddle_8}.

\begin{table}[ht]
\begin{center}
\begin{tabular}{|l|l|l|l|}
\hline
Branches $\dot x = f(x)$ and $\dot x = g(x)$ & Proof & Fig.\\
\hline \hline
\textit{unstable saddle - unstable node}: &&\\
- node lies outside manifolds & Theorem \ref{thm:saddle_collinear-chaotic_set} (\ref{proof:g_neq_mu.f}), \ref{thm:chaotic_set-Devaney_chaos}, \ref{thm:chaotic_set_non_empty_interior-Li-Yorke_distributional_chaos} & \ref{fig:unstable_node_unstable_saddle_3}\\
- node lies on unstable manifold & Remark \ref{rmk:to_thm-saddle_collinear-chaotic_set}, Theorem \ref{thm:saddle_collinear-chaotic_set} (\ref{proof:g_eq_mu.f}), \ref{thm:chaotic_set-Devaney_chaos}, \ref{thm:chaotic_set_non_empty_interior-Li-Yorke_distributional_chaos}, \ref{thm:chaotic_set_empty_interior-Li-Yorke_distributional_chaos} & \ref{fig:unstable_node_unstable_saddle_4} \\
- node lies on stable manifold & Theorem \ref{thm:saddle_collinear-chaotic_set} (\ref{proof:g_eq_mu.f}), \ref{thm:chaotic_set-Devaney_chaos}, \ref{thm:chaotic_set_empty_interior-Li-Yorke_distributional_chaos} & \ref{fig:unstable_node_unstable_saddle_5} \\
\hline
\textit{unstable saddle - stable node}:  &&\\
- node lies outside manifolds & Theorem \ref{thm:saddle_collinear-chaotic_set} (\ref{proof:g_neq_mu.f}), \ref{thm:chaotic_set-Devaney_chaos}, \ref{thm:chaotic_set_non_empty_interior-Li-Yorke_distributional_chaos} & \ref{fig:stable_node_unstable_saddle_3} \\
- node lies on stable manifold & Remark \ref{rmk:to_thm-saddle_collinear-chaotic_set}, Theorem \ref{thm:saddle_collinear-chaotic_set} (\ref{proof:g_eq_mu.f}), \ref{thm:chaotic_set-Devaney_chaos}, \ref{thm:chaotic_set_non_empty_interior-Li-Yorke_distributional_chaos}, \ref{thm:chaotic_set_empty_interior-Li-Yorke_distributional_chaos} & \ref{fig:stable_node_unstable_saddle_4} \\
- node lies on unstable manifold & Theorem \ref{thm:saddle_collinear-chaotic_set} (\ref{proof:g_eq_mu.f}), \ref{thm:chaotic_set-Devaney_chaos}, \ref{thm:chaotic_set_empty_interior-Li-Yorke_distributional_chaos} & \ref{fig:stable_node_unstable_saddle_5} \\
\hline
\textit{unstable saddle - unstable focus} & Theorem \ref{thm:saddle_collinear-chaotic_set} (\ref{proof:g_neq_mu.f}), \ref{thm:chaotic_set-Devaney_chaos}, \ref{thm:chaotic_set_non_empty_interior-Li-Yorke_distributional_chaos} & \ref{fig:unstable_focus_unstable_saddle_2} \\
& Remark \ref{rmk:to_thm-saddle_collinear-chaotic_set}, Theorem \ref{thm:saddle_collinear-chaotic_set} (\ref{proof:g_eq_mu.f}), \ref{thm:chaotic_set-Devaney_chaos}, \ref{thm:chaotic_set_non_empty_interior-Li-Yorke_distributional_chaos}, \ref{thm:chaotic_set_empty_interior-Li-Yorke_distributional_chaos} & \ref{fig:unstable_focus_unstable_saddle_3} \\
\hline
\textit{unstable saddle - stable focus} & Theorem \ref{thm:saddle_collinear-chaotic_set} (\ref{proof:g_neq_mu.f}), \ref{thm:chaotic_set-Devaney_chaos}, \ref{thm:chaotic_set_non_empty_interior-Li-Yorke_distributional_chaos}& \ref{fig:stable_focus_unstable_saddle_2} \\
& Remark \ref{rmk:to_thm-saddle_collinear-chaotic_set}, Theorem \ref{thm:saddle_collinear-chaotic_set} (\ref{proof:g_eq_mu.f}), \ref{thm:chaotic_set-Devaney_chaos}, \ref{thm:chaotic_set_non_empty_interior-Li-Yorke_distributional_chaos}, \ref{thm:chaotic_set_empty_interior-Li-Yorke_distributional_chaos} & \ref{fig:stable_focus_unstable_saddle_3} \\
\hline
\textit{unstable saddle - unstable saddle}: &&\\
- manifolds are not overlapped & Theorem \ref{thm:saddle_collinear-chaotic_set} (\ref{proof:g_neq_mu.f}), \ref{thm:chaotic_set-Devaney_chaos}, \ref{thm:chaotic_set_non_empty_interior-Li-Yorke_distributional_chaos} & \ref{fig:unstable_saddle_unstable_saddle_3_4} \\
- stable manifold overlaps & Remark \ref{rmk:to_thm-saddle_collinear-chaotic_set}, Theorem \ref{thm:saddle_collinear-chaotic_set} (\ref{proof:g_eq_mu.f}), \ref{thm:chaotic_set-Devaney_chaos}, \ref{thm:chaotic_set_non_empty_interior-Li-Yorke_distributional_chaos}, \ref{thm:chaotic_set_empty_interior-Li-Yorke_distributional_chaos} & \ref{fig:unstable_saddle_unstable_saddle_5} \\
~ unstable manifold and second &&\\
~ manifolds intersect each other &&\\
- stable manifold overlaps & Theorem \ref{thm:saddle_collinear-chaotic_set} (\ref{proof:g_eq_mu.f}), \ref{thm:chaotic_set-Devaney_chaos}, \ref{thm:chaotic_set_empty_interior-Li-Yorke_distributional_chaos} & \ref{fig:unstable_saddle_unstable_saddle_6} \\
~ unstable manifold and second &&\\
~ manifolds do not intersect &&\\
- stable manifolds are overlapped & Theorem \ref{thm:saddle_collinear-chaotic_set} (\ref{proof:g_eq_mu.f}), \ref{thm:chaotic_set-Devaney_chaos}, \ref{thm:chaotic_set_empty_interior-Li-Yorke_distributional_chaos} &  \ref{fig:unstable_saddle_unstable_saddle_7} \\
- unstable manifolds are overlapped & Theorem \ref{thm:saddle_collinear-chaotic_set} (\ref{proof:g_eq_mu.f}), \ref{thm:chaotic_set-Devaney_chaos}, \ref{thm:chaotic_set_empty_interior-Li-Yorke_distributional_chaos} &  \ref{fig:unstable_saddle_unstable_saddle_8} \\
\hline
\end{tabular}
\vspace{0.3cm}
\caption{Overview of hyperbolic singular points combinations with saddle point fulfilling $g(x^*) = \alpha e_1$ or $g(x^*) = \beta e_2$}
\label{tab:overview_hyp_sing_points_saddle_collinear}
\end{center}
\end{table}

On the illustrative figures we present the both situations, where the branch $f$ corresponds to the black trajectories or to the blue trajectories according to relevant theorems, similarly as above. For application of Theorem \ref{thm:saddle_collinear-chaotic_set} we must consider the blue trajectories belong to the first branch $\dot x=f(x)$, where the type of hyperbolic singular point is unstable saddle, then the black trajectories belong to the second branch $\dot x = g(x)$, where the type of hyperbolic singular point is various. The chaotic sets are figured by red hatched areas. In the red hatched area (chaotic set) we go first along the black trajectories in direction indicated by corresponding arrows, then we switch on the blue trajectories and go along these trajectories in direction indicated by corresponding arrows etc., or vice versa.

\begin{figure}[htp]
\centering
\begin{minipage}[c]{225pt}
\hspace{0.7cm}  
 \includegraphics[width=180pt]{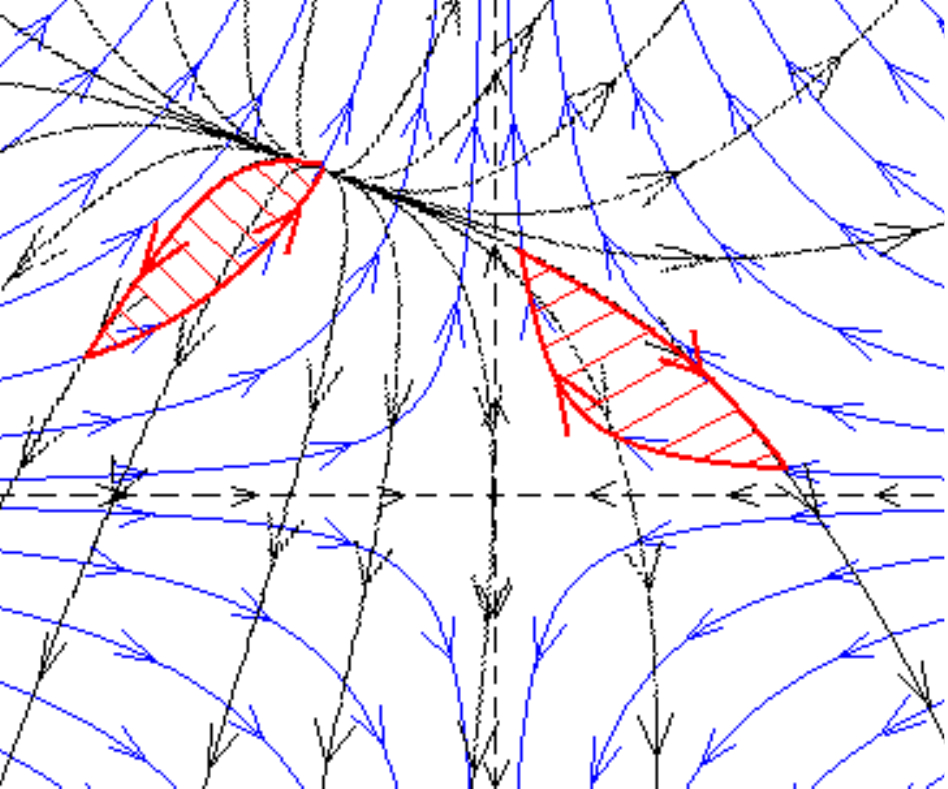}
\end{minipage} 
\begin{minipage}[c]{225pt}
\hspace{0.6cm}  
 \includegraphics[width=180pt]{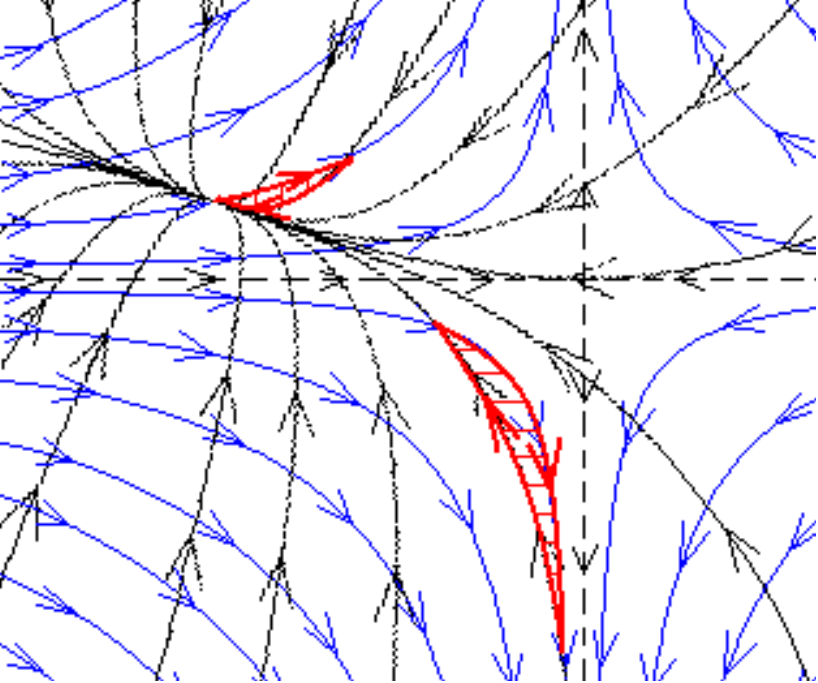}
\end{minipage}
\begin{minipage}[c]{225pt} 
 \caption{Chaotic sets with non-empty interior between unstable node and unstable saddle fulfilling condition $g(x^*)=\alpha e_1$ or $g(x^*)=\beta e_2$}
  \label{fig:unstable_node_unstable_saddle_3}
\end{minipage}  
\begin{minipage}[c]{225pt}  
   \caption{Chaotic sets with non-empty interior between stable node and unstable saddle fulfilling condition $g(x^*)=\alpha e_1$ or $g(x^*)=\beta e_2$}
  \label{fig:stable_node_unstable_saddle_3}
\end{minipage}
\\
\vspace{0.5cm}  
\begin{minipage}[c]{225pt}
\hspace{0.8cm}  
 \includegraphics[width=170pt]{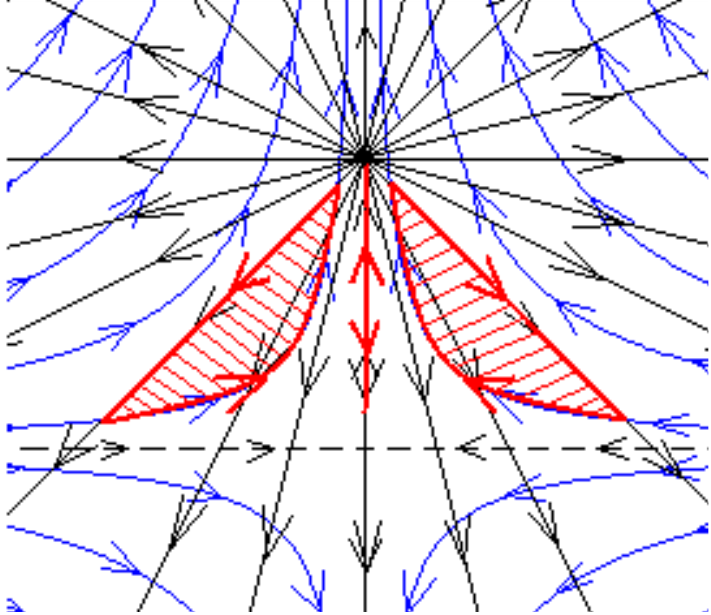}
\end{minipage} 
\begin{minipage}[c]{225pt}
\hspace{0.9cm}  
 \includegraphics[width=167pt]{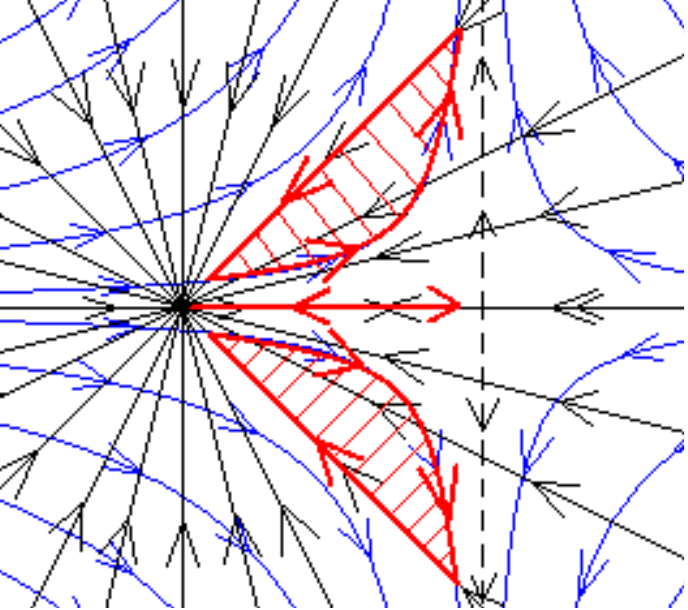}
\end{minipage}
\begin{minipage}[c]{225pt}
   \caption{Chaotic sets with non-empty interior and homeomorphic to $[0,1]$ with $\|f(x)\| \|g(x)\|>0$ and $\theta=\pi$ between unstable node and unstable saddle fulfilling $g(x^*)=\alpha e_1$ or $g(x^*)=\beta e_2$}
   \label{fig:unstable_node_unstable_saddle_4}
\end{minipage}  
\begin{minipage}[c]{225pt}  
  \caption{Chaotic sets with non-empty interior and homeomorphic to $[0,1]$ with $\|f(x)\| \|g(x)\|>0$ and $\theta=\pi$ between stable node and unstable saddle fulfilling $g(x^*)=\alpha e_1$ or $g(x^*)=\beta e_2$} 
   \label{fig:stable_node_unstable_saddle_4}
\end{minipage}
\end{figure}

\begin{figure}[htp]
\centering
\begin{minipage}[c]{225pt}
\hspace{0.8cm}  
 \includegraphics[width=170pt]{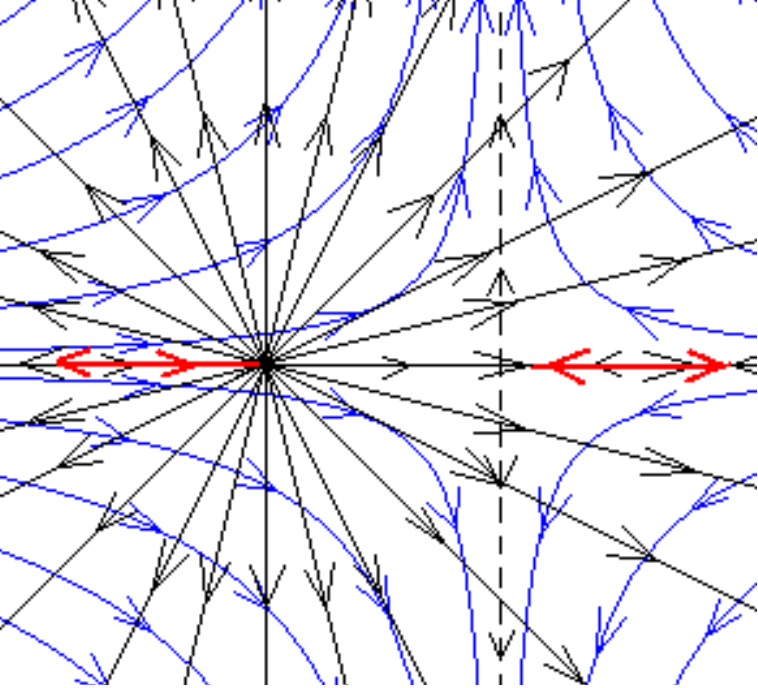}
\end{minipage} 
\begin{minipage}[c]{225pt}
\hspace{0.8cm}  
 \includegraphics[width=167pt]{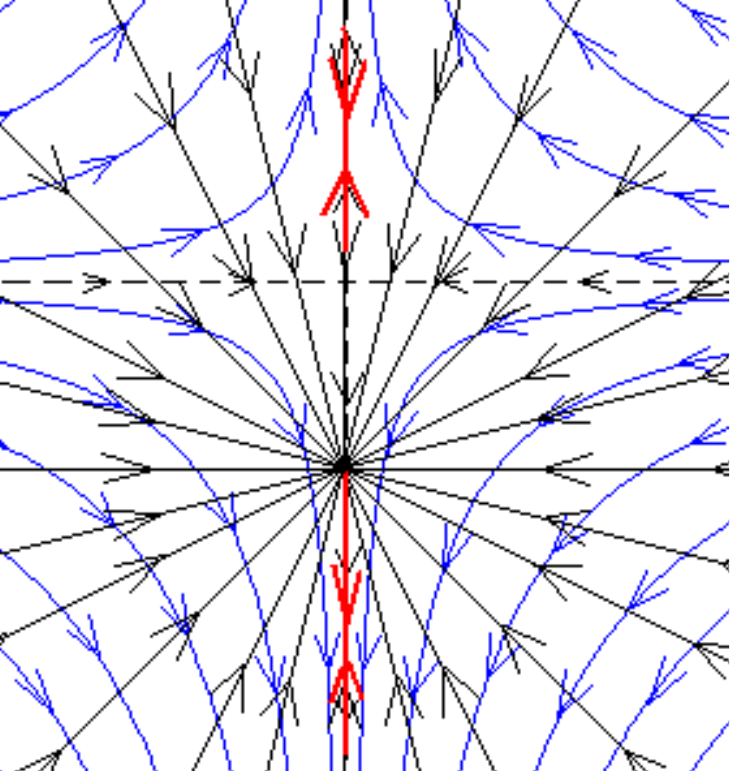}
\end{minipage}
\begin{minipage}[c]{225pt}
   \caption{Chaotic sets homeomorphic to $[0,1]$ with $\|f(x)\| \|g(x)\|>0$ and $\theta=\pi$ between unstable node and unstable saddle fulfilling $g(x^*)=\alpha e_1$ or $g(x^*)=\beta e_2$}
   \label{fig:unstable_node_unstable_saddle_5}
\end{minipage}  
\begin{minipage}[c]{225pt}  
  \caption{Chaotic sets homeomorphic to $[0,1]$ with $\|f(x)\| \|g(x)\|>0$ and $\theta=\pi$ between stable node and unstable saddle fulfilling $g(x^*)=\alpha e_1$ or $g(x^*)=\beta e_2$} 
   \label{fig:stable_node_unstable_saddle_5}
\end{minipage}
\\
\vspace{1cm}
\begin{minipage}[c]{210pt}
\hspace{0.1cm}  
  \includegraphics[width=180pt]{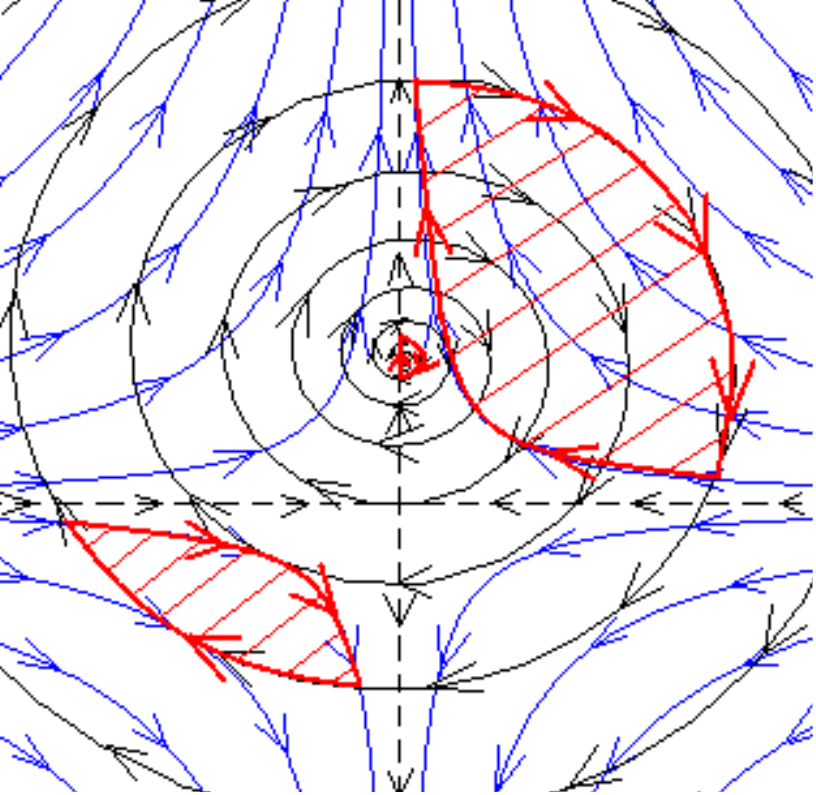}
\end{minipage}
\begin{minipage}[c]{210pt}
\hspace{0.2cm}  
 \includegraphics[width=195pt]{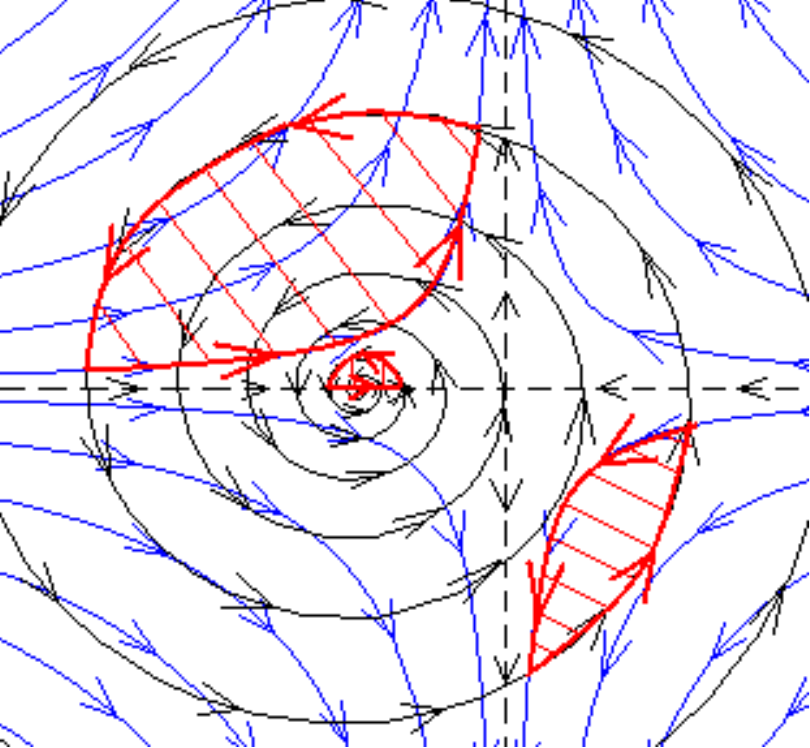}
\end{minipage}
\begin{minipage}[c]{225pt}  
  \caption{Chaotic sets with non-empty interior between unstable focus and unstable saddle fulfilling $g(x^*)=\alpha e_1$ or $g(x^*)=\beta e_2$}
  \label{fig:unstable_focus_unstable_saddle_2}
\end{minipage}
\begin{minipage}[c]{225pt}  
   \caption{Chaotic sets with non-empty interior between stable focus and unstable saddle fulfilling $g(x^*)=\alpha e_1$ or $g(x^*)=\beta e_2$}
  \label{fig:stable_focus_unstable_saddle_2}
\end{minipage}  
\end{figure}

\begin{figure}[htp]
\centering
\begin{minipage}[c]{210pt}
\hspace{0.3cm}
  \includegraphics[width=185pt]{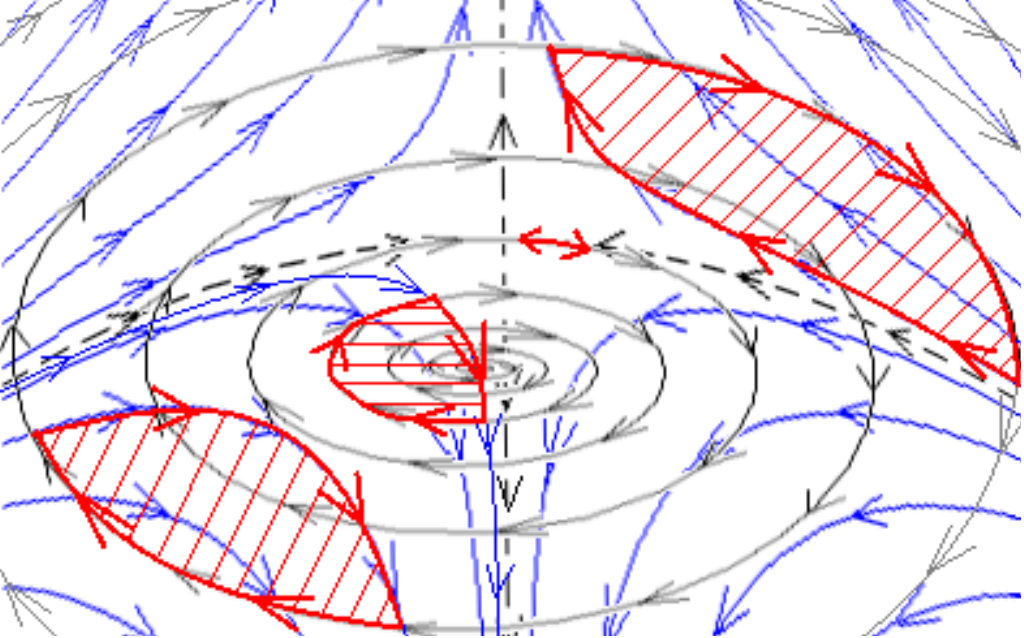}
\end{minipage}
\begin{minipage}[c]{210pt}
\hspace{0.6cm}  
 \includegraphics[width=170pt]{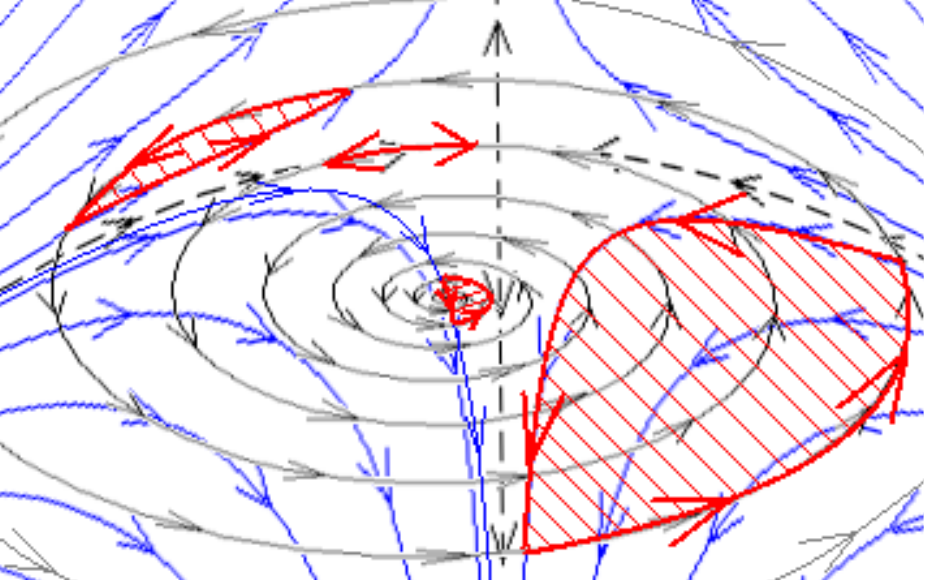}
\end{minipage}  
\begin{minipage}[c]{225pt} 
  \caption{Chaotic sets with non-empty interior and homeomorphic to $[0,1]$ with $\|f(x)\| \|g(x)\|>0$ and $\theta=\pi$ between unstable focus and unstable saddle fulfilling $g(x^*)=\alpha e_1$ or $g(x^*)=\beta e_2$}
  \label{fig:unstable_focus_unstable_saddle_3}
\end{minipage}  
\begin{minipage}[c]{225pt} 
  \caption{Chaotic sets  with non-empty interior and homeomorphic to $[0,1]$ with $\|f(x)\| \|g(x)\|>0$ and $\theta=\pi$ between stable focus and unstable saddle fulfilling $g(x^*)=\alpha e_1$ or $g(x^*)=\beta e_2$}
  \label{fig:stable_focus_unstable_saddle_3}
\end{minipage}
\\
\vspace{1cm}
\begin{minipage}[c]{225pt}
\hspace{0.5cm} 
 \includegraphics[width=190pt]{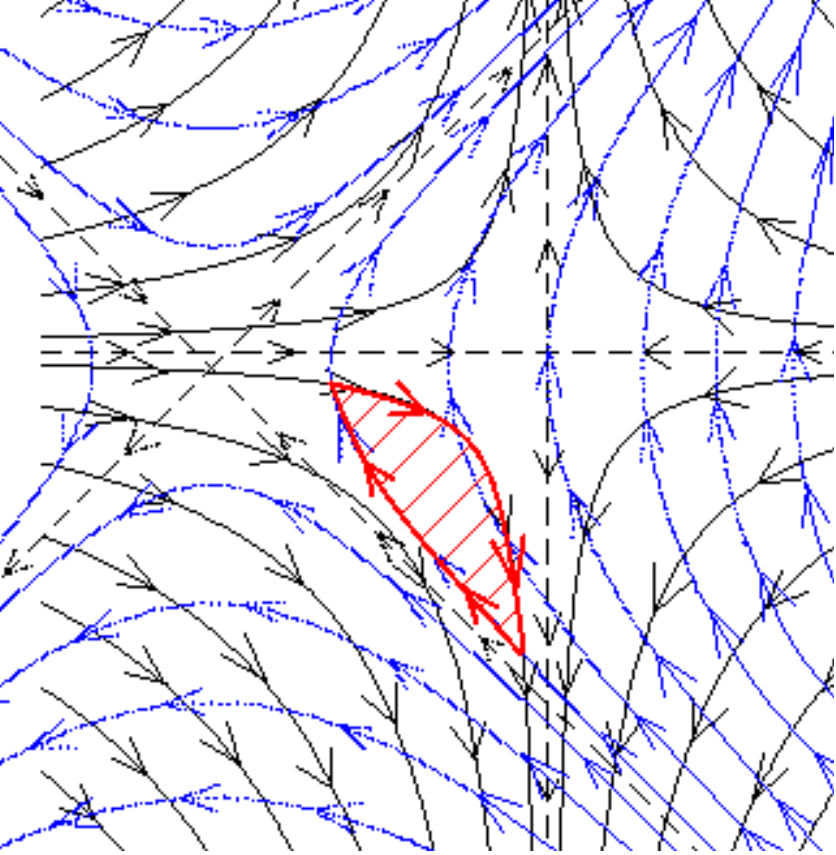}
\end{minipage} 
\begin{minipage}[c]{225pt}
\hspace{0.2cm} 
 \includegraphics[width=190pt]{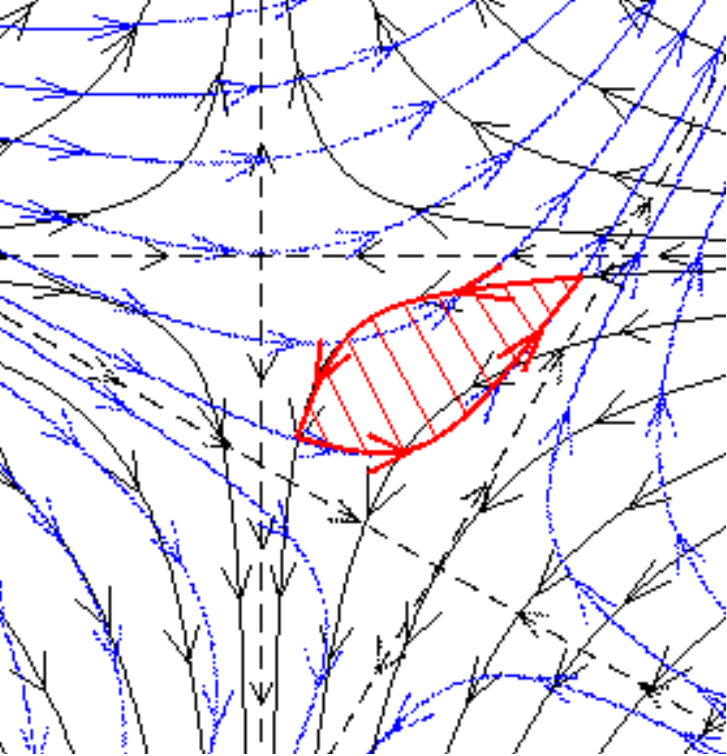}
\end{minipage}
  \caption{Chaotic set with non-empty interior between two unstable saddles fulfilling $g(x^*)=\alpha e_1$ or $g(x^*)=\beta e_2$ when stable or unstable manifolds are not overlapped}
  \label{fig:unstable_saddle_unstable_saddle_3_4}
\end{figure}

\begin{figure}
\begin{minipage}[c]{225pt}
\hspace{0.5cm}
  \includegraphics[width=183pt]{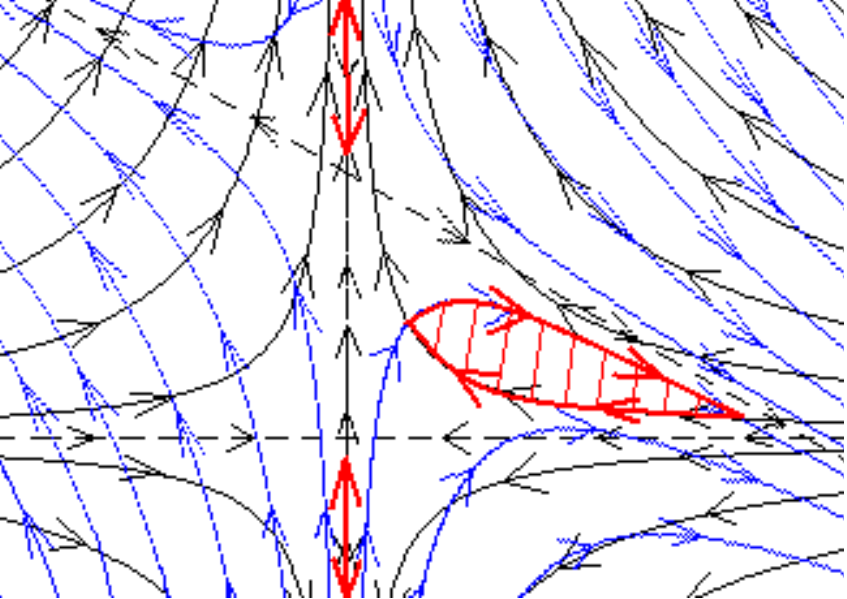}
\end{minipage}
\begin{minipage}[c]{225pt}
\hspace{0.5cm}
  \includegraphics[width=190pt]{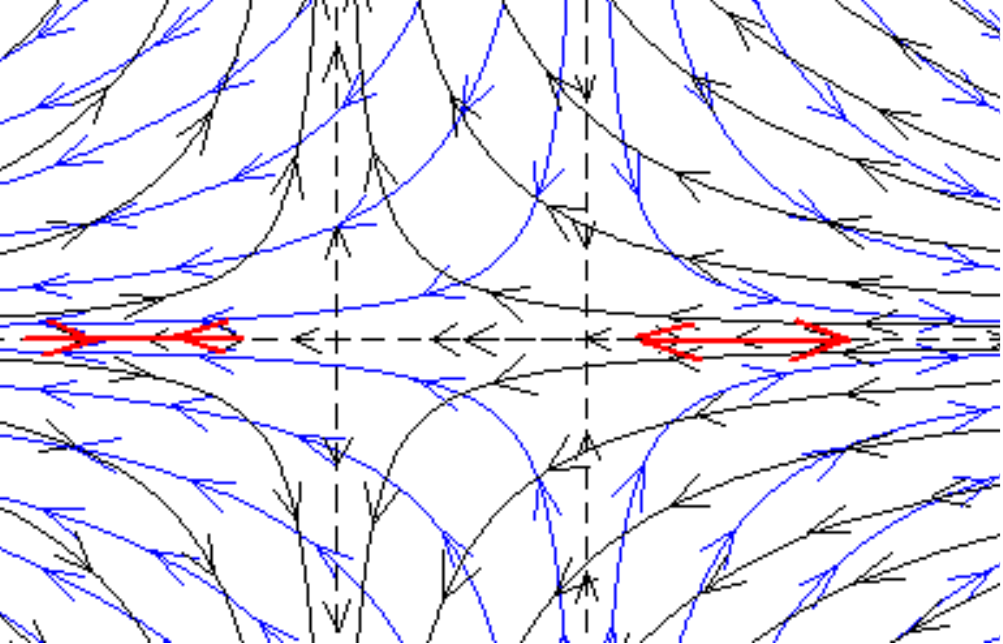}
\end{minipage} 
\begin{minipage}[c]{225pt}
  \caption{Chaotic sets  with non-empty interior and homeomorphic to $[0,1]$ with $\|f(x)\| \|g(x)\|>0$ and $\theta=\pi$ between two unstable saddles fulfilling condition $g(x^*)=\alpha e_1$ or $g(x^*)=\beta e_2$ when stable manifold overlaps unstable manifold}
  \label{fig:unstable_saddle_unstable_saddle_5}
\end{minipage}
\begin{minipage}[c]{225pt}
  \caption{Chaotic sets homeomorphic to $[0,1]$ with $\|f(x)\| \|g(x)\|>0$ and $\theta=\pi$ between two unstable saddles fulfilling condition $g(x^*)=\alpha e_1$ or $g(x^*)=\beta e_2$ when stable manifold overlaps unstable manifold}
  \label{fig:unstable_saddle_unstable_saddle_6}
\end{minipage}
\\
\vspace{0.3cm}
\begin{minipage}[c]{225pt}
\hspace{0.5cm}
  \includegraphics[width=183pt]{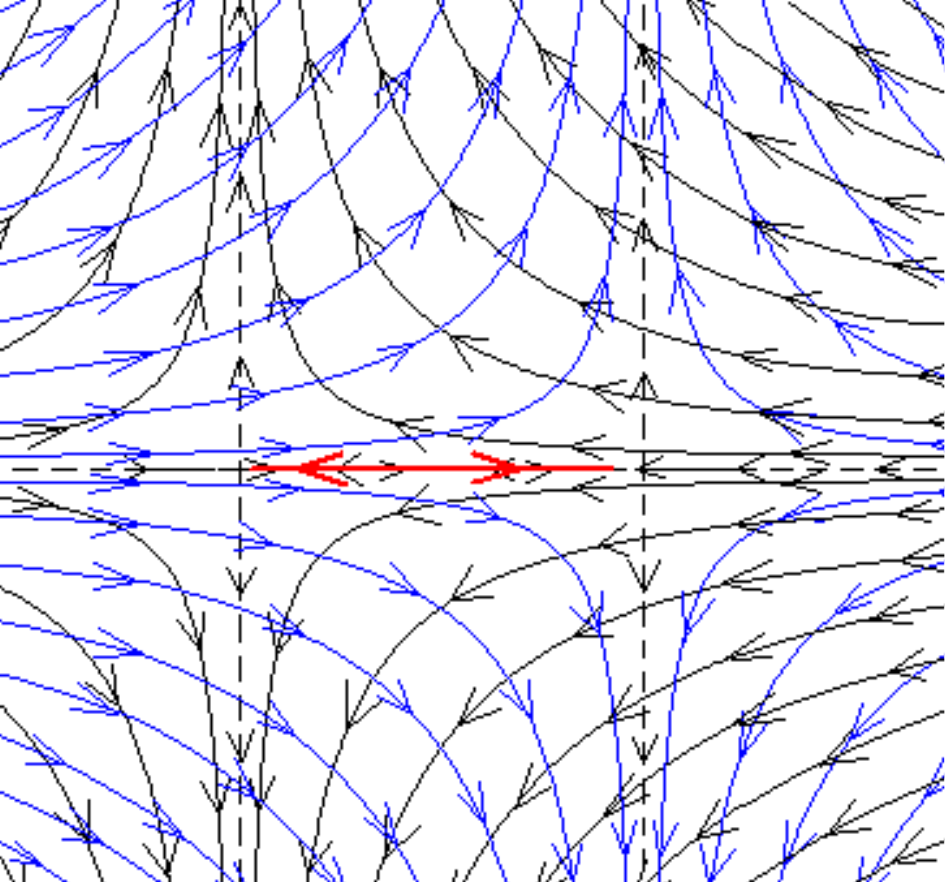}
\end{minipage}
\begin{minipage}[c]{225pt}
\hspace{0.5cm}
  \includegraphics[width=190pt]{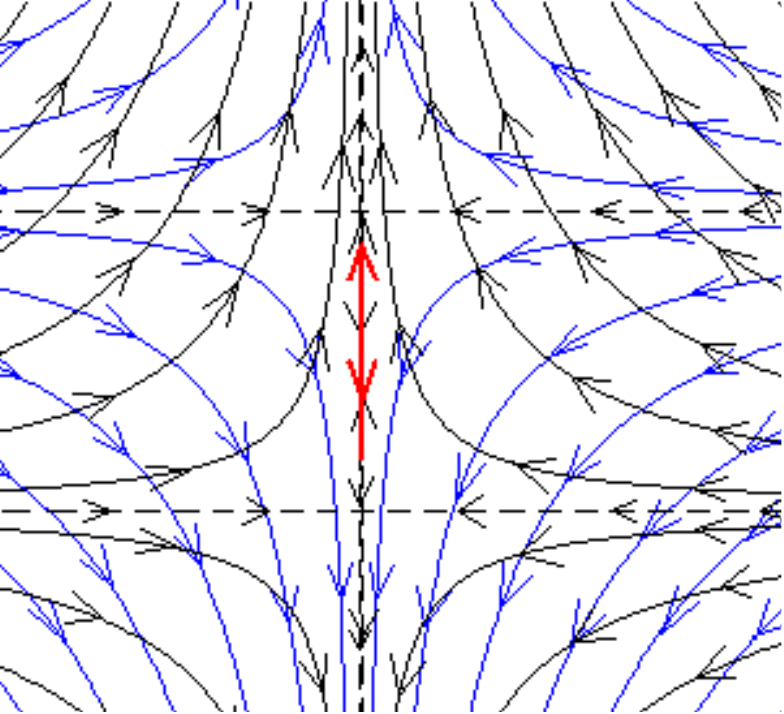}
\end{minipage} 
\begin{minipage}[c]{225pt}
  \caption{Chaotic set homeomorphic to $[0,1]$ with $\|f(x)\| \|g(x)\|>0$ and $\theta=\pi$ between two unstable saddles fulfilling condition $g(x^*)=\alpha e_1$ or $g(x^*)=\beta e_2$ when stable manifolds are overlapped}
  \label{fig:unstable_saddle_unstable_saddle_7}
\end{minipage}
\begin{minipage}[c]{225pt}
  \caption{Chaotic set homeomorphic to $[0,1]$ with $\|f(x)\| \|g(x)\|>0$ and $\theta=\pi$ between two unstable saddles fulfilling condition $g(x^*)=\alpha e_1$ or $g(x^*)=\beta e_2$ when unstable manifolds are overlapped}
  \label{fig:unstable_saddle_unstable_saddle_8}
\end{minipage}
\end{figure}

\clearpage

\subsection{Devaney, Li-Yorke and Distributional Chaotic Set of Solutions $V^*$}
\label{chaotic_set_of_solutions}

In the previous section \ref{overview_of_possibilities_admitting_chaotic_set}, we show that the chaotic set $V$ is always admitted in $\mathbb{R}^2$ with two hyperbolic singular points not lying in the same point but the existence of chaos also depends on set of solutions $V^*$ generating solutions $\gamma$ ensuring the chaotic set $V$, see Definition \ref{df:chaotic_set}.

Let us remind: the dynamical system generated by Euler equation branching $D:=\{\gamma \in Z | \dot \gamma(t) \in F(\gamma(t)) \textit{ a.e.} \}$, $Z=\{ \gamma| \gamma: T \rightarrow X \}$, $\gamma$ continuous and continuously differentiable a.e., $V \subset X \subseteq \mathbb{R}^2$ non-empty, compact $F$-invariant set and $V^*=\{ \gamma \in D|\gamma(t) \in V, \textit{ for all } t \in T \}$. The solution $\gamma$ from $D$ is composed of the part corresponding to the solution of the branch $\dot x = f(x)$, of the part corresponding to the solution of the branch $\dot x = g(x)$ and of the switching system between these two branches. For each point $x_0 \in X$ the solution $\gamma$ contains also the consequence of times $T_i,~i=0,1,2,3,...$ such that $\gamma(T_0)=\gamma(0)=x_0$ and $T_i>0$ for odd $i$ give the times of switching from the branch $\dot x = f(x)$ to $\dot x = g(x)$ and  $T_i>0$ for even $i$ give the times of switching from the branch $\dot x = g(x)$ to $\dot x = f(x)$, if we start by branch $f$, or vice versa (odd $i$ for switch from $g$ to $f$ and even $i$ for switch from $f$ to $g$), if we start by branch $g$. From the nature of such solutions follows that for each point $x_0 \in V$ there exist uncountable many solutions differing just in such switching system.

The set of solutions $V^*$ corresponding to $V \subset \mathbb{R}^2$ has to fulfil three conditions to be Devaney, Li-Yorke and distributional chaotic - every solution $\gamma \in V^*$ "stays forever in $V$" (F-invariant set $V$ and definition of $V^*$), each point $x \in V$ "can be connected with each another point" in $V$ by simple path given by some $\gamma \in V^*$ (the property (\ref{df:chaotic_set_a}) of chaotic set $V$) and there exists $\gamma \in V^*$ such that $\{ \gamma(t):t \in T \}$ is not dense in $V$ (the property (\ref{df:chaotic_set_b}) of chaotic set $V$).

\begin{thm}
\label{thm:chaotic_set_of_solutions}
In dynamical system $D$ generated by Euler equation branching in $\mathbb{R}^2$ there exists the set of solutions $V^*$ which ensures chaotic set $V$, hence this set $V^*$ is Devaney, Li-Yorke and distributional chaotic. Moreover the set of such $V^*$ is uncountable.
\end{thm}
\begin{proof}
We assume that starting point $x_0$ is influenced firstly by branch $f$, i.e. the consequence of switching times  $T_i,~i=0,1,2,3,...$ is such that $\gamma(T_0)=\gamma(0)=x_0$ and $T_i>0$ for odd $i$ give the times of switching from the branch $f$ to $g$ and  $T_i>0$ for even $i$ give the times of switching from the branch $g$ to $f$. So, we consider  $x^* \in X \subset \mathbb{R}^2$, $f(x^*)=0$, $g(x^*) \neq 0$,  $\delta > 0$ such that the solution of $\dot x = g(x)$ is unbounded in $\bar{B}_{\delta}(x^*)$ and $g(x) \neq 0$ for every $x \in \bar B_{\delta}(x^*)$. If we assume an opposite situation (start in $g$), then the proof will be analogous with difference that we consider $g(y^*)=0$ and unbounded solution of $\dot x = f(x)$ in $\bar{B}_{\delta}(y^*)$ etc. Denote by $\varphi^t(x)$ and $\psi^t(x)$ the flow belonging  to $f$ and to $g$. \\
First, we construct such set of solutions $V^*$ for $V$ with non-empty interior. We initially assume $x^*$ is unstable (node, focus, saddle). We denote by $\{ \gamma_0, \gamma_1\}$ a chaotic set of solutions $V^*$. We describe the solutions $\gamma_0,\gamma_1$ using the Figure \ref{fig:proof_chaotic_set_of_solutions_theorem_1} for $x^*$ node, using the Figure \ref{fig:proof_chaotic_set_of_solutions_theorem_2} for $x^*$ focus, and using the Figure \ref{fig:proof_chaotic_set_of_solutions_theorem_3} for $x^*$ saddle. By yellow areas the sets $V$ are figured. If $x^*$ is node, the set $V$ is bounded by trajectory of $g$ denoted by $\psi_B$ arbitrary closed to the trajectory of $g$ (figured by dashed line) passing through the point $x^*$ and by arbitrary trajectory of $f$ denoted by $\varphi_B$ intersecting trajectory $\psi_B$ in two points such that for some $t_1<t_2$, $s_1<s_2$ it holds $\varphi_B^{t_1}=\psi_B^{s_2}=:z_1$ and $\varphi_B^{t_2}=\psi_B^{s_1}=:z_2$ and $\psi_B^s \notin \varphi_B $ for every $s \in (s_1,s_2)$, see Figure \ref{fig:proof_chaotic_set_of_solutions_theorem_1}.
\begin{figure}[ht]
  \centering
  \includegraphics[height=2.1cm]{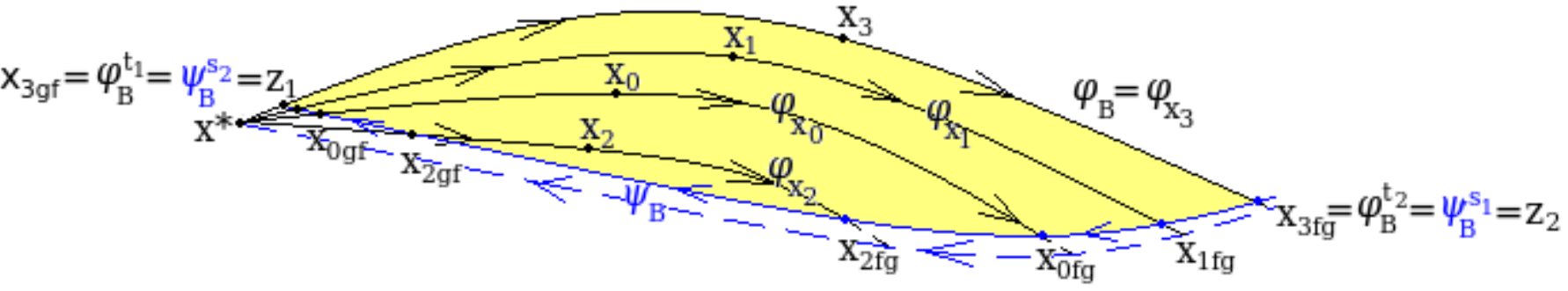}
  \caption{The first proof scheme of Theorem \ref{thm:chaotic_set_of_solutions}}
  \label{fig:proof_chaotic_set_of_solutions_theorem_1}
\end{figure}
If $x^*$ is focus, the set $V$ is bounded by trajectory of $g$ denoted by $\psi_B$ passing through the point $x^*$ and by arbitrary trajectory of $f$ denoted by $\varphi_B$ intersecting trajectory $\psi_B$ in two points such that for some $t_1<t_2$, $s_1<s_2$ it holds $\varphi_B^{t_1}=\psi_B^{s_2}=:z_1$ and $\varphi_B^{t_2}=\psi_B^{s_1}=:z_2$ and $\varphi_B^t \notin \psi_B$ for every $t \in (t_1,t_2)$ and $\psi_B^s \notin \varphi_B $ for every $s \in (s_1,s_2)$, see Figure  \ref{fig:proof_chaotic_set_of_solutions_theorem_2}.
\begin{figure}[ht]
  \centering
  \includegraphics[height=4.2cm]{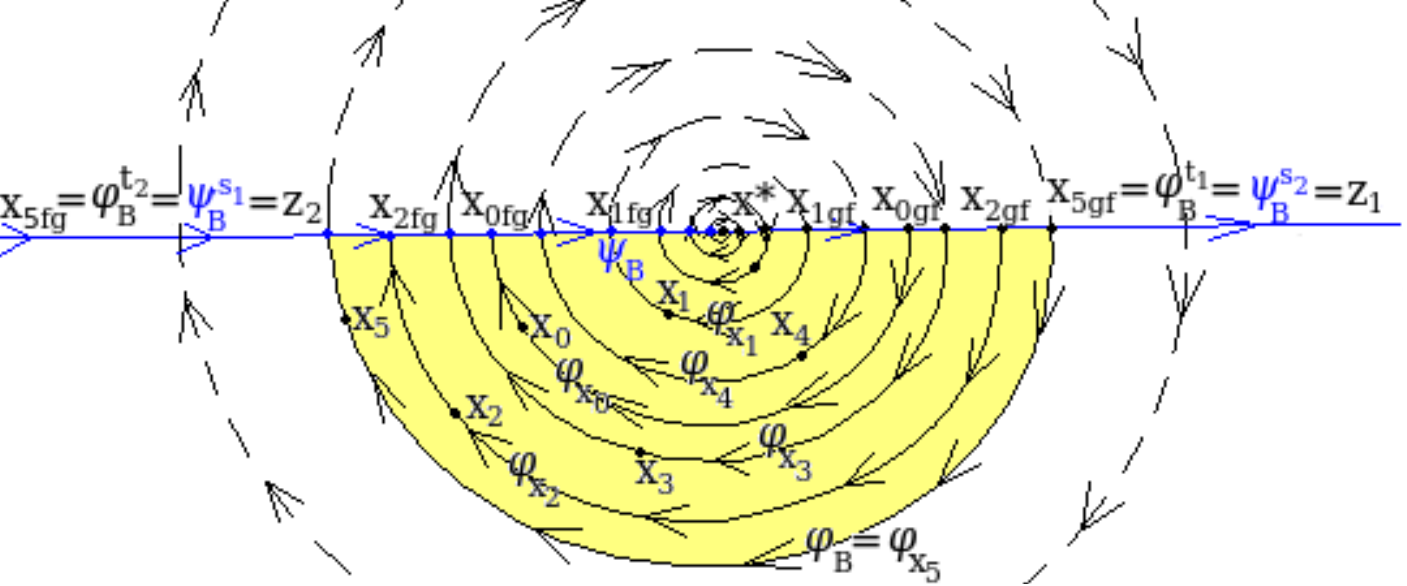}
  \caption{The second proof scheme of Theorem \ref{thm:chaotic_set_of_solutions}}
  \label{fig:proof_chaotic_set_of_solutions_theorem_2}
\end{figure}
If $x^*$ is saddle, the set $V$ is bounded by arbitrary trajectory of $g$ denoted by $\psi_B$ intersecting the unstable manifold $W^u$ in time $T_{W_u}$ and the stable manifold $W^s$ in time $T_{W_s}$ such that $T_{W_u}<T_{W_s}$ and by arbitrary trajectory of $f$ denoted by $\varphi_B$ intersecting trajectory $\psi_B$ in two points such that for some $t_1<t_2$, $T_{W_u}<s_1<s_2<T_{W_u}$ it holds $\varphi_B^{t_1}=\psi_B^{s_2}=:z_1$ and $\varphi_B^{t_2}=\psi_B^{s_1}=:z_2$ and $\psi_B^s \notin \varphi_B $ for every $s \in (s_1,s_2)$, see Figure  \ref{fig:proof_chaotic_set_of_solutions_theorem_3} (on the left), in "non-degenerated" case describing by Theorem \ref{thm:hyperb_sing_point-chaotic_set} (\ref{thm:saddle-chaotic_set}) or by part (\ref{proof:g_neq_mu.f}) of the proof of Theorem \ref{thm:saddle_collinear-chaotic_set}. In "degenerated" case describing by Remark \ref{rmk:to_thm-saddle_collinear-chaotic_set}, where stable manifold of one saddle overlaps the unstable manifold of the second saddle, the set $V$ is bounded by arbitrary trajectory of $g$ denoted by $\psi_B$ intersecting the separate not overlapped manifold in two points and by arbitrary trajectory of $f$ denoted by $\varphi_B$ intersecting trajectory $\psi_B$ in two points. Let $\varphi_B^{t_1}=\psi_B^{s_2}=:z_1$ and $\varphi_B^{t_2}=\psi_B^{s_1}=:z_2$ for some $t_1<t_2$, $s_1<s_2$, see Figure \nolinebreak \ref{fig:proof_chaotic_set_of_solutions_theorem_3} (on the right).
\begin{figure}[ht]
  \centering
  \includegraphics[height=6.5cm]{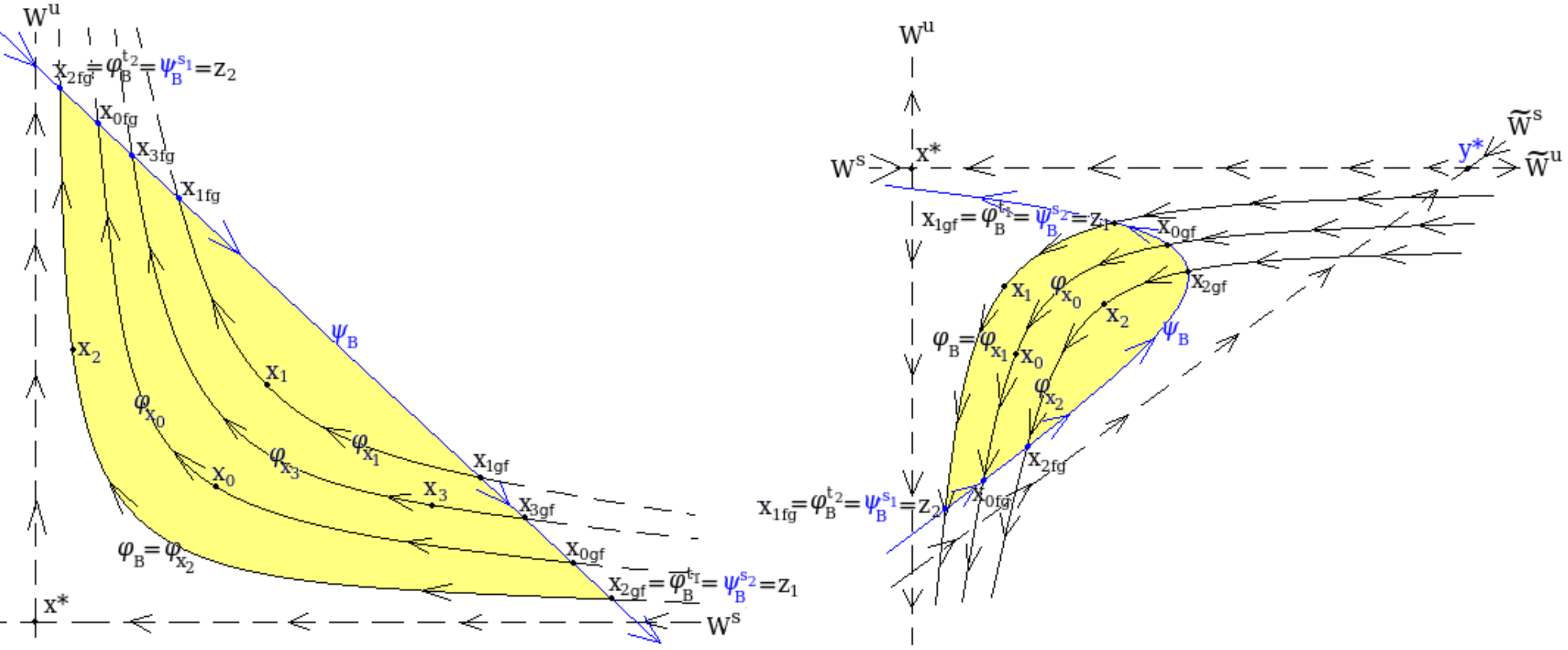}
  \caption{The third proof scheme of Theorem \ref{thm:chaotic_set_of_solutions}}
  \label{fig:proof_chaotic_set_of_solutions_theorem_3}
\end{figure}
For every point $x_0 \in V$ the solution $\gamma_0$ is given by the trajectory of branch $f$ denoted by $\varphi_{x_0}$ from the point $x_0$ to the point $x_{0fg} \in \psi_B$, then by the trajectory $\psi_B$ from the point $x_{0fg}$ to the point $x_{0gf}$, then by the trajectory $\varphi_{x_0}$ from the point $x_{0gf}$ to the point $x_{0fg}$ through the point $x_0$ and so on, see Figure \ref{fig:proof_chaotic_set_of_solutions_theorem_1} for node $x^*$, see Figure \nolinebreak \ref{fig:proof_chaotic_set_of_solutions_theorem_2} for focus $x^*$ and see Figure \ref{fig:proof_chaotic_set_of_solutions_theorem_3} for saddle $x^*$. So, the solution $\gamma_0$ is given by consequence of time $T_0=0$, $T_1$, $T_{2k}=T_1+(k-1)T_{\varphi}+kT_{\psi}$, $T_{2k+1}=T_1+kT_{\varphi}+kT_{\psi}$, $k \in \{ 1,2,3, ...\}$, such that $\varphi^{T_0}(x_0)=\varphi^0(x_0)=x_0$, $\varphi^{T_1}(x_0)=x_{0fg} \in \psi_B$, $\psi^{T_{\psi}}(\varphi^{T_1}(x_0))=\psi^{T_{\psi}}(x_{0fg})=x_{0gf}$, $\varphi^{T_{\varphi}}(x_{0gf})=\varphi^{T_1}(x_0)=x_{0fg}$. We denote by $\gamma_1$ the solutions which are described below. For every point $x_0 \in V$ the solution $\gamma_{x_j} \in \gamma_1$ passing through every point $x_j \in V \setminus \varphi_{x_0}$ ($j \in \mathbb{R} \setminus \{0\}$) is given by the trajectory of branch $f$ (denoted by $\varphi_{x_0}$) from the point $x_0$ to the point $x_{0fg} \in \psi_B$, then by the trajectory $\psi_B$ from the point $x_{0fg}$ to the point $x_{jgf}$, then by the trajectory of $f$ denoted by $\varphi_{x_j}$ passing through the point $x_j$ from the point $x_{jgf}$ to the point $x_{jfg} \in \psi_B$, then by the trajectory $\psi_B$ from the point $x_{jfg}$ to the point $x_{0gf}$, then by the trajectory $\varphi_{x_0}$ from the point $x_{0gf}$ through the point $x_0$ to the point $x_{0fg}$, then by the trajectory $\psi_B$ from the point $x_{0fg}$ to the point $x_{jgf}$ etc., see e.g. for the point $x_2$ on Figure \ref{fig:proof_chaotic_set_of_solutions_theorem_1} for node $x^*$, on Figure \ref{fig:proof_chaotic_set_of_solutions_theorem_2} for focus $x^*$ and on Figure \ref{fig:proof_chaotic_set_of_solutions_theorem_3} for saddle $x^*$. So, the solution $\gamma_{x_j}$ is given by the consequence of the time $T_0=0$, $T_1$, $T_{4k-2}=T_1+kT_{\psi_{x_0 x_j}}+(k-1)T_{\varphi_{x_j}}+(k-1)T_{\psi_{x_j x_0}}+(k-1)T_{\varphi_{x_0}}$, $T_{4k-1}=T_1+kT_{\psi_{x_0 x_j}}+kT_{\varphi_{x_j}}+(k-1)T_{\psi_{x_j x_0}}+(k-1)T_{\varphi_{x_0}}$, $T_{4k}=T_1+kT_{\psi_{x_0 x_j}}+kT_{\varphi_{x_j}}+kT_{\psi_{x_j x_0}}+(k-1)T_{\varphi_{x_0}}$ and $T_{4k+1}=T_1+kT_{\psi_{x_0 x_j}}+kT_{\varphi_{x_j}}+kT_{\psi_{x_j x_0}}+kT_{\varphi_{x_0}}$, $k \in \{ 1,2,3, ...\}$ such that $\varphi^{T_0}(x_0)=\varphi^0(x_0)=x_0$, $\varphi^{T_1}(x_0)=x_{0fg} \in \psi_B$, $\psi^{T_{\psi_{x_0 x_j}}}(x_{0fg})=x_{jgf}$, $\varphi^{T_{\varphi_{x_j}}}(x_{jgf})=x_{jfg}$, $\psi^{T_{\psi_{x_j x_0}}}(x_{jfg})=x_{0gf}$, $\varphi^{T_{\varphi_{x_0}}}(x_{0gf})=\varphi^{T_1}(x_0)=x_{0fg}$. Especially $T_1=0$ for starting points lying on $\psi_B$ (we start with the switch). It is obvious that $V^*=\{ \gamma_0, \gamma_1\}$ fulfils three conditions mentioned above to be chaotic. The solutions $\gamma_0$ and also every $\gamma_{x_j}$ "stay forever in $V$". Each point $x \in V$ "can be connected with each another point" in $V$ by some $\gamma_{x_j}$ or $\gamma_0$. Neither $\gamma_0$ nor $\gamma_{x_j} \in \gamma_1$ for every $j \in \mathbb{R} \setminus \{0\}$ are dense in $V$. The proof for stable node or stable focus is analogous with reversed time direction.\\
The construction of set of solutions $V^*$ for $V$ homeomorphic to $[0,1]$ is similar. The set $V$ is the curve segment with two distinct end points not containing $x^*$ or $y^*$ where $cos^{-1}\left( \frac{f(x) \cdot g(x)}{\|f(x)\| \|g(x)\|} \right)=\pi$ for all $x \in V$, i.e. where the trajectories corresponding to the flows $\varphi$ and $\psi$ have the opposite direction, see Figure \ref{fig:unstable_node_unstable_node_2}, \ref{fig:unstable_node_stable_node_2}, \ref{fig:unstable_node_unstable_saddle_2}, \ref{fig:stable_node_stable_node_2}, \ref{fig:stable_node_unstable_saddle_2}, \ref{fig:unstable_node_unstable_saddle_4}-\ref{fig:stable_node_unstable_saddle_5}, \ref{fig:unstable_focus_unstable_saddle_3}, \ref{fig:stable_focus_unstable_saddle_3} or \ref{fig:unstable_saddle_unstable_saddle_5}-\ref{fig:unstable_saddle_unstable_saddle_8}. Let $\varphi_B$ and $\psi_B$ be trajectories corresponding to $f$ and $g$ such that $\varphi_B^{t_1}=\psi_B^{s_2}=z_1$ and $\varphi_B^{t_2}=\psi_B^{s_1}=z_2$ for some $t_1<t_2$, $s_1<s_2$ where $z_1$ and $z_2$ be two distinct end points of the curve segment, see Figure \ref{fig:proof_chaotic_set_of_solutions_theorem_4}.
\begin{figure}[ht]
  \centering
  \includegraphics[height=1.2cm]{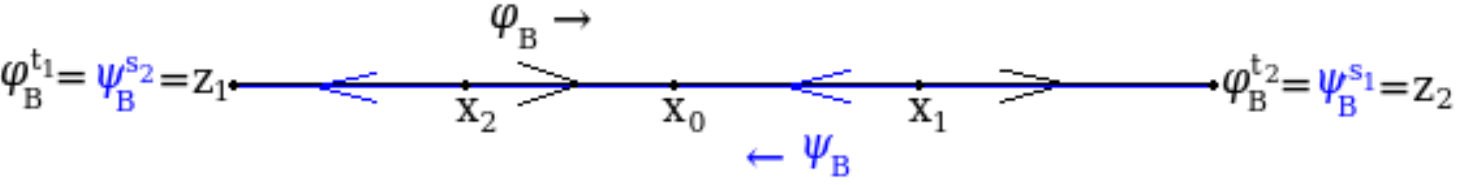}
  \caption{The fourth proof scheme of Theorem \ref{thm:chaotic_set_of_solutions}}
  \label{fig:proof_chaotic_set_of_solutions_theorem_4}
\end{figure}
Let analogously $\{ \gamma_0, \gamma_1\}$ be a chaotic set of solutions $V^*$. We describe the solutions $\gamma_0,\gamma_1$ using Figure \ref{fig:proof_chaotic_set_of_solutions_theorem_4}. The chaotic set $V$ is displayed by curve with two distinct end points $z_1$ and $z_2$, see Figure \nolinebreak \ref{fig:proof_chaotic_set_of_solutions_theorem_4}. For every point $x_0 \in V$ the solution $\gamma_0$ is given by the trajectory $\varphi_B$ from the point $x_0$ to the point $z_2$, then by the trajectory $\psi_B$ from the point $z_2$ to the point $x_0$, then by the trajectory $\varphi_B$ from the point $x_0$ to the point $z_2$ and so on, see Figure \ref{fig:proof_chaotic_set_of_solutions_theorem_4}. So, the solution $\gamma_0$ is given by consequence of time $T_0=0$, $T_{2k-1}=kT_{\varphi}+(k-1)T_{\psi}$, $T_{2k}=kT_{\varphi}+kT_{\psi}$, $k \in \{ 1,2,3, ...\}$ such that $\varphi^{T_0}(x_0)=\varphi^0(x_0)=x_0$, $\varphi^{T_{\varphi}}(x_0)=z_2$, $\psi^{T_{\psi}}(\varphi^{T_{\varphi}}(x_0))=\psi^{T_{\psi}}(z_2)=x_0$. We denote analogously by $\gamma_1$ the solutions which are described below. For every point $x_0 \in V$ the solution $\gamma_{x_j} \in \gamma_1$ denoted per $x_j \in V$, $x_j \neq x_0$ ($j \in \mathbb{R} \setminus \{0\} $) is given by the trajectory $\varphi_B$ from the point $x_0$ to the point $z_2$, then by the trajectory $\psi_B$ from the point $z_2$ to the point $x_j$, then by the trajectory $\varphi_B$ from the point $x_j$ to the point $z_2$, then by the trajectory $\psi_B$ from the point $z_2$ to the point $x_0$ etc., see e.g. for the point $x_2$ on Figure \ref{fig:proof_chaotic_set_of_solutions_theorem_4}. So, the solution $\gamma_{x_j}$ is given by the consequence of the time $T_0=0$, $T_{4k-3}=kT_{\varphi_{x_0 z_2}}+(k-1)T_{\psi_{z_2 x_j}}+ (k-1)T_{\varphi_{x_j z_2}}+(k-1)T_{\psi_{z_2 x_0}}$, $T_{4k-2}=kT_{\varphi_{x_0 z_2}}+kT_{\psi_{z_2 x_j}}+(k-1)T_{\varphi_{x_j z_2}}+(k-1)T_{\psi_{z_2 x_0}}$, $T_{4k-1}=kT_{\varphi_{x_0 z_2}} +kT_{\psi_{z_2 x_j}}+kT_{\varphi_{x_j z_2}}+(k-1)T_{\psi_{z_2 x_0}}$ and $T_{4k}=kT_{\varphi_{x_0 z_2}}+kT_{\psi_{z_2 x_j}}+kT_{\varphi_{x_j z_2}}+kT_{\psi_{z_2 x_0}}$, $k \in \{ 1,2,3, ...\}$ such that $\varphi^{T_0}(x_0)=\varphi^0(x_0)=x_0$, $\varphi^{T_{\varphi_{x_0 z_2}}}(x_0)=z_2$, $\psi^{T_{\psi_{z_2 x_j}}}(z_2)=x_j$, $\varphi^{T_{\varphi_{x_j z_2}}}(x_j)=z_2$, $\psi^{T_{\psi_{z_2 x_0}}}(z_2)=x_0$. Especially $T_1=0$ for starting points $z_2$ (we start with the switch). It is analogously obvious that $V^*=\{ \gamma_0, \gamma_1\}$ is chaotic chaotic set of solutions.\\
At the end of the proof, we remark that from the nature of such type of set of solutions $V^*$ follows uncountability of the set of such $V^*$. We can construct uncountable many sets of solutions $V^*$ based on presented construction. Let $\gamma_2$ be set of solution such that for every $x_0 \in V$ the moving point goes on the trajectory of $f$ from the point $x_0$ to the point $x_{0fg}$ or $z_2$, then on the trajectory of $g$ from the point $x_{0fg}$ or $z_2$ to the point $x_{jgf}$ where $x_{jgf} \neq x_{0gf}$ or to $x_j$ where $j \neq 0$, then on the trajectory of $f$ to the point $x_{jfg}$ or $z_2$, then on the trajectory of $g$ to the point $x_{kgf}$ where $x_{kgf} \neq x_{jgf} \neq x_{0gf}$ or to $x_k$ where $k \neq j \neq 0$, then on the trajectory of $f$ to the point $x_{kfg}$ or $z_2$, then on the trajectory of $g$ to the point $x_{0gf}$ or $x_0$, then on the trajectory of $f$ to the point $x_{0fg}$ or $z_2$ etc., see Figure \ref{fig:proof_chaotic_set_of_solutions_theorem_1}, \ref{fig:proof_chaotic_set_of_solutions_theorem_2}, \ref{fig:proof_chaotic_set_of_solutions_theorem_3} or \ref{fig:proof_chaotic_set_of_solutions_theorem_4}. Let similarly $\gamma_{1.2}$ be set of solution such that for every $x_0 \in V$ the moving point goes on the trajectory of $f$ from the point $x_0$ to the point $x_{0fg}$ or $z_2$, then on the trajectory of $g$ from the point $x_{0fg}$ or $z_2$ to the point $x_{jgf}$ where $x_{jgf} \neq x_{0gf}$ or to $x_j$ where $j \neq 0$, then on the trajectory of $f$ to the point $x_{jfg}$ or $z_2$, then on the trajectory of $g$ again to the point $x_{jgf}$ where $x_{jgf} \neq x_{0gf}$ or to $x_j$ where $j \neq 0$, then on the trajectory of $f$ to the point $x_{jfg}$ or $z_2$, then on the trajectory of $g$ to the point $x_{0gf}$ or $x_0$, then on the trajectory of $f$ to the point $x_{0fg}$ or $z_2$ etc., see Figure \ref{fig:proof_chaotic_set_of_solutions_theorem_1}, \ref{fig:proof_chaotic_set_of_solutions_theorem_2}, \ref{fig:proof_chaotic_set_of_solutions_theorem_3} or \ref{fig:proof_chaotic_set_of_solutions_theorem_4}. Then $\{\gamma_0, \gamma_2 \}$ or $\{\gamma_0, \gamma_{1.2} \}$ is chaotic set of solutions. We can construct in this way uncountable many sets of solutions $V^*$ based on combinations of every relevant path ("loops" and number of the same "loops") of solutions $\gamma \in V^*$.
\end{proof}

\begin{rmk}
In the proof of Theorem \ref{thm:chaotic_set_of_solutions} there are described boundaries of chaotic set $V$ with non-empty interior but every combination of considered singular points has a "maximal" area in plane $\mathbb{R}^2$ where the chaotic sets $V$ can exist as we can discern on Figures \ref{fig:unstable_node_unstable_node_1}-\ref{fig:unstable_saddle_unstable_saddle_1_2} and \ref{fig:unstable_node_unstable_saddle_3}-\ref{fig:unstable_saddle_unstable_saddle_8}. For example the combination of two unstable nodes has this area bounded by trajectory of $f$ passing through the point $y^*$ and by the trajectory of $g$ passing through the point $x^*$ (without these trajectories), see Figure \ref{fig:unstable_node_unstable_node_1}. For another example in the case of unstable and stable node this area is consists of two disconnected parts. The first part is bounded by trajectory of $f$ passing through the point $y^*$ and by the last trajectory of $g$ intersected this trajectory of $f$ (without this trajectory corresponding to $f$), and vice versa the second part is bounded by the trajectory of $g$ passing through the point $x^*$ and by the last trajectory of $f$ intersected this trajectory of $g$ (without this trajectory corresponding to $g$), see Figure \ref{fig:unstable_node_stable_node_1}. In the combination of two foci this maximal area is bounded analogously by first trajectories passing through the point $x^*$ or $y^*$ (with these trajectories) and this area can be one connected part or can consist of two disconnected parts and can not contain $x^*$ or $y^*$ as an interior point, i.e. $x^*$ or $y^*$ lies on the boundary line, see Figure \ref{fig:unstable_focus_unstable_focus}, \ref{fig:unstable_focus_stable_focus} or \ref{fig:stable_focus_stable_focus}. From Remark \ref{rmk:to_thm-saddle_collinear-chaotic_set} for "degenerated" case of combination of two saddles (stable manifold of one saddle overlaps the unstable manifold of the second saddle) follows that this maximal area is bounded by "triangle" $\triangle x^*y^*z$ (without boundary manifolds), see Figure \ref{fig:remark_to_proof_admitted_chaos_theorem_2}. In fact in this special case there are another areas where chaotic sets $V$ can exist. These areas are situated in "half-plane" bounded by the overlapped manifold and containing $\triangle x^*y^*z$, but the concrete position of this area depends of concrete situation.   
\end{rmk}

\begin{lemma}
\label{non_chaotic_set_of_one_solution}
The Devaney, Li-Yorke and distributional chaotic $V^*$ can not be set of only one solution $\gamma$.
\end{lemma}
\begin{proof}
If the set $V^*$ contains only one solution $\gamma$ then such solution will not fulfil both property (\ref{df:chaotic_set_a}) and (\ref{df:chaotic_set_b}) from Definition \ref{df:chaotic_set} together. Obviously, we can use only one solution $\gamma \in V^*$ to ensure property (\ref{df:chaotic_set_a}) (we want to every pair $(x,y),~x,y \in V$ be scrambled), hence $\gamma$ has to go through every point of the set $V$. But then $\{ \gamma(t):t \in T \}$ is dense in $V$. But either such solution $\gamma$ is not Devaney, Li-Yorke and distributional chaotic itself because the points lying in the same trajectory (corresponding $f$ or $g$) are not scrambled.
\end{proof}

\begin{lemma}
The set of all Devaney, Li-Yorke and distributional chaotic sets of solutions $V^*$ on $V \subset X \subseteq \mathbb{R}^2$ is not dense or nowhere dense in $\mathcal{P}(D)$ ($D$ restricted on $V \subset X$).
\end{lemma}
\begin{proof}
Consider $D:=\{\gamma \in Z | \dot \gamma(t) \in F(\gamma(t)) \textit{ a.e.} \}$ restricted on $V$. Let $\mathcal{P}(D)$ be the power set of $D$ (the set of all subset of $D$), $\mathcal{P}_C(D)$ be set of all Devaney, Li-Yorke and distributional chaotic sets of solutions $\gamma$ and $\mathcal{P}_N(D)$ be set of all non-chaotic sets of solutions $\gamma$. Obviously $\mathcal{P}_C(D) \cup \mathcal{P}_N(D)=\mathcal{P}(D)$ and $\mathcal{P}_C(D) \cap \mathcal{P}_N(D)=\emptyset$. $\mathcal{P}_C(D) \neq \emptyset$, see Theorem \ref{thm:chaotic_set_of_solutions}, and $\mathcal{P}_N(D) \neq \emptyset$, non-chaotic set of solutions $\gamma$ is for example the set of only one $\gamma$, see Lemma \ref{non_chaotic_set_of_one_solution}, or of not F-invariant $\gamma$. We have topological space $\mathcal{P}(D)$ with discrete topology. $\mathcal{P}_C(D),\mathcal{P}_N(D) \subset \mathcal{P}(\mathcal{P}(D))$. So, $cl(\mathcal{P}_C(D))=\mathcal{P}_C(D)\neq \mathcal{P}(D)$, hence $\mathcal{P}_C(D)$ is not dense in $\mathcal{P}(D)$. Similarly $\mathcal{P}(D) \setminus \mathcal{P}_C(D)=\mathcal{P}_N(D)$ and  $cl(\mathcal{P}_N(D))=\mathcal{P}_N(D)\neq \mathcal{P}(D)$, hence $\mathcal{P}_C(D)$ is not nowhere dense in $\mathcal{P}(D)$.
\end{proof}

\begin{rmk}
It is obvious that in dynamical system $D$ there exist solutions $\gamma$ corresponding only to the one branch without branch switching, or solutions $\gamma$ "running out" the set $V$. The "force" causing the switch is exogenously determined. It is required the switch before solution $\gamma$ leaves the set $V$. And there has to be the reason depending on concrete modelled problem for this switch. It also depends on interpretation, see the application part of this paper - section \ref{aplication_in_macroeconomics}.
\end{rmk}

\section{Application in Macroeconomics}
\label{aplication_in_macroeconomics}

In this section we apply the theoretical findings from previous section in macroeconomics. We construct the new overall macroeconomic equilibrium model containing two branches - demand-oriented and supply-oriented and for connection of these two branches we use Euler equation branching. The switching between these two branches is interpreted by influence of the economic cycle. Then we describe economic behaviour of such overall model leading to the chaos and we submit reasonable economic interpretation of a cause of such behaviour. 

\subsection{Construction of New Overall Macroeconomic Equilibrium Model}

This new macroeconomic equilibrium model describes the macroeconomic situation in two sector economy, precisely the goods market equilibrium and the money market equilibrium simultaneously including every important economic phenomena like an economic cycle, an inflation effect, an endogenous money supply etc. in one overall model. This model follows from fundamental macroeconomic equilibrium model called IS-LM model. The original IS-LM model is strictly demand-oriented, assumes a constant price level and an exogenous money supply. This original conception is obsolete. Thus, we create new model eliminating these deficiencies and containing also a supply-oriented part. The demand-oriented model (the modified IS-LM model) holds in the recession and the supply-oriented model (the new QY-ML model) holds in the expansion. Then we join these two (sub-)models to one overall model called IS-LM/QY-ML by Euler equation branching interpreted by existence of economic cycle. So, we first present the modified IS-LM model eliminating mentioned deficiencies of the original model, then we construct the new QY-ML model supply oriented, then we explain when these models hold and at the end we join these two "sub-models" two one overall IS-LM/QY-ML model.

\subsubsection{Demand-Oriented Sub-model - Modified IS-LM Model}
\label{demad_oriented_sub-model}
    
We can find the original IS-LM model in e.g. \cite{gandolfo}. This model describes aggregate macroeconomic equilibrium, i.e. the goods market equilibrium and the money market (or  financial assets market) equilibrium simultaneously from the demand-oriented point of view. The demand-oriented model means that the supply is fully adapted to the demand. Here, we present our modification of the original model which eliminates its deficiencies or obsolete assumptions. This model is still demand-oriented but we eliminate the assumption of constant price level by modelling of inflation and the assumption of strictly exogenous money supply by connection of the endogenous and exogenous conception of the money supply.

\begin{df}
The \textit{modified IS-LM model} is given by the following system  
\begin{equation}
\label{dynamic_IS-LM}
\begin{array}{ll}
\textrm{IS:} & \frac{d Y}{d t} = \alpha_d [I(Y,R)-S(Y,R)] \\
\textrm{LM:} & \frac{d R}{d t} = \beta_d [L(Y,R- MP + \pi^e)-M(Y,R - MP + \pi^e)-M_{CB}],
\end{array}
\end{equation}
where \\
\begin{tabular}{lp{10cm}}
     $t$                   & is time, \\     
     $Y$                   & is aggregate income (GDP, GNP), \\
     $R$                   & is long-term real interest rate, \\
     $I(Y,R)$              & is investment function, \\
     $S(Y,R)$              & is saving function, \\
     $L(Y,R- MP + \pi^e)$  & is money demand function, \\
     $M(Y,R- MP + \pi^e)$  & is money supply function, \\
     $M_{CB} > 0$          & is money stock determined by central bank, \\
     $MP>0$                & is maturity premium, \\
     $\pi^e>0$             & is expected inflation rate, \\
     $\alpha_d, \beta_d>0$ & are parameters of dynamics.
\end{tabular}
\end{df}

\begin{rmk}
The goods market from the demand-oriented point of view is described by \textit{equation IS}. Analogously the money market from the demand-oriented point of view is described by the \textit{equation LM}. There are the investment and saving function on the goods market and the money demand function and money supply function on the money market. We suppose that all of these functions are differentiable.
\end{rmk}

\begin{rmk}
The main variables in the original IS-LM model is the aggregate income $Y$ and the interest rate $R$. In the modified model, we eliminate the original assumption of constant price level, so we need to distinguish two type of interest rate - a long-term real interest rate $R$ and a short-term nominal interest rate $i$. There is the long-term real interest rate on the goods market and the short-term nominal interest rate on the money market (or financial assets market). The well-known relation $i = R - MP + \pi^e$ holds. The sort-term nominal interest rate is positive and long-term real interest rate can be also negative because of an inflation rate. While $MP$ and $\pi^e$ are constants, $\frac{d i}{d t} = \frac{d (R - MP + \pi^e)}{d t} = \frac{d R}{d t}$ holds.
\end{rmk} 

\begin{rmk}
The one of the main criticised assumption of the original IS-LM model is an assumption of strictly exogenous money supply. This means that the money supply is some money stock determined by central bank. The endogenous money supply means that money is generated in economy by credit creation. Today's economists can not find any consensus between these two conceptions of the money supply. So, we join these two conceptions into one. We consider that the money supply is the endogenous quantity (some new defined function $M(Y,R-MP+\pi^e)$) with some exogenous part (constant $M_{CB}$).
\end{rmk}

\begin{rmk}
The \textit{curve IS} represents the goods market equilibrium. The curve IS is the set of ordered pairs fulfilling $I(Y,R)=S(Y,R)$. Similarly, the \textit{curve LM} represents the money market equilibrium. The \textit{curve LM} is the set of ordered pairs fulfilling $L(Y,R-MP+\pi^e)=M(Y,R-MP+\pi^e)$. So, the aggregate macroeconomic equilibrium (i.e. equilibrium on goods market and on money market, or on financial assets market, simultaneously) is the intersection point (one or more) of the curve IS and the curve LM. In the dynamic version we research the (un)stability of this macroeconomic system.
\end{rmk}

\begin{df}
\label{df:economic_properties_I_S_L_M}
\textit{Economic properties of the functions} $I(Y,R)$, $S(Y,R)$, $L(Y,R - MP + \pi^e)$ and $M(Y,R - MP + \pi^e)$ are the following:
\begin{equation}
\label{economic_I} 
0<\frac{\partial I}{\partial Y}<1, \frac{\partial I}{\partial R}<0,
\end{equation}
\begin{equation}
\label{economic_S}
0<\frac{\partial S}{\partial Y}<1, \frac{\partial S}{\partial R}>0,
\end{equation}
\begin{equation}
\label{economic_L}
\frac{\partial L}{\partial Y}>0, \frac{\partial L}{\partial R}<0. 
\end{equation}
\begin{equation}
\label{economic_M} 
0 < \frac{\partial M}{\partial Y} < \frac{\partial L}{\partial Y}, \frac{\partial M}{\partial R} > 0.
\end{equation}
\end{df}

\begin{rmk}
The economic properties of the investment, saving and money demand function are standard. The economic properties of the new defined money supply function means that the relation between supply of money and aggregate income and also interest rate is positive and that the rate of increase of money supply depending on aggregate income is smaller than the rate of increase of money demand depending on aggregate income because the banks are more cautious than another subjects. $\frac{\partial L(Y,i)}{\partial i}=\frac{\partial L(Y,R - MP + \pi^e)}{\partial R}$ and $\frac{\partial M(Y,i)}{\partial i}=\frac{\partial M(Y,R - MP + \pi^e)}{\partial R}$ hold, because we assume constant $MP$ and $\pi^e$.
\end{rmk}

\subsubsection{Supply-Oriented Sub-model - New QY-ML Model}
\label{supply_oriented_sub-model}

In this subsection we construct new model describing the aggregate macroeconomic equilibrium or (un)stability which is supply-oriented in opposite of the IS-LM model. The construction of this new model is similar as of the IS-LM model. We find the simultaneous equilibrium on the goods market and on the money market but under assumption that the demand is fully adapted to the supply. We also consider the floating price level and the endogenous money supply with some exogenous part in the same way as in the modified IS-LM model.

\begin{df}
The \textit{QY-ML model} is given by the following system  
\begin{equation}
\label{dynamic_QY-ML}
\begin{array}{ll}
\textrm{QY:} & \frac{d Y}{d t} = \alpha_s [Q(\mathcal{K}(Y,R),\mathcal{L}(Y,R),\mathcal{T}(Y,R))-Y] \\
\textrm{ML:} & \frac{d R}{d t} = \beta_s [M(Y,R - MP + \pi^e)+M_{CB}-L(Y,R- MP + \pi^e)],
\end{array}
\end{equation}
where \\
\begin{tabular}{lp{10cm}}
     $t$                                      & is time, \\     
     $Y$                                      & is aggregate income (GDP, GNP), \\
     $R$                                      & is long-term real interest rate, \\    
     $Q(\mathcal{K},\mathcal{L},\mathcal{T})$ & is production function, \\
     $\mathcal{K}(Y,R)$                       & is capital function, \\
     $\mathcal{L}(Y,R)$                       & is labour function, \\
     $\mathcal{T}(Y,R)$                       & is technical progress function, \\
     $M(Y,R- MP + \pi^e)$                     & is money supply function, \\
     $M_{CB} > 0$                             & is money stock determined by central bank, \\
     $L(Y,R- MP + \pi^e)$                     & is money demand function, \\
     $MP>0$                                   & is maturity premium, \\
     $\pi^e>0$                                & is expected inflation rate, \\
     $\alpha_s, \beta_s>0$                    & are parameters of dynamics. 
\end{tabular}
\end{df}

\begin{rmk}
Thus, how do we proceed during the construction? If the goods demand is fully adapted to the goods supply, than the aggregate production has to be covered by demand. So, the supply side is represented by some aggregate production function $Q$ and the demand side is represented by level of aggregate income $Y$ on the goods market. The production $Q$ is the function of capital $\mathcal{K}$, labour $\mathcal{L}$ and technical progress $\mathcal{T}$. We will consider that $\mathcal{K}$, $\mathcal{L}$ and $\mathcal{T}$ are dependent on aggregate income $Y$ and long-term real interest rate $R$. Summary, $Q=Q(\mathcal{K}(Y,R),\mathcal{L}(Y,R),\mathcal{T}(Y,R))$. So, the goods market from the supply-oriented point of view is described by \textit{equation QY}. The money demand is fully adapted to the money supply, so we have the \textit{equation ML} to describe the money market from the supply-oriented point of view. There are the production function on the goods market and the money supply and money demand function on the money market. We suppose that all of these functions are differentiable.
\end{rmk}

\begin{rmk}
The \textit{curve QY} represents the goods market equilibrium. The curve QY is the set of ordered pairs fulfilling $Q(\mathcal{K}(Y,R),\mathcal{L}(Y,R),\mathcal{T}(Y,R))=Y$. Similarly, the \textit{curve ML} represents the money market equilibrium. The curve ML is the set of ordered pairs fulfilling $M(Y,R-MP+\pi^e)=L(Y,R-MP+\pi^e)$. So, the aggregate macroeconomic equilibrium (i.e. the equilibrium on goods market and on money market, or on financial assets market, simultaneously) is the intersection point (one or more) of the curve QY and the curve ML. In the dynamic version we research the (un)stability of this macroeconomic system.
\end{rmk}

\begin{df}
\label{df:economic_properties_Q_K_L_T}
\textit{Economic properties of the aggregate production function} are given by
\begin{equation}
\label{economic_Q}
\frac{\partial Q}{\partial \mathcal{K}}>0, \frac{\partial Q}{\partial \mathcal{L}}>0, \frac{\partial Q}{\partial \mathcal{T}}>0.
\end{equation} 
\textit{Economic properties of the production factors functions} are the following
\begin{equation}
\label{economic_K_L_T_Y}
\frac{\partial \mathcal{K}}{\partial Y}>0, \frac{\partial \mathcal{L}}{\partial Y}>0, \frac{\partial \mathcal{T}}{\partial Y}>0,
\end{equation}
\begin{equation}
\label{economic_K_L_T_R}
\frac{\partial \mathcal{K}}{\partial R}<0, \frac{\partial \mathcal{L}}{\partial R}<0, \frac{\partial \mathcal{T}}{\partial R}<0.
\end{equation}
\end{df}

\begin{rmk}
The economic interpretation of these properties is following. If the production factors (i.e. capital, labour and technical progress) are increased then the aggregate production will increase. The relations between production factors ($\mathcal{K}$, $\mathcal{L}$, $\mathcal{T}$) and aggregate income ($Y$) are positive. And the relations between production factors ($\mathcal{K}$, $\mathcal{L}$, $\mathcal{T}$) and long-term real interest rate ($R$) are negative.
\end{rmk}

\subsubsection{Influence of Economic Cycle}
\label{influence_economic_cycle}

The previous "sub-models" hold in the different economic situations. This holding of demand-oriented and supply-oriented "sub-model" and switching between them depends on the phase of the economic cycle. In this section, we explain this context. On the Figure \ref{fig:economic_cycle}, there is illustrated the economic cycle. We can see its phases: expansion, peak, recession, trough and so on along the economic growth trend. 
\begin{figure}[ht]
  \centering
  \includegraphics[height=3.8cm]{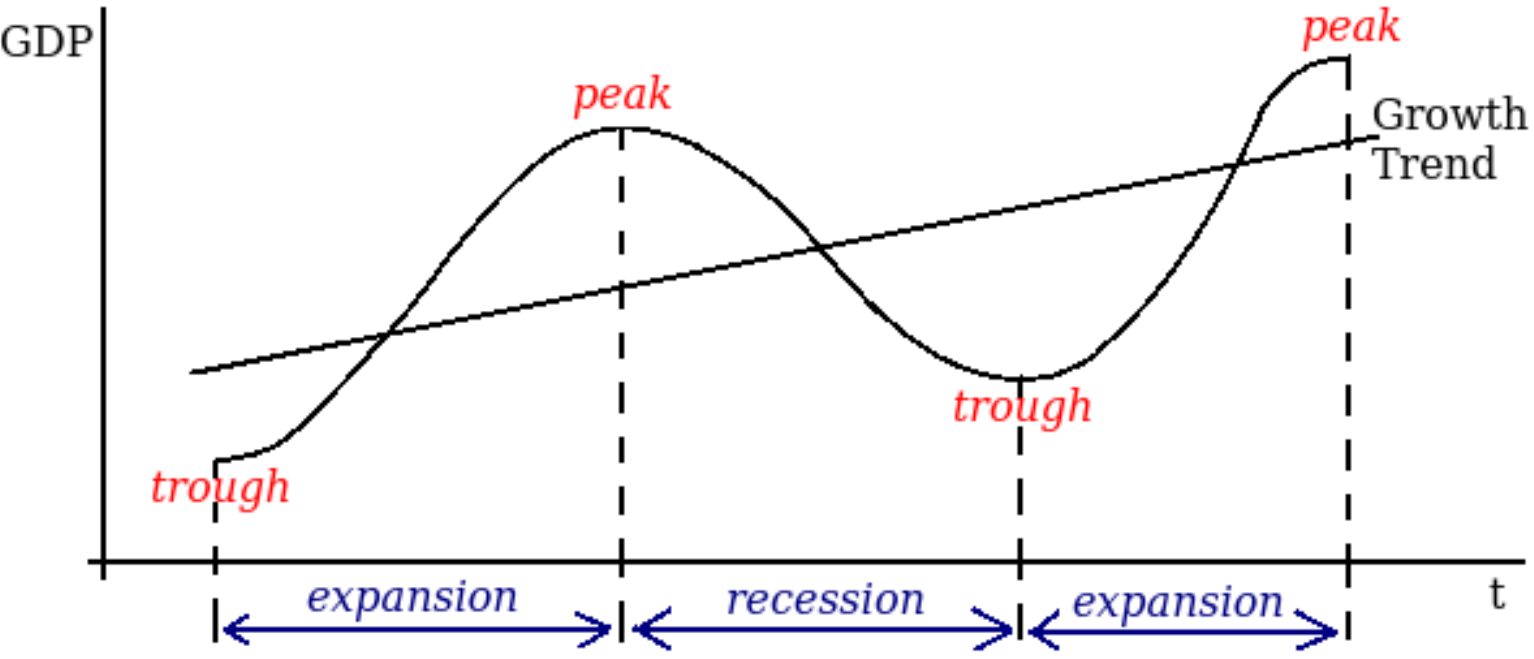}
  \caption{Economic cycle}
  \label{fig:economic_cycle}
\end{figure}

The demand-oriented "sub-model" - modified IS-LM model holds in the recession phase. In the recession the economy is under the production-possibility frontier. So, the firms can flexibly react on the demand, thus the supply is adapted to the demand. In the trough the economic situation is changed. The demand-oriented "sub-model" is switched to the supply-oriented "sub-model". Then, there is the expansion phase where supply-oriented model - new QY-ML model holds. In the expansion the production increases, the production factors are sufficiently rewarded and change theirs income to the aggregate demand. So, the demand is adapted to the supply (aggregate production). The new QY-ML model holds until the peak where the economic situation is changed. In the peak the supply-oriented "sub-model" is switched to the demand-oriented "sub-model". Further, the cycle continues in such a way. We summarize this mechanism in Table \ref{tab:switching_mechanism_according_to_economic_cycle}.
\begin{table}[h]
\begin{center}
\begin{tabular}{|l|l|l|}
  \hline
  Phase          & Orientation                   & Describing by \\
  \hline \hline
  Recession         & Demand-oriented model         & Modified IS-LM model \\
  \hline
  Trough            & Changing from demand-oriented & Switching from IS-LM model\\
                    & to supply-oriented model      & to QY-ML model \\
  \hline
  Expansion         & Supply-oriented model & New QY-ML model \\
  \hline
  Peak              & Changing from supply-oriented & Switching from QY-ML model \\
                    & to demand-oriented model      & to IS-LM model \\
  \hline
\end{tabular}
\vspace{0.1cm}
\caption{Switching mechanism according to economic cycle}
\label{tab:switching_mechanism_according_to_economic_cycle}
\end{center}
\end{table}

\subsubsection{Overall Macroeconomic IS-LM/QY-ML Model}

In this subsection, we formulate the overall macroeconomic IS-LM/QY-ML model. This model consists of demand-oriented "sub-model" - the modified IS-LM model (defined in subsection \ref{demad_oriented_sub-model}) and of supply-oriented "sub-model" - the new QY-ML model (defined in subsection \ref{supply_oriented_sub-model}). These two "sub-models" are connected by Euler equation branching (see Definition \ref{df:Euler_equation_branching}).
\begin{df}
The \textit{overall macroeconomic IS-LM/QY-ML} model is given by the following differential inclusion
{\Small
\begin{equation}
\label{model_IS-LM/QY-ML}
\left(
\begin{array}{c}
\dot Y \\
\dot R 
\end{array}
\right)
\in \left\{
\left(
\begin{array}{c}
\alpha_d [I(Y,R)-S(Y,R)] \\
\beta_d [L(Y,i)-M(Y,i)-M_{CB}
\end{array}
\right),
\left(
\begin{array}{c}
\alpha_s [Q(\mathcal{K}(Y,R),\mathcal{L}(Y,R),\mathcal{T}(Y,R))-Y] \\
\beta_s [M(Y,i)+M_{CB}-L(Y,i)]
\end{array}
\right)
\right\}
\end{equation}}
where $i=R - MP + \pi^e$, constant $M_{CB}>0$ and parameters of dynamics $\alpha_d>0$, $\alpha_s>0$, $\beta_d>0$, $\beta_s>0$.
\end{df}
The investment function $I(Y,R)$, saving function $S(Y,R)$, money demand function $L(Y,i)$, money supply function $M(Y,i)$ and aggregate production function $Q(\mathcal{K},\mathcal{L},\mathcal{T})$ have the previously introduced economic properties (\ref{economic_I}), (\ref{economic_S}), (\ref{economic_L}), (\ref{economic_M}) and (\ref{economic_Q}). The production factors functions $\mathcal{K}(Y,R),\mathcal{L}(Y,R),\mathcal{T}(Y,R)$ have the economic properties (\ref{economic_K_L_T_Y}) and (\ref{economic_K_L_T_R}) mentioned above.

The solutions of this differential inclusion are solutions of demand-oriented branch (IS-LM model), of supply-oriented branch (QY-ML model) and also switching between these two branches. These solutions follow from the economic cycle, see the previous subsection \nolinebreak \ref{influence_economic_cycle}.

This presented differential inclusion with two branches generates a continuous dynamical system describing overall macroeconomic situation in every phase of economic cycle.

\subsection{Chaos in IS-LM/QY-ML Model}

In this section, we describe a dynamical behaviour of the system (\ref{model_IS-LM/QY-ML}) with relevant economic interpretation. In economics, equilibrium points are important. The equilibrium point represents an ideal situation where the demand is equal to the supply. In our situation the equilibrium points on the goods market are points of whole curve IS or QY and the equilibrium points on the money market are points of whole curve LM=ML. So, the overall macroeconomic equilibrium (simultaneous equilibrium on goods and money market) is represented by intersection point(s) of the curve IS and LM, or QY and ML. But the economic situation relevant to the equilibrium is very rare. Thus, the description of disequilibrium points are more important. Such disequilibrium points are all other points in plane given by $[Y,R]$. There commonly exists an excess of goods demand, an excess of goods supply, or an excess of money demand and an excess of money supply. A description of dynamical behaviour of our IS-LM/QY-ML model using phase portraits show us behaviour of the moving point in the area of disequilibrium points more than in equilibrium points. 

We show the dynamical behaviour of the most typical economic case of IS-LM/QY-ML model. We consider the modified IS-LM model described by (\ref{dynamic_IS-LM}) with economic functions properties (\ref{economic_I}), (\ref{economic_S}), (\ref{economic_L}), (\ref{economic_M}). The typical IS curve is decreasing and the typical LM curve is increasing. Using Implicit Function Theorem we see that properties (\ref{economic_L}) and (\ref{economic_M}) ensure increasing LM curve. If we furthermore assume 
\begin{equation}
\label{economic_I_S}
\frac{\partial I}{\partial Y} < \frac{\partial S}{\partial Y}
\end{equation}  
in addition to the properties (\ref{economic_L}) and (\ref{economic_M}) then the curve IS is decreasing. We know that $Y \geq 0$ and $R \in \mathbb{R}$. Now, let $R_{IS}(Y)$ denote a function whose graph is the curve IS, and $R_{LM}(Y)$ denote the function whose graph is the curve LM. These functions exist because of Implicit Function Theorem. If we assume
\begin{equation}
\label{intersection_IS-LM}
\lim_{Y \rightarrow 0^+} R_{IS}(Y) > \lim_{Y \rightarrow 0^+} R_{LM}(Y),
\end{equation}
in addition to the conditions (\ref{economic_I}), (\ref{economic_S}), (\ref{economic_L}), (\ref{economic_M}) and (\ref{economic_I_S}) then there exists one intersection point of the curve IS and LM.

\begin{prp}
The singular point of the IS-LM model (\ref{dynamic_IS-LM}) with economic functions properties (\ref{economic_I}), (\ref{economic_S}), (\ref{economic_L}), (\ref{economic_M}), (\ref{economic_I_S}) and (\ref{intersection_IS-LM}) is stable node or focus.
\end{prp}
\begin{proof}
The eigenvalues of Jacobi's matrix $J$ of the system (\ref{dynamic_IS-LM}) in this singular point are \\
$\lambda_{1,2} = \frac{1}{2} \left[ \alpha_d (I_Y - S_Y) + \beta_d (L_R - M_R) \pm \sqrt{\left[ \alpha_d (I_Y - S_Y) + \beta_d (L_R - M_R)  \right]^2 - 4detJ} \right]$\\ where $det J = \alpha_d \beta_d \left[ (I_Y - S_Y)(L_R-M_R) - (I_R - S_R) (L_Y - M_Y) \right]$. The real part of eigenvalues $Re(\lambda_{1,2}) < 0$ in this point because of (\ref{economic_I_S}), $L_R - M_R<0$ according to (\ref{economic_L}) and (\ref{economic_M}) and $det J>0$ according to economic condition (\ref{economic_I}), (\ref{economic_S}), (\ref{economic_L}), (\ref{economic_M}) and (\ref{economic_I_S}). From this follows that this singular point is stable node if in addition $\left[ \alpha_d (I_Y - S_Y) + \beta_d (L_R - M_R)  \right]^2 > 4detJ$ or stable focus if in addition $\left[ \alpha_d (I_Y - S_Y) + \beta_d (L_R - M_R)  \right]^2 < 4detJ$.
\end{proof}

Now, we consider new QY-ML model described by (\ref{dynamic_QY-ML}) with economic functions properties (\ref{economic_L}), (\ref{economic_M}), (\ref{economic_Q}), (\ref{economic_K_L_T_Y}) and (\ref{economic_K_L_T_R}). Using Implicit Function Theorem we see that properties (\ref{economic_L}) and (\ref{economic_M}) ensure increasing ML curve. If we furthermore assume 
\begin{equation}
\label{economic_Q_Y}
\frac{\partial Q}{\partial Y} < 1
\end{equation}  
in addition to the properties (\ref{economic_Q}), (\ref{economic_K_L_T_Y}) and (\ref{economic_K_L_T_R}) like analogy to the condition $0<I_Y<1$ and $0<S_Y<1$ on the goods market then the curve QY is decreasing. Analogously let $R_{QY}(Y)$ denote a function whose graph is the curve QY, and $R_{ML}(Y)(=R_{LM}(Y))$ denote the function whose graph is the curve ML. These functions exist because of Implicit Function Theorem. If we assume
\begin{equation}
\label{intersection_QY-ML}
\lim_{Y \rightarrow 0^+} R_{QY}(Y) > \lim_{Y \rightarrow 0^+} R_{ML}(Y),
\end{equation}
in addition to the conditions (\ref{economic_L}), (\ref{economic_M}), (\ref{economic_Q}), (\ref{economic_K_L_T_Y}), (\ref{economic_K_L_T_R}) and (\ref{economic_Q_Y}) then there exists one intersection point of the curve QY and ML.

\begin{prp}
The singular point of the QY-ML model (\ref{dynamic_QY-ML}) with economic functions properties (\ref{economic_L}), (\ref{economic_M}), (\ref{economic_Q}), (\ref{economic_K_L_T_Y}), (\ref{economic_K_L_T_R}), (\ref{economic_Q_Y}) and (\ref{intersection_QY-ML}) is unstable saddle point.
\end{prp}
\begin{proof}
The eigenvalues of Jacobi's matrix $J$ of the system (\ref{dynamic_QY-ML}) in this singular point are \\
$\lambda_{1,2} = \frac{1}{2} \left[ \alpha_s (Q_Y - 1) + \beta_s (M_R - L_R) \pm \sqrt{\left[ \alpha_s (Q_Y - 1) + \beta_s (M_R - L_R)  \right]^2 - 4detJ} \right]$
where $det J = \alpha_s \beta_s \left[ (Q_Y - 1)(M_R-L_R) - Q_R (M_Y - L_Y) \right]$.
The conditions (\ref{economic_L}), (\ref{economic_M}), (\ref{economic_Q}), (\ref{economic_K_L_T_Y}), (\ref{economic_K_L_T_R}) and (\ref{economic_Q_Y}) imply negative determinant of Jacobi's matrix $J$ of the system (\ref{dynamic_QY-ML}) in this singular point. From this follows that this point is unstable saddle point.
\end{proof}

We can see curve IS and LM and stable focus as equilibrium point of IS-LM model displayed on Figure \ref{fig:curves_IS_LM_focus}. The curve QY and ML and unstable saddle as equilibrium point of QY-ML model are shown on Figure \ref{fig:curves_QY_ML_saddle}.
\begin{figure}[h]
\begin{minipage}[c]{225pt}
\hspace{0.2cm}
  \includegraphics[width=215pt]{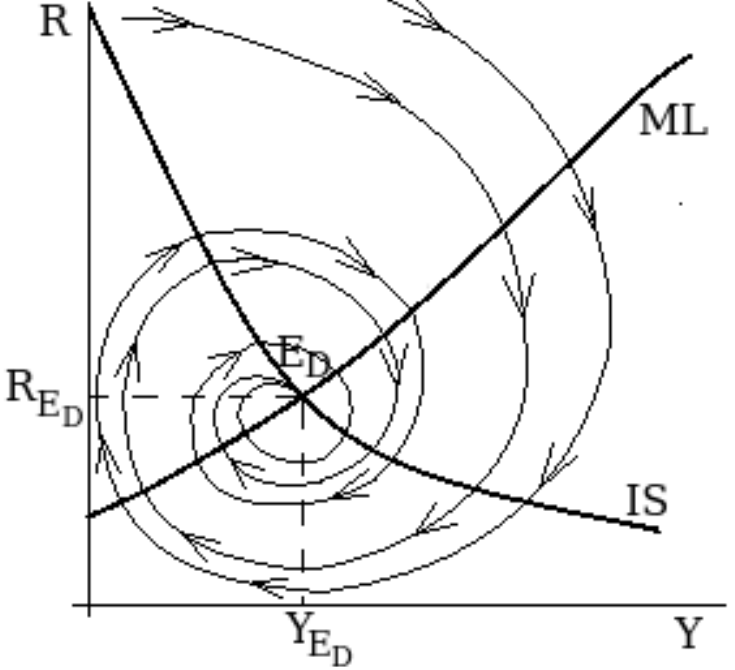}
\end{minipage}
\begin{minipage}[c]{225pt}
\hspace{0.2cm}
  \includegraphics[width=205pt]{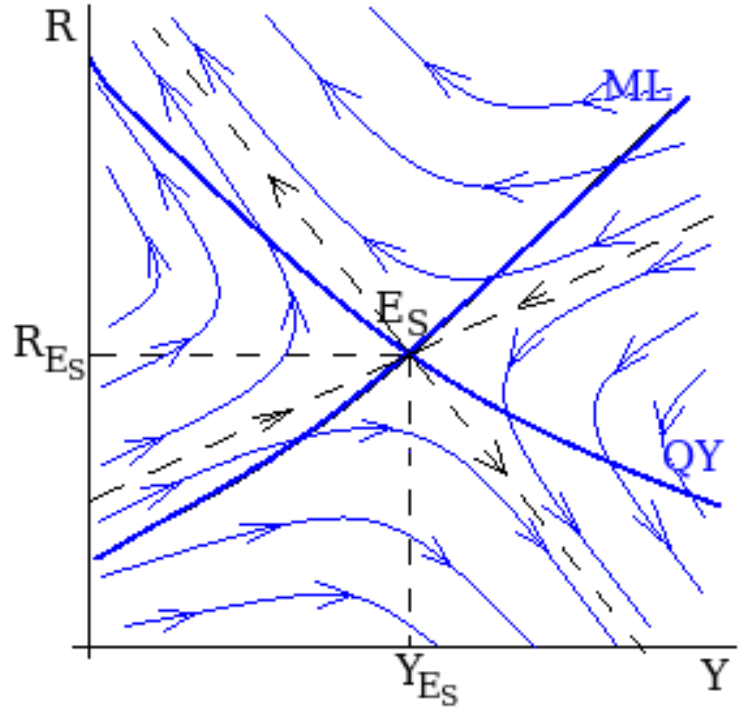}
\end{minipage} 
\begin{minipage}[c]{225pt}
  \caption{Typical case of IS-LM model}
  \label{fig:curves_IS_LM_focus}
\end{minipage}
\begin{minipage}[c]{225pt}
  \caption{Typical case of QY-ML model}
  \label{fig:curves_QY_ML_saddle}
\end{minipage}
\end{figure}

Now, if we consider the differential inclusion (\ref{model_IS-LM/QY-ML}) consisting of the IS-LM model (\ref{dynamic_IS-LM}) and of the QY-ML model (\ref{dynamic_QY-ML}) with economic functions properties (\ref{economic_I}), (\ref{economic_S}), (\ref{economic_L}), (\ref{economic_M}), (\ref{economic_Q}), (\ref{economic_K_L_T_Y}), (\ref{economic_K_L_T_R}), (\ref{economic_I_S}), (\ref{intersection_IS-LM}), (\ref{economic_Q_Y}) and (\ref{intersection_QY-ML}), then there chaotic sets in plane $\mathbb{R}^2$ (with variables $Y,R$) are admitted, see yellow areas on Figure \ref{fig:chaos_IS-LM_QY-ML}.

\begin{figure}[ht]
  \centering
  \includegraphics[height=7.5cm]{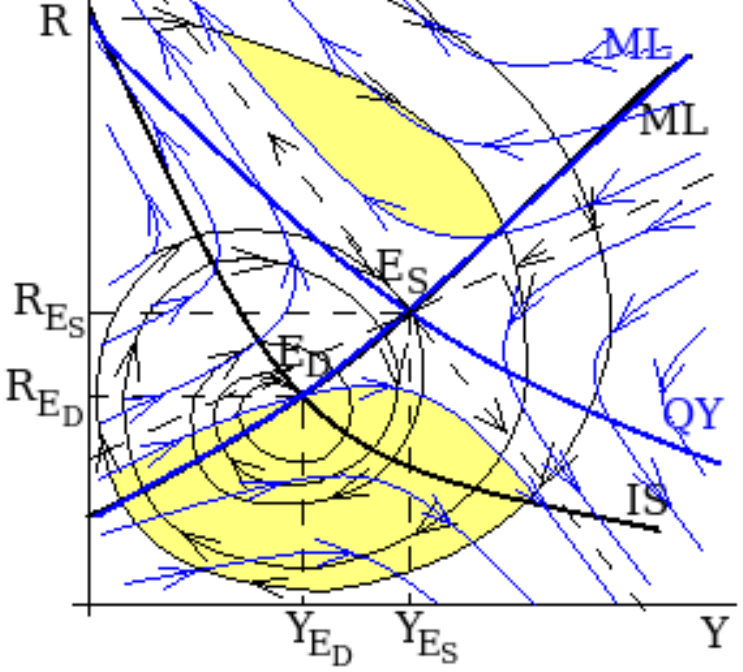}
  \caption{Chaotic sets in typical case of IS-LM/QY-ML model}
  \label{fig:chaos_IS-LM_QY-ML}
\end{figure}

As we present in the section \ref{chaotic_set_of_solutions} we need to have Devaney, Li-Yorke and distributional chaotic set of solutions for chaos existence in plane $\mathbb{R}^2$. We use the set of solutions $V^*=\{\gamma_0,\gamma_1\}$ from the proof of Theorem \ref{thm:chaotic_set_of_solutions}. It is necessary to provide some economic interpretation or economic explanation of such set of solutions $V^*$. 
\begin{itemize}
\item The solution $\gamma_0$ describes an economy in the regular economic cycle where the expansion phase has still the same duration and also the recession phase. Thus, this economic cycle has still the same period. In plane $\mathbb{R}^2$ describing by ordered pairs $[Y,R]$ the concrete changes of aggregate income $Y$ and interest rate $R$ during the expansion and recession phase are the same in every period.
\item The solutions $\gamma_i \in \gamma_1$ describes more irregular economic cycle. The period of whole (one) cycle is different in two consecutive periods. Every solution $\gamma_i$ explains every relevant change in the following period duration.    
\end{itemize}

For node or focus types of equilibrium points of the IS-LM and QY-ML model the trajectory corresponding to the IS-LM model passing through the equilibrium of the QY-ML model and the trajectory corresponding to the QY-ML model passing through the equilibrium of the IS-LM model is critical and very "strong". It is critical trajectory because we have to overcome the frontier given by this trajectory to get out of the chaotic area. It is "strong" trajectory because the market economy "inclines" to the equilibrium points, which are ideal states, and this trajectory ensures the path containing this equilibrium, and the economy "is essentially attracted to use" this path. For saddle type of equilibrium point the critical and "strong" trajectory of the opposite model is not passing through this saddle but the last (the
"farthest" from the saddle point) trajectory intersecting first the unstable manifold and then the stable manifold. The intersection of the stable and unstable manifold is exactly the saddle point and if we start (after switching) on the stable or unstable manifold and go along this model (with saddle type of equilibrium) we are on the path containing this equilibrium. In fact in typical case of IS-LM and QY-ML model presented above the one chaotic area is whole "quadrant" corresponding to the QY-ML model bounded only by the stable and unstable manifold (without this boundary) and containing the trajectories of IS-LM model intersecting first the unstable and then the stable manifold, see the Figure \ref{fig:chaos_IS-LM_QY-ML}. Here, the "critical" trajectory does not exist because the singular point type of the IS-LM model is focus. To get out of the chaotic area we must overstep exactly this stable or unstable manifold of QY-ML model. 

So, if the starting point of the economy (combination of concrete levels of aggregate income $Y$ and long-term real interest rate $R$) lies in the chaotic area (chaotic set $V$, yellow area on Figure \ref{fig:chaos_IS-LM_QY-ML}) and the economic cycle principle described above (Devaney, Li-Yorke, distributional chaotic set of solutions $V^*$, construction in the proof of Theorem \ref{thm:chaotic_set_of_solutions}) works, then there will exist Devaney, Li-Yorke and distributional chaos in economy. If the starting point of the economy is equilibrium point of the IS-LM model and  first the IS-LM model works, then for the recession phase the economy stays in this equilibrium but in the expansion phase goes along the QY-ML model. Then in such case it depends on concrete situation and concrete economic cycle and its periods whether chaos exists. And vice versa for starting equilibrium point of QY-ML model.

\begin{rmk}
The chaotic behaviour of the IS-LM and also QY-ML model depends on economic functions properties. We specified some standard economic properties in Definition \nolinebreak \ref{df:economic_properties_I_S_L_M} and \ref{df:economic_properties_Q_K_L_T}. Furthermore, we assumed $I_Y<S_Y$ in (\ref{economic_I_S}) for the IS-LM model. But there can be also another properties like for example conditions described by Kaldor in \cite{kaldor} which define three parts of investment and saving function course. In the first and third part $I_Y<S_Y$ and in the second part in the middle $I_Y>S_Y$. Such IS-LM model exists, see \cite{volna-kalicinska}. Similarly, we furthermore assumed $Q_Y<1$ in (\ref{economic_Q_Y}) for the QY-ML model. But also opposite condition $Q_Y>1$ can exist. Every such conditions lead to combination of the "classical" hyperbolic singular points lying in the different point in the IS-LM and QY-ML model. From this follows that there always exists Devaney, Li-Yorke and distributional chaos in macroeconomic situation described by IS-LM model in the recession and by QY-ML model in the expansion with previously presented principle of the economic cycle. 
\end{rmk}

\begin{rmk} 
In economy there can also arise some non-standard situations or behaviours of economic subjects. Then there can be different economic function properties than standard, e.g. the opposite properties of modelled functions. This can lead to the "classical" hyperbolic singular points in both models, also to some non-hyperbolic singular points or centres. The possibility of chaos existence in such cases with these different singular points seems to be similar - chaos is typically admitted.
\end{rmk}

\section*{Conclusion}
The first theoretical part of this paper examines the existence of chaos in plane $\mathbb{R}^2$ given by continuous dynamical system generated by Euler equation branching. There are few theorems, lemmas and remarks illustrating this existence of chaos. The second part is application in macroeconomics. Using newly constructed overall macroeconomic equilibrium model called IS-LM/QY-ML we show chaotic behaviour of the economy with relevant economic interpretation of chaos causes. 

In perspective of these days and of some economic situation not explicable by "classical" macroeconomic models the chaos seems to be a natural part of the economy behaviour. As we can see in this paper there exist the areas given by levels of the aggregate income (GDP) and of the long-term real interest rate where economy described by IS-LM/QY-ML model behaves chaotically under economic cycle principle described in this paper.

\section*{Discussion}
We researched the combinations of hyperbolic singular points of two branches of Euler equation branching not lying in the same point in plane $\mathbb{R}^2$ with non-zero determinant of Jacobi's matrix. Even though the situation in other cases (including also cycles/centres or singular points of both branches lying in the same point in plane $\mathbb{R}^2$) seems to be similar, the research dealing with this can be also interesting for the completion of the overview of all possible singular points combinations in plane $\mathbb{R}^2$.

Another interesting question lies in differences between economic interpretation of Devaney, Li-Yorke or distributional chaos. It can be also meaningful to give explanation whether there are some differences from the economic point of view or whether not.

\section*{Acknowledgements}
The research was supported by the Student Grant Competition of Silesian University in Opava, grant no. SGS/2/2013.


\begin{thebibliography}{99}
\bibitem{badarudin_ariff_khalid} BADARUDIN, Z. E.; ARIFF, M.; KHALID, A. M.: \textit{Exogenous or endogenous money supply: evidence from Australia}, The Singapore Economic Review, 2012, Volume 57, No. 4, 1250025 (12 pages).
\bibitem{baily_friedman} BAILY, M.; FRIEDMAN, P.: \textit{Macroeconomics, financial markets, and the international sector}, Irwin, 1991.
\bibitem{cesare_sportelli} CESARE L. DE; SPORTELLI M.: \textit{A Dynamic IS-LM Model with Delayed Taxation Revenues}, Chaos, Solitons \& Fractals, 2005, Volume 25, pp. 233-244.
\bibitem{chiba_leong} CHIBA, S.; LEONG, K.: \textit{Can the IS-LM Model Truly Explain Macroeconomic Phenomena?}, Journal of Young Investigators, 2007, Volume 17, Issue 3, September, http://www.jyi.org/research/re.php?id=1242.
\bibitem{gandolfo} GANDOLFO, G.: \textit{Economic Dynamics}, 3rd ed., Springer-Verlag, Berlin-Heidelberg, 1997.
\bibitem{hicks} HICKS, J. R.: \textit{Mr. Keynes and the Classics - A Suggested Interpretation}, Econometrica, 1937, volume 5, April, pp. 147-159.
\bibitem{kaldor} KALDOR, N.: \textit{A Model of the Trade Cycle}, Economic journal, 1940, Volume 50, March, pp. 78-92.
\bibitem{king} KING, R. G.: \textit{The New IS-LM Model: Language, Logic, and Limits}, Economic Quarterly, Federal Reserve Bank of Richmond, 2000, Volume  86/3, Summer, pp. 45-103.
\bibitem{puu} PUU, T.: \textit{Attractors, Bifurcations} \& \textit{Chaos - Nonlinear Phenomena in Economics}, 2nd ed., Springer-Verlag, Berlin-Heidelberg, 2003.
\bibitem{neri_venturi} NERI, U.; VENTURI, B.: \textit{Stability and bifurcations in IS-LM economic models}, International Review of Economics, 2007, Volume 54, Issue 1, pp. 53-65. 
\bibitem{romer} ROMER, D.: \textit{Keynesian Macroeconomics without the LM Curve}, National bureau of economic research, January 2000, working paper series.
\bibitem{sedlacek} SEDL\'{A}\v{C}EK, P.: \textit{Post-Keynesian theory of money - Alternative view}, Politick\'{a} ekonomie, 2002, Volume 50, Issue 2, pp. 281-292.  
\bibitem{smirnov} SMIRNOV, G. V.: \textit{Introduction to the Theory of Differential Inclusions}, vol. 41 of Graduate Studies in Mathematics, volume 41, American Mathematical Society, Providence, Rhode Island, 2002.
\bibitem{sojka} SOJKA, M.: \textit{Monetární politika evropské centrální banky a její teoretická východiska pohledem postkeynesovské ekonomie (Monetary Policy of the European Central Bank and Its Theoretical Resources in the View of Postkeynesian Economy)}, Politická ekonomie, 2010, Issue 1, pp. 3-19.
\bibitem{stockman_raines} STOCKMAN, D. R.; RAINES, B. R.: \textit{Chaotic sets and Euler equation branching}, Journal of Mathematical Economics, 2010, Volume 46, pp. 1173-1193.
\bibitem{volna-kalicinska} VOLN\'{A} KALI\v{C}INSK\'{A}, B.: \textit{Augmented IS–LM model based on particular functions}, Applied Mathematics and Computation, 2012, Volume 219, Issue 3, pp. 1244-1262.
\bibitem{wiggins} WIGGINS, S.: \textit{Introduction to Applied Nonlinear Dynamical Systems and Chaos}, 2nd ed., Springer-Verlag, New York, 2003.
\bibitem{zhang} ZHANG, W-B.: \textit{Differential Equations, Bifurcations, and Chaos in Economics}, World Scientific Publishing Co. Pte. Ltd., Singapore, 2005. 
\bibitem{zhou-li} ZHOU, L.; LI, Y.: \textit{A Dynamic IS-LM Business Cycle Model with Two Time Delay in Capital Accumulation Equation}, Journal of Computational and Applied Mathematics, 2009, Volume 228, Issue 1, June, pp. 182-187.
\end{thebibliography}
\end{document}